\documentclass[12pt]{article}

\usepackage{graphicx}
\usepackage{amsmath, amssymb}

\usepackage{graphicx}
\usepackage[toc,page]{appendix}
\usepackage{rotating}
\usepackage{xcolor}
\usepackage{soul}
\usepackage{multirow}
\usepackage{mathtools}
\usepackage{enumitem}
\usepackage{algorithm}
\usepackage{algpseudocode}
\usepackage{xpatch}
\makeatletter
\renewcommand{\ALG@beginalgorithmic}{\normalsize}
\xpatchcmd{\algorithmic}{\itemsep\z@}{\itemsep=2ex plus2pt}{}{}
\makeatother

\usepackage{tikz, tikz-3dplot}
\usepackage[numbers]{natbib}
\usepackage{lipsum}
\usepackage{todonotes}
\usepackage{comment}
\usepackage{adjustbox}
\usepackage{multirow}
\usepackage{nicematrix}
\usepackage{enumitem}
\newlist{casesp}{enumerate}{3} %% new list environment based on enumerate with a max depth of 3
\setlist[casesp]{align=left, %% alignment of labels
                 listparindent=\parindent, %% same indentation as in normal text
                 parsep=\parskip, %% same parskip as in normal text
                 font=\normalfont\bfseries, %% font used for labels
                 leftmargin=0pt, %% total amount by which text is indented
                 labelwidth=0pt, %% width of labels (=how much they stick out on the left because align=left)
                 itemindent=.4em,labelsep=.4em, %% space between label and text
                 partopsep=0pt, %% extra vertical space above and below if separate paragraph
                 }
\setlist[casesp,1]{label=Case~\arabic*:,ref=\arabic*}
\setlist[casesp,2]{label=Case~\thecasespi.\roman*:,ref=\thecasespi.\roman*}
\setlist[casesp,3]{label=Case~\thecasespii.\alph*:,ref=\thecasespii.\alph*}

\usepackage{etoolbox}
\newtoggle{createappendix}
\settoggle{createappendix}{false}
\newtoggle{allproofsinpaper}
\settoggle{allproofsinpaper}{false}
\newtoggle{ancilliarydiscussions}
\settoggle{ancilliarydiscussions}{true}
\newtoggle{createarxiv}
\makeatletter
\@ifclassloaded{article}{\settoggle{createarxiv}{true}}{\settoggle{createarxiv}{false}}
\makeatother

\iftoggle{createarxiv}{
\usepackage{arxiv}
\usepackage{hyperref}
\usepackage{amsthm}
\usepackage[numbers]{natbib}
\newtheorem{lemma}{Lemma}
\newtheorem{proposition}{Proposition}
\newtheorem{remark}{Remark}
\newtheorem{definition}{Definition}
\newtheorem{assumption}{Assumption}
\newtheorem{theorem}{Theorem}
\newtheorem{corollary}{Corollary}
\newtheorem{example}{Example}
}{\newtheorem{assumption}{Assumption}}

\usepackage{accents}
\usepackage{subcaption}

\renewcommand{\Re}{\mathbb{R}}

\DeclareMathOperator{\conv}{conv}
\DeclareMathOperator{\extpt}{vert}
\DeclareMathOperator{\extray}{extray}
\DeclareMathOperator{\DBP}{DBP}
\DeclareMathOperator{\D}{D}
\DeclareMathOperator{\DD}{DD}
\DeclareMathOperator{\DDR}{DDR}
\DeclareMathOperator{\opt}{opt}

\DeclareMathOperator{\sign}{sign}
\DeclareMathOperator{\AC}{AC}
\DeclareMathOperator{\ACD}{ACD}

\DeclareMathOperator{\ACRLT}{ACRLT}
\DeclareMathOperator{\DDP}{DDP}
\DeclareMathOperator{\CM}{CM}
\DeclareMathOperator{\closure}{cl}
\DeclareMathOperator{\CE}{CE}
\DeclareMathOperator{\Trace}{Trace}
\DeclareMathOperator{\FD}{FD}
\DeclareMathOperator{\ri}{ri}
\DeclareMathOperator{\Id}{Id}
\DeclareMathOperator{\DE}{DE}
\DeclareMathOperator{\FDR}{FDR}

\newcommand{\K}{{\bar{K}}}

\newcommand{\ftilde}{{\tilde{f}}}

\newcommand{\cl}{\mathord{:}}
\newcommand{\Z}{\mathbb{Z}}
\newcommand{\lin}{\mathop{\text{lin}}\nolimits}
\newcommand{\A}{{\bar{A}}}
\newcommand{\starred}[1]{\accentset{\star}{#1}}

\newcommand{\V}{\mathcal{V}}

\newcommand{\x}{\boldsymbol{x}}

\newcommand{\myratio}[2]{\genfrac{}{}{0pt}{}{#1}{#2}}

\begin{document}

\iftoggle{createarxiv}{}{%
\journalname{Math. Program., Ser. B}

\title{New finite relaxation hierarchies for concavo-convex, disjoint bilinear programs, and facial disjunctions}
%%\subtitle{Do you have a subtitle?\\ If so, write it here}
\titlerunning{Finite Relaxation Hierarchy for Disjoint Bilinear and Facial Disjunctions}
%%\titlerunning{Short form of title}        % if too long for running head
\author{Mohit Tawarmalani}
        
\institute{Mohit Tawarmalani\\
\email{mtawarma@purdue.edu}\\
Mitch Daniels School of Business,\\ Purdue University, 100 S. Grant Street, West Lafayette, IN 47907-2076, USA}

%%\authorrunning{Short form of author list} % if too long for running head

\date{Received: date / Accepted: date}
%% The correct dates will be entered by the editor
}

\iftoggle{createarxiv}{
\title{New finite relaxation hierarchies for concavo-convex, disjoint bilinear programs, and facial disjunctions}
\author{ \href{https://orcid.org/0000-0003-3085-0084}{\includegraphics[scale=0.06]{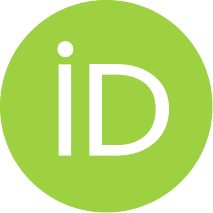}\hspace{1mm}Mohit Tawarmalani} \\
	Mitch Daniels School of Business\\
	Purdue University\\
	West Lafayette, IN 47907\\
	\texttt{mtawarma@purdue.edu}
}
\renewcommand{\shorttitle}{Finite Relaxation Hierarchy for Disjoint Bilinear and Facial Disjunctions}
\date{\today}}{}

\maketitle

\begin{abstract}
This paper introduces novel relaxation hierarchies for concavo-convex programs (CXP), a class of problems that includes disjoint bilinear programming (DBP) and concave minimization (CM) as special cases. At the core of these hierarchies is an algorithm  based on double-description (DD) that computes the barycentric coordinates of a polyhedral cone as rational, non-negative functions representing multipliers associated with the cone's rays. These hierarchies combine geometric structure derived from barycentric coordinates with algebraic techniques via rational functions, achieving the convex hull in $m$ iterations, where $m$ is the number of inequalities that a subset of the variables must satisfy. Our framework offers the first unified approach to analyze and tighten relaxations from disjunctive programming (DP) and reformulation-linearization technique (RLT) for CXP. We also demonstrate that our methods extend to facial disjunctive programs (FDP), where solutions are constrained to lie on faces of a Cartesian product of polytopes, generalizing known hierarchies for 0-1 programs.
    \iftoggle{createarxiv}{}{\subclass{65K05 \and 90C26 \and 90C11 \and 90C57}}
\end{abstract}

\section{Introduction}
This paper focuses primarily on concavo-convex programs (CXP):
\begin{equation}
  \nonumber
  (\D) \quad \opt_{\D} = \min \left\{ f(x,y)\;\middle|\; y\in C, x\in P\right\},
  \label{eq:defineD}
\end{equation}
where $C$ is a compact convex set in $\Re^{n_y}$, $P$ is the polytope $\{x\in \Re^n\mid Ax\le b\}$, $A\in \Re^{m\times n}$, $b\in \Re^m$, and $f:\Re^{n\times n_y}\rightarrow \Re$ is a continuous concavo-convex function, \textit{i.e.},\/ $f$ is concave in $x$ for a fixed $y$ and convex in $y$ for a fixed $x$. We also assume that, for any $x$, $f(x,\cdot)$ is proper (not identically $+\infty$ and never $-\infty$). Since $f$ is continuous, and $P$, $C$ are compact, the minimum of $(D)$ is achieved.
By viewing the objective as a minimum of concave functions, one for each $y\in C$, the problem reduces to concave minimization in $x$. Consequently, there exists an optimal solution of $(D)$ where $x$ is a vertex of $P$. This follows by recognizing that minimum of concave functions is concave and thus the minimizer over $P$ can be chosen to be a vertex (\textit{cf}, Theorem 1.19 in \cite{horst2000introduction}).

We also briefly consider convex relaxations for facial disjunctive sets:
\begin{equation}
  \nonumber
      \FD = \left\{(x,y)\;\middle|\; Bx + Dy \le_\K c, x_i\in \bigcup_{j=1}^{j_i} F_{ij}, i=1,\ldots,n_p\right\},
  \label{eq:defineFD}
\end{equation}
where $x_i\in \Re^{n_i}$, $F_{i1},\ldots,F_{ij_i}$ are disjoint faces of a polytope $P_i$,
% such that each vertex of $P_i$ belongs to some $F_{ij}$, 
$\K$ is a convex cone in $\Re^{m_k}$, $B\in \Re^{m_k\times n}$, $D\in \Re^{m_k\times n_y}$, and $c\in \Re^{m_k}$. The facial disjunctive program (FDP) optimizes a linear function over $\FD$.

These problems have widespread applications in mathematical optimization. For example, $(\D)$ includes, as a special case, the following affine-convex program:
\begin{equation}
  \nonumber
    (\AC) \quad \opt_{\AC} = \min\left\{g_0(y) + \sum_{i=1}^n x_ig_i(y) \;\middle|\;  y\in C, x\in P\right\},
  \label{eq:defineAC}
 \end{equation}
where each $g_i:\Re^{n_y}\rightarrow \Re$ is a convex function and, if $g_i(\cdot)$ is not affine, then $P\subseteq \{x\mid x_i\ge 0\}$. A subclass of $(\AC)$ includes disjoint bilinear programs:
 \begin{equation}
  \nonumber
   (\DBP)\quad \min\bigl\{x^\intercal Qy\mid y\in P^y, x\in P\bigr\}
   \label{eq:defineDBP}
 \end{equation}
where $Q\in \Re^n\times\Re^{n_y}$, and $P$ as well as $P^y$ are polytopes. We write $P^y = \{y\mid A^yy\le b^y\}$. To see $(\DBP)$ as a special case of $(\AC)$, set $C=P^y$, $g_0(x) = 0$, and $g_i(y) = Q_{i,\cl{}} y$, where $Q_{i,\cl{}}$ is the $i^{\text{th}}$ row of $Q$. Disjoint bilinear programs (DBP) have various applications in economics and engineering, including location-allocation \cite{adams1993mixed, lara2016global}, robust network design \cite{chang2017robust}, 
%bin packing \cite{caprara2009bidimensional, cote2021combinatorial, wang2021chance}, multi-dimensional assignment \cite{custic2017bilinear, frieze1974bilinear}, 
constrained bimatrix games \cite{konno1971bilinear}, and production planning \cite{as2011bilinear, bruni2018computational, chu2013integration}. More generally, $(\AC)$ has applications in optimal control of discrete-time hybrid dynamical systems~\cite{marcucci2024shortest}. The following concave minimization problem is also a subclass of $(\D)$:
 \begin{equation}
  \nonumber
  (\CM)\quad \min \bigl\{c(x) \mid x\in P\bigr\},
  \label{eq:defineCM}
 \end{equation}
where $c:\Re^n\rightarrow \Re$ is a concave function. To see this, let $C=\{0\}$ and define $c(x) = f(x,0)$. Finally, FDP includes 0-1 linear optimization problems, as seen by setting $n_p = n$, $P_i = [0,1]$, $K = \Re^{m_k}_+$, and $\left(F_{i1},F_{i2}\right) = \left(\{0\},\{1\}\right)$.
 
Relaxation techniques for each of these problem classes have been studied extensively. Two seminal relaxation schemes related to this paper are disjunctive programming (DP) \cite{balas2018disjunctive} and the reformulation linearization technique \cite{sherali2013reformulation}. Both were initially developed to solve 0-1 programming problems \cite{balas1985disjunctive,sherali1990hierarchy} where they are able to compute the convex hull in finite steps. These approaches have been found to be closely related \cite{balas1993lift} and have since been generalized to apply to a wide range of problem classes. While DP and RLT share similarities, they differ fundamentally in their underlying principles: DP leverages geometric structure, whereas RLT is based on algebraic manipulation. This distinction is particularly evident in their application to mixed-integer nonlinear programming. DP-based methods focus on exploiting geometric properties~\cite{stubbs1999branch,ceria1999convex,ruiz2012hierarchy}, while RLT extensions rely on the inherent algebraic structure of the problem~\cite{Sherali1992,lasserre2002semidefinite}. Although both approaches effectively capture the geometric structure of 0-1 programs, their distinct applications to continuous optimization problems remain largely unexplored, and the relationship between DP and RLT-based methods in this context is still unclear.
 
Consider the case of disjoint bilinear programs (DBP). The optimal solution to DBP can be found by restricting attention to a finite set of solutions consisting of the vertices of $P$ and $P^y$. Disjunctive programming (DP) has been used in this context to develop cutting planes \cite{sherali1980finitely}. On the other hand, RLT has been found to provide the most numerically performant relaxation \cite{Sherali1992,chang2017robust,zhen2022disjoint,marcucci2024shortest}. Recently, new insights have emerged into relaxations for DBP based on adjustable robust optimization (ARO) \cite{chang2017robust,xu2018copositive,zhen2018adjustable,zhen2022disjoint} and by exploiting semidefinite programming \cite{kellner2016semidefinite}. DP, ARO \cite{zhen2018adjustable}, and the semidefinite hierarchy \cite{kellner2016semidefinite} terminate finitely with the convex hull. Unlike DP, RLT does not restrict attention to faces of $P$ or $P^y$, and it is not known to terminate with the convex hull in finite steps. Finiteness in ARO is achieved by removing variables \cite{zhen2018adjustable}, while relationships to specific forms of RLT have been derived by restricting to special types of policies \cite{chang2017robust,ardestani2021linearized,chandra2022probability}. ARO-based relaxations improve slowly \cite{zhen2022disjoint}, and while RLT-based relaxations are numerically stronger at lower levels, they only converge asymptotically \cite{handelman1988representing}. The semidefinite hierarchy requires solving large-scale semidefinite programs, but finite termination is guaranteed under certain technical conditions \cite{kellner2016semidefinite}.

There is also a significant body of literature on solving (CM) \cite{benson1995concave}. Concave minimization is amenable to disjunctive arguments \cite{tuy1964concave}. For the special case of minimizing a finite number of linear functions, \cite{zhen2022disjoint} reduces the problem to a disjoint bilinear program. However, no generally applicable RLT-based schemes are available for relaxing (CM). Despite the extensive research on relaxations for DBP and CM, a key gap in the literature is the lack of a unified theoretical framework that simultaneously captures the geometric and algebraic structures of these problems while ensuring finite termination. This motivates the development of a new finite hierarchy that integrates the strengths of multiple approaches, offers a cohesive relaxation strategy, and reveals structural relationships between existing hierarchies. 

In this paper, we propose new finite hierarchies of relaxations for $(\D)$ and $\FD$, addressing the challenges highlighted in the discussion of DBP and CM. These hierarchies are obtained by constructing outer-approximations of $P$ using the double-description (DD) algorithm \cite{motzkin1953double,Fukuda1995,padberg2013linear}. In related work, insights from the DD algorithm and relaxations developed here have been leveraged recently to design a finite simplicial branch-and-bound algorithm for $(\D)$ with encouraging computational performance \cite{oh2022algorithms, oh2024branchandbound}. For (CM), our relaxations encode the steps of the SUM algorithm proposed in \cite{falk1976successive} while deriving additional strength from the algebraic representation of barycentric coordinates. These coordinates, derived using the DD algorithm, are non-negative rational functions with broad applications in computer-aided design \cite{warren1996barycentric} and convexification \cite{tawarmalani2010inclusion}. Prior work has shown that polynomial functions do not suffice, and rational functions are required to express barycentric coordinates \cite{wachspress2016rational}. Our construction computes these rational functions for general polyhedra, tackling an existing challenge in computational geometry. Using these functions, we propose a new hierarchy that inherits the strengths of both DP and RLT-based hierarchies. The barycentric coordinates $\lambda_v(x)$ satisfy the linear precision property, expressing $x$ as $\sum_{v \in \extpt(P)} \lambda_v(x) v$, where $\extpt(P)$ is the set of vertices of $P$. These equalities are central to DP and will also play a key role in our hierarchy. Our approach tightens these equalities by introducing a new affine relationship between successive barycentric coordinate iterates. Additionally, we leverage fractional optimization techniques \cite{he2023convexification} to generalize RLT inequalities, inheriting the strengths of both DP and RLT. For $\FD$, we construct a relaxation hierarchy using barycentric coordinates, their multiplicative property over the Cartesian product of polytopes, and disjunctive programming, ensuring recovery of the convex hull. In the special case of 0-1 programs, this hierarchy reduces to the corresponding RLT relaxation.

\paragraph{Notation:} We write $(X; Y)$ (resp. $(X,Y)$) to stack matrices $X$ and $Y$ along the rows (resp. columns), where the number of columns (resp. rows) of $X$ and $Y$ match. Vectors are assumed to be column vectors. We will denote the $i^{\text{th}}$ principal vector in $\Re^q$ as $e_{q;i}$ while $e_{q;:}$ will denote the vector of all ones. We will drop the subscript $q$ when it is apparent from the context. Submatrices of a matrix $X$ will be denoted as $X_{I,J}$ where $I$ and $J$ are appropriately defined subsets. We denote the list of indices $[i,i+1,\ldots,j]$ as $[i\cl{}j]$ and drop $i$ or $j$ if they are the first or last indices. Similarly, we use $\{i\cl{}j\}$ to denote the set of indices $\{i,i+1,\ldots,j\}$. For a transposed matrix, the transpose is taken before indices are selected. The identity matrix (resp. zero matrix) of dimension $n\times n$ ($m\times n$) will be represented as $\Id_n$ (resp. $0_{m\times n}$) and the standard $n$-dimensional simplex $\{\gamma \in \Re^n_+\mid \sum_{i=1}
^n \gamma_i = 1\}$ will be represented as $\Lambda_n$. The optimal value of an optimization problem $(D)$ will be denoted as $\opt_D$, the convex hull of a set $S$ as $\conv(S)$. We will denote the extreme points of a polytope $P$ as $\extpt(P)$, its extreme rays as $\extray(P)$, and its relative interior as $\ri(P)$. 
In this paper, we repeatedly use several sets and acronyms. We provide these in Table~\ref{tab:notation} for ease of reference.

\begin{table}[ht]
\centering
\begin{tabular}{|l|l|}
\hline
\textbf{Notation/Acronym} & \textbf{Meaning} \\
\hline

$\AC$ & Affine-convex program; see Page~\pageref{eq:defineAC} \\
$\underline{\ACD}^k$ & Level-$k$ of a relaxation hierarchy for $\AC$; see Theorem~\ref{thm:simpleAChierarchy}\\
$\ACRLT^k$ & Level-$k$ RLT relaxation for $\AC$ when $x\in [0,1]^n$; see \eqref{eq:RLThypercube}\\
$\CM$ & Concave Minimization; see Page~\ref{eq:defineCM} \\
$\D$ & Formulation of concavo-convex program; see Page~\pageref{eq:defineD} \\
$\DD$ & Double Description Algorithm\\
$\DDP_\V$ & Convex hull of $\D$ using vertex-representation of $P$; see Page~\pageref{eq:DDPv}\\
$\underline{\DDR}^k$ & Level-$k$ of a relaxation hierarchy for $\D$; see Theorem~\ref{thm:simplehierarchy}\\
$\DBP$ & Disjoint Bilinear Program; see Page~\pageref{eq:defineDBP} \\
DP & Disjunctive Programming \\
$\FD$ & Facial Disjunctive Set; see Page~\pageref{eq:defineFD} \\
FDP & Facial Disjunctive Program\\
$\underline{\FDR}^k$ & Level-$k$ relaxation for $\FD$; see Theorem~\ref{thm:FDPhierarchy}\\
$\bar{A}$ & Matrix $\begin{pmatrix}b & -A\end{pmatrix}$ obtained while homogenizing $Ax\le b$ as $\bar{A}(x_0;x)\ge 0$\\
$C$ & Convex set containing $y$ in concavo-convex program $\D$ \\
$K$ & A polyhedral cone described as $\{(x_0,x)\mid \A(x_0; x) \ge 0, x_0\ge 0\}$\\
$P$ & Polytope describing constraints on $x$, written as $\{x\mid Ax\le b\}$ \\
$P^y$ & Polytope describing constraints on $y$, written as $\{y\mid A^yy\le b^y\}$ \\
RLT & Reformulation-Linearization Technique \\
$\V$ & Vertex representation of $P$; see Page~\pageref{eq:defineV} \\
CXP & Concavo-convex program; see its definition in $\D$\\
$x^{S,S'}$ & RLT product factor $\prod_{i\in S} x \prod_{i\in S'}(1-x_i)$\\
$\mu(x_0;x)$ & barycentric coordinate of $(x_0;x)$; See Definition~\ref{defn:barycentric}\\
$\theta(x_0;x)$ & barycentric coordinate of $(x_0;x)$; See Definition~\ref{defn:barycentric}\\
\hline
\end{tabular}
\caption{Common Notation and Acronyms}
\label{tab:notation}
\end{table}

\section{A Roadmap}

The primary goal of this paper is to propose finite hierarchies for $(\D)$ and $\FD$. We also establish a geometric-algebraic bridge connecting DP with RLT by deriving rational formulas for barycentric coordinates of polyhedral sets using the double-description procedure. This procedure addresses a long-standing question in computational geometry \cite{warren1996barycentric,wachspress2016rational}. Specializing the method to 0-1 hypercube provides a constructive interpretation of RLT product-factors as barycentric coordinates clarifying their role in relaxation construction.

\paragraph{Disjunctive Programming and Vertex Representations:} Assume a vertex description of $P$ is given as follows:
  \begin{equation}
    \nonumber
      \V = \left\{x\in \Re^n\;\middle|\; V\lambda = x, \lambda\in \Lambda_p\right\},
    \label{eq:defineV}
  \end{equation}
where $\Lambda_p = \left\{\lambda\in \Re^p\;\middle|\; \lambda_i\ge 0, \sum_{i=1}^p\lambda_i = 1\right\}$. Here, the columns of $\V$ contain the vertices of $P$. Then, in Section~\ref{sec:disjandRLT}, we use disjunctive programming to show that DBP can be reformulated as:
\begin{equation}\label{eq:DBPhull}
  \begin{alignedat}{2}
      &\min &\quad&\Trace\left(QY V^\intercal\right) \\
      &&& \begin{pmatrix}b^y  &- A^y\end{pmatrix}\begin{pmatrix}\lambda^\intercal\\ Y\end{pmatrix}\ge 0\\
      &&& y = Y e, x=V\lambda, \lambda\in \Lambda_p,
  \end{alignedat}            
\end{equation}
where $(\lambda_i V_{:,i}; Y_{:,i}) \in \lambda_i (P\times P^y)$. The objective $x^\intercal Q y$ is expressed as $\sum_i V^\intercal_{i,:}Q Y_{:,i}$ above, where, for $\lambda_i > 0$, $V^\intercal_{i,:}Q Y_{:,i}$ is a $\lambda_i$-scaling of $V^\intercal_{i,:}Q \frac{Y_{:,i}}{\lambda_i}$, the objective value of a feasible point. Therefore, \eqref{eq:DBPhull} is a convex combination of feasible objective values. These ideas are generalized for $(\D)$ in Lemma~\ref{lemma:disjVrep}.

\paragraph{Barycentric coordinates and representations:} A key contribution is to find algebraic expressions for $\lambda$ in terms of $x$; such expressions are referred to as barycentric coordinates. This is achieved for polyhedral cones in Section~\ref{sec:doubleandsimple} by adapting the double-description algorithm (see Algorithm~\ref{alg:ddwithmult}). Although barycentric coordinates are not in general unique, our specific choice leads to rational functions (see Proposition~\ref{prop:rationalfunction}). Then, in Section~\ref{sec:algebraic}, these coordinates are expressed as ratios of non-negative combinations of products of constraints (see Proposition~\ref{prop:mustructure}). They are further shown to be strictly positive in the interior of the polytope (see Proposition~\ref{prop:positivemu}). Proposition~\ref{prop:DDhypercube} shows that for $[0,1]^n$, these coordinates match the product-factors used in constructing RLT relaxations for 0-1 programs. The formulation of barycentric coordinates in terms of constraints with dual certificates from \eqref{eq:DBPhull} reveals, as shown in Theorem~\ref{thm:representationresult} and Example~\ref{ex:DBPproof}, algebraic proofs of optimality for DBP.

\paragraph{Convex hull formulations:} A simple hierarchy of relaxations for DBP is constructed by relaxing \eqref{eq:DBPhull} as follows:
\begin{equation}\label{eq:DBPHierarchy}
  \begin{alignedat}{2}
      &\min &\quad&\Trace\left(QY^k \left(V^k\right)^\intercal\right) \\
      &&& \begin{pmatrix}b^y &-A^y\end{pmatrix}\begin{pmatrix}\left(\lambda^k\right)^\intercal\\ Y^k\end{pmatrix}\ge 0\\
      &&& y = Y^k e, x=V^k\lambda^k, \lambda^k\in \Lambda_{p_k},
  \end{alignedat}
\end{equation}
where $V^k$ has $p_k$ columns, and these columns are feasible to and contain the vertices of a polytope $P^k$ that outer-approximates $P$. A similar construction is used in Theorems~\ref{thm:simplehierarchy} and \ref{thm:simpleAChierarchy} to derive a finite hierarchy for $(\D)$ and $(\AC)$. The polytopes $P^k$ are derived using constraints indexed $[1\cl{}k]$ in $P$'s description.

\paragraph{Tighter relaxations:} A significant effort in this paper is to tighten the simple hierarchy~\eqref{eq:DBPHierarchy} by exploiting the expressions for $\lambda^k$ as functions of $x$ by using algebraic techniques. We provide two illustrations. First, although \eqref{eq:DBPHierarchy} approximates $P$ using a polytope $P^k$ at level-$k$, we suggest using several approximations of $P$ at the same level. Then, the barycentric coordinates arising from these different outer-approximations are related with one another using algebraic techniques adding strength to the relaxation. To illustrate, assume that $P$ is outer-approximated by two different full-dimensional simplices. A superscript $r\in \{1,2\}$ is added to $V^k, Y^k, \lambda^k$ to distinguish these approximations, rewriting them as $V^{k,r}, Y^{k,r}, \lambda^{k,r}$. Since $V^{k,r}$ is a simplex, the barycentric coordinates are affine functions and can be derived by writing $(1;x) = (e; V^{k,r})\lambda^{k,r}$ such that $\lambda^{k,r} = (e; V^{k,r})^{-1}(1;x)$. Then, the following expresses $\left(\lambda^{k,r}, \bigl(Y^{k,r}\bigr)^\intercal\right)$ using original variables and linearization of $xy^\intercal$ as:
\begin{align*}
\left(\lambda^{k,r}, \bigl(Y^{k,r}\bigr)^\intercal\right) \xrightarrow[\text{represents}]{}\lambda^{k,r}(1;y)^\intercal = (e; V^{k,r})^{-1} 
   \begin{pmatrix}
      1 & y\\
      x & xy^\intercal
   \end{pmatrix}\\ \xrightarrow[\text{linearized as}]{} 
  (e; V^{k,r})^{-1} 
   \begin{pmatrix}
      1 & y\\
      x & W
   \end{pmatrix}.
\end{align*} 
Since $(1,y;x,W)$ is independent of $r$, the following equality is implied:
\begin{equation}\label{eq:expansioneq}
  (e; V^{k,1})\left(\lambda^{k,1}, \bigl(Y^{k,1}\bigr)^\intercal\right) = 
  (e; V^{k,2})\left(\lambda^{k,2}, \bigl(Y^{k,2}\bigr)^\intercal\right).
\end{equation}
This discussion highlights the advantage of using functional expressions for barycentric coordinates. It also shows that approximating $P$ in different ways simultaneously helps construct tighter relaxations. Such constructions play a crucial role in the hierarchy \eqref{eq:tighthierarchy} introduced in Section~\ref{sec:algebraic}.

Second, consider two disjoint subsets, $S$ and $S'$, of $\{1\cl{}n\}$ with $S\cup S' = k$. Then, for $j\not\in S\cup S'$, we have
\begin{equation*}
 \prod_{i\in S}x_i \prod_{i\in S'}(1-x_i) = \prod_{i\in S\cup \{j\}}  x_i \prod_{i\in S'}(1-x_i) + \prod_{i\in S}  x_i \prod_{i\in S'\cup \{j\}}(1-x_i),
\end{equation*}
an affine relation that relates product factors at level $k$ with those at level $k+1$. We show in Theorem~\ref{thm:expressmulin} that such relations generalize to arbitrary polyhedra by relating the barycentric coordinates of $P^k$ with those of $P^{k-1}$. Such a relation does not, however, hold for product-factors for polyhedral sets used in RLT (see Example~\ref{ex:affineDDnotRLT}). When used simultaneously, the two techniques (several outer-approximations and the affine relationship) work in tandem to yield additional benefits. Since a polytope $P^k$ can be tightened in various ways to approach $P$, the barycentric coordinates of these improved outer-approximations are related to one another via those of $P^k$.

\paragraph{Inequalities in the hierarchy:} We briefly describe the proposed relaxations. First, Algorithm~\ref{alg:ddwithmult} generates symbolic expressions for barycentric coordinates in the form of rational, non-negative functions which are used to construct the hierarchy described in \eqref{eq:tighthierarchy}. In particular, we draw the reader's attention to \eqref{eq:DEkmurecurse} and \eqref{eq:DEkmuy} which are derived in Theorem~\ref{thm:expressmulin}. These equalities play a central role in unrolling the derivations performed by Algorithm~\ref{alg:ddwithmult} providing a geometric-algebraic bridge connecting different levels of the hierarchy, a property RLT does not share (see Example~\ref{ex:affineDDnotRLT}). Expanding the polynomials in the numerators of these rational functions yields atomic expressions consisting of monomials divided by polynomials (see Example~\ref{ex:megacontinues}). These denominator polynomials are computed in Step~\ref{algstep:defineNtot}~of Algorithm~\ref{alg:ddwithmult} and can be expressed as sums of products of constraints (see Proposition~\ref{prop:mustructure}). To construct our relaxations, we linearize these atomic expressions using new variables and obtain non-trivial inequalities (see Proposition~\ref{prop:twoineq} and Remark~\ref{rmk:advantageoverRLT}). By leveraging the non-negativity of the denominators over the feasible region, these atomic expressions satisfy constraints from the sum-of-squares hierarchy~\cite{lasserre2002semidefinite} and fractional programming~\cite{he2023convexification} (see Remark~\ref{rmk:nonlinineq}). Furthermore, the inequalities defining $P$ can be lifted to the space of barycentric coordinates, as discussed in Remark~\ref{rmk:other_ineq} and thereafter.

\paragraph{Facial Disjunctions:} In Section~\ref{sec:fdp}, we consider the set $\FD$. In Theorem~\ref{thm:FDPhierarchy} we propose a relaxation hierarchy that terminates with the convex hull of $\FD$. Then, using Remark~\ref{rmk:barycentricproduct}, we introduce barycentric indicators for certain faces of $\prod_{i=1}^{n_p} P_i$. Finally, Proposition~\ref{prop:fdpnonlin}, provides algebraic functions that can replace variables introduced in the hierarchy of Theorem~\ref{thm:FDPhierarchy}.

%Then, we construct a hierarchy of relaxations of $\FD$ by rewriting the constraint $x_i\in \bigcup_{j=1}^{j_i} F_{ij}$ for $i\in T$ where $|T| = k$ in terms of the corresponding barycentric coordinates.  

\section{Preliminaries: Disjunctive Programming and RLT}\label{sec:disjandRLT}
Since the optimal solution to CXP (see the formulation $(D)$) occurs at a vertex of $P$, a convex hull formulation for $(D)$ can be developed using DP. Recall that 
$\V = \{x\in \Re^n\mid V\lambda = x, \lambda\in \Lambda_p\}$
is the $V$-representation of $P$, where $V\in \Re^{n\times p}$ and the columns of $V$ contain $\extpt(P)$. Let 
\begin{equation}
  \nonumber
    (\DDP_{\V})\quad \min\left\{\sum_{i=1}^p \ftilde\left(\lambda_iV_{:,i},y^i,\lambda_i\right) \;\middle|\; y^i\in \lambda_i C\; \forall i\in [\cl{}p], \lambda\in \Lambda_p, x=V\lambda\right\},
  \label{eq:DDPv}
\end{equation}
where $0C=\{0\}$ and $\ftilde(x,y,\lambda)$ is the homogenization of $f(x,y)$ using variable $\lambda$. Specifically, $\ftilde(x,y,\lambda) = \lambda f\left((x,y)/\lambda\right)$ when $\lambda > 0$ and $\ftilde(0,0,0) = 0$ \cite[Theorem 8.5]{Rockafella1970}. Then, standard DP arguments yield the following result.

\begin{lemma}\label{lemma:disjVrep}
    The optimal value of $(\DDP_{\V})$ matches that of $(D)$. Consider an optimal solution $(x_*, \lambda_*,y^1_*,\ldots,y^p_*)$ for $\DDP_{\V}$. Let $j'\in J := \{j\mid {(\lambda_*)}_j > 0\}$. Then, $\left(V_{\cl{},j'},y^{j'}_*/(\lambda_*)_{j'}\right)$ is optimal to $(D)$. Given $(\bar{x},\bar{y})\in P\times C$, introduce additional constraints, $x=\bar{x}$ and $\sum_{i=1}^p y^i = \bar{y}$, to $(\DDP_{\V})$. Then, the optimal value is the convex envelope of $f(x,y)$ over $P\times C$ at $(\bar{x},\bar{y})$.
  \end{lemma}

\long\def\disjstandardsimplex{
  We first argue the last statement of the result. Let $h(x,y)$ be the convex envelope of $f(x,y)$ over $P\times C$ and let:
  \begin{equation*}
    (\CE_\Lambda)\quad \check{f}(\bar{x},\bar{y}) = \min\left\{\sum_{i=1}^p\ftilde\left(\lambda_iV_{:,i},y^i,\lambda_i\right) \;\middle|\;
    \begin{aligned} &y^i\in \lambda_i C\; \forall i\in [\cl{}p]\\
       &\bar{x} = V\lambda, \sum_{i=1}^p y^i = \bar{y}, \lambda\in \Lambda_p\end{aligned}\right\}.
  \end{equation*}
  Since $(\CE_{\Lambda})$ is a convex program, its value function $\check{f}(\bar{x},\bar{y})$ is convex \cite[Corollary 3.32]{RockWets98}. Also, since $x\in P$ and $V\supseteq \extpt(P)$, we can choose $\lambda\in \Lambda_p$ such that $V\lambda=\bar{x}$ and conclude that
  \begin{equation*}
    f(\bar{x},\bar{y})\ge \sum_{i=1}^p \lambda_i f(V_{\cl{},i},\bar{y}) = \sum_{i=1}^p \ftilde(\lambda_iV_{:,i},\bar{y}\lambda_i,\lambda_i) \ge \check{f}(\bar{x},\bar{y}),
  \end{equation*}
  where the first inequality is because $f$ is concave with respect to $x$ for a fixed $y$ and the last inequality follows because $\lambda_i\bar{y}\in \lambda_iC$ and $\sum_{i=1}^p \lambda_i \bar{y} = \bar{y}$ imply that $(\lambda_iV_{\cl{}i},\bar{y}\lambda_i,\lambda_i)_{i\in \{\cl{}p\}}$ is feasible to $(\CE)_\Lambda$. Since $h(x,y)$ is the highest convex underestimator of $f(x,y)$, we have $h(\bar{x},\bar{y})\ge \check{f}(\bar{x},\bar{y})$. It follows from compactness of $P\times C$ and continuity of $f(x,y)$ that the convex hull of $f$ over $P\times C$ is compact \cite[Corollary 2.30]{RockWets98}. Assume $(\tilde{\lambda},\tilde{y}^1,\ldots,\tilde{y}^p)$ is optimal to $(\CE)_\Lambda$ and $J_{\tilde{\lambda}} = \{i\mid \tilde{\lambda}_i > 0\}$. Then,
  \begin{equation*}
    \check{f}(\bar{x},\bar{y}) = \sum_{i=1}^p \tilde{f}(\tilde{\lambda}_iV_{\cl{},i},\tilde{y}^i,\tilde{\lambda}_i) = \sum_{i \in J_{\tilde{\lambda}}} \tilde{\lambda}_i f\left(V_{\cl{},i},\dfrac{\tilde{y}^i}{\tilde{\lambda}_i}\right) \ge h(\bar{x},\bar{y}),
  \end{equation*}
  where the second equality is because it follows from $0C=\{0\}$ and $\tilde{y}^i\in \tilde{\lambda}_iC$ that $\tilde{y}^i = 0$ whenever $\tilde{\lambda}_i = 0$ and $\tilde{f}(0,0,0)$ was assumed to be zero. The last inequality is because $\tilde{y}^i/\tilde{\lambda}^i\in C$, $V\tilde{\lambda} = \bar{x}$, and $\sum_{i:\tilde{\lambda}_i > 0} \tilde\lambda_i\left(\tilde{y}^i/\tilde{\lambda}_i\right) = \bar{y}$. In other words, $h(\bar{x},\bar{y}) = \check{f}(\bar{x},\bar{y})$.

  The first statement of the result follows since $\check{f}(x,y)$ being the convex envelope of $f(x,y)$ over $P\times C$ achieves the same minimum as $f(x,y)$. Given the optimal solution for $\DDP_\Lambda$, we recover the solution for $(\D)$ as follows: 
  \begin{equation*}
    \begin{split}
    \opt_{\DDP_{\Lambda}} = \sum_{i\in J} f\left((\lambda_*)_iV_{\cl{},i},y_*^i, (\lambda_*)_i\right) = \sum_{i\in J} (\lambda_*)_i f(V_{\cl{},i}, y_*^i/(\lambda_*)_i)\\
    \quad \ge \min_{i\in J} f(V_{\cl{},i}, y_*^i/(\lambda_*)_i) \ge \opt_D,
    \end{split}
  \end{equation*}
  where the first equality is because $\tilde{f}(0,0,0) = 0$ and $J=\{j\mid (\lambda_*)_j > 0\}$, the first inequality is because $\lambda_*\in \Lambda_p$, and the second inequality is because $y^i/(\lambda_*)_i \in C$. Therefore, equality holds throughout. In particular, since the first inequality above is an equality, it follows that $\left(V_{\cl{}j},y_*^j/(x_*)_j\right)$ is optimal to $(\D)$ for any $j\in J$.}

  \iftoggle{allproofsinpaper}{
  \begin{proof}
    \disjstandardsimplex\qed
  \end{proof}}{}

  \iftoggle{allproofsinpaper}{}{The missing proofs are in the Appendix.} 
 % \begin{remark}\label{rmk:disjbil}
 %Consider $(\DBP)$. Lemma~\ref{lemma:disjVrep} provides the reformulation 
 %\begin{alignat*}{2}
% &\min &\quad&\Trace\left(QZ V^\intercal\right) \\
%  &&& Z_{\cl{}i} \in \lambda_i P^y\quad i\in [1\cl{}p]\\
%  &&& y = Z e\\
%  &&& x=V\lambda\\
%  &&&\lambda\in \Lambda_p.
% \end{alignat*}
% We mention that most of our discussion applies when $P^y$ is a polyhedron (instead of a polytope). Here, we will use $0P^y$ to denote the recession cone of $P^y$. To see this, assume that $(\lambda^*, Z^*)$ is optimal to the above reformulation. Then, we may write $QZ^*V^T$ as $\sum_{i=1}^pQ Z^*_{\cl{}i} V_{i\cl{}}^\intercal$.  If $\lambda^*_i = 0$, we have that $Z^*_{\cl{}i}$ belongs to the recession cone of $P^y$. 
% \qed 
%  \end{remark}
  We now consider the case where $P=[0,1]^n$. By Lemma~\ref{lemma:disjVrep}, a convex reformulation is obtained as follows:
\begin{equation*}
  (\DDP_\square)\;\min\left\{\sum_{S\subseteq\{\cl{}n\}} \ftilde\left(\lambda_S\chi_S,y^S,\lambda_S\right) \;\middle|\; y^S\in \lambda_S C, \lambda\in \Lambda_{2^n}, x=\sum_{\mathclap{S\subseteq \{\cl{}n\}}}\lambda_S\chi_S\right\},
\end{equation*}
where the vector $\lambda$ is indexed using subsets of $\{\cl{}n\}$ and $\chi_S\in\{0,1\}^n$ is $1$ indicates membership in $S$ by taking value $1$ at the corresponding indices. 

Throughout the paper, for $\alpha\in \Z^n$, we denote $\prod_{i=1}^n x_i^{\alpha_i}$ as $x^\alpha$. The RLT relaxation at level-$k$ linearizes monomials $x^{\chi_S}$ for $|S|\le k$ in inequalities obtained using the product factors $x^{\chi_{S'}}(1-x)^{\chi_{S''}}$, where $S'\cap S'' =\emptyset$ and $|S'\cup S''| = k$. We express $x^{\chi_S}$ as simply $x^S$ and $x^{\chi_{S'}}(1-x)^{\chi_{S''}}$ as $x^{S',S''}$. New variables are also introduced to linearize $yx^{S}$ and $zx^S$. A key insight of RLT is that the product linearizations and product factors are related to one another via a linear transformation \cite{sherali1990hierarchy}. Given a set $T$ such that $|T|=k$, we have 
\begin{equation}\label{eq:invertmult}
  x^{S,T\backslash S} = \prod_{S':[S,T]}(-1)^{|S'\backslash S|}x^{S'} \text{ and } x^{S} = \sum_{S'\subseteq T\backslash S} x^{S\cup S',T\backslash (S\cup S')},
\end{equation} 
where $[I,J]$ are the sets that contain $I$ and are contained in $J$. At level-$k$, there are $\sum_{i=0}^k\binom{n}{i}$ monomials---$x^S$ for all $S\subseteq \{\cl{}n\}$ such that $|S|\le k$---while the number of product factors---$x^{S,S'}$, $S\cap S'=\emptyset$, and $|S\cup S'| = k$---is $\binom{n}{k}2^k$. The monomials being fewer, if linearized, yield tighter relaxations. So, the level-$k$ RLT relaxation for $(\AC)$ expands and linearizes the following:
\begin{subequations}\label{eq:RLThypercube}
\begin{alignat}{4}
  &(\ACRLT)^k&\;&\min &\;& z_0 + \sum_{i=1}^n z_ix_i\\
  &&&\text{s.t.}&& zx^{S,S'} \ge x^{S,S'}g\left(\dfrac{yx^{S,S'}}{x^{S,S'}}\right) &\quad& S\cap S'=\emptyset, |S\cup S'|=k\label{eq:RLTHghomogenize}\\
%  &&&z_jx^{S,S'} \le x^{S,S'}g_j\left(\dfrac{yx^{S,S'}}{x^{S,S'}}\right) &\quad& S\cap S'=\emptyset, |S\cup S'|=k, j\in J\\
%  &&&z_j x^{S,S'}\le z^u_jx^{S,S'}&&S\cap S'=\emptyset, |S\cup S'|=k,j\not\in J\\
  &&&&&yx^{S,S'} \in x^{S,S'}C&& S\cap S'=\emptyset, |S\cup S'|=k\label{eq:RLTHChomogenize}\\
  &&&&&x^{S,S'} \ge 0&& S\cap S'=\emptyset, |S\cup S'|=k.\label{eq:RLTHlambda}
\end{alignat}
\end{subequations}

\begin{proposition}\label{prop:RLThypercubetight}
  When $P=[0,1]^n$, RLT produces a convex reformulation of $(\AC)$ at level $n$.
\end{proposition}
\begin{proof}
  At level $n$, the number of monomials is equal to the number of product factors because $\sum_{i=1}^n \binom{n}{i} = 2^n = \binom{n}{n}2^n$, and this counting argument shows that the linear transformation \eqref{eq:invertmult} relating $x^{S,\{\cl{}n\}\backslash S}$ and $x^S$ is invertible. Let $\lin(\cdot)$ denote the linearization of $(\cdot)$ after $x^{S,\{\cl{}n\}\backslash S}$ is expanded using the first expression in \eqref{eq:invertmult}. We write $\lin(x^{S,\{\cl{}n\}\backslash S})$ as $\lambda_S$, $\lin(yx^{S,\{\cl{}n\}\backslash S})$ as $y^S$, and $\lin(zx^{S,\{\cl{}n\}\backslash S})$ as $z^S$ and create inequalities satisfied by these linearizations. Since non-negativity of $\lambda_S$ follows from \eqref{eq:RLTHlambda} and the second inequality in \eqref{eq:invertmult} implies $\sum_{S\subseteq \{\cl{}n\}}\lambda_S = 1$, $(\lambda_S)_{S\subseteq\{\cl{}n\}}$ belongs to $\Lambda_{2^n}$.
  %, where $\Lambda_{2^n}$ is the standard simplex in $\Re^{2^n}$, \textit{i.e.}\/, $\left\{\gamma\in \Re^{2^n}\mid \gamma\ge 0, \sum_{i=1}^{2^n}\gamma_i = 1\right\}$.  
  Using the second equality in \eqref{eq:invertmult}, we obtain 
  \begin{align*}
    &z_0 = \sum_{S\subseteq \{\cl{}n\}}\lin\left(z_0 x^{S,\{\cl{}n\}\backslash S}\right) = \sum_{S\subseteq\{\cl{}n\}}z^S_0\\
    &\lin(z_ix_i) = \sum_{S\subseteq \{\cl{}n\}, i\in S}\lin\left(z_i x^{S,\{\cl{}n\}\backslash S}\right) = \sum_{S\subseteq \{\cl{}n\}, i\in S}z^S_i.
  \end{align*} 
  Combining, we have
  \begin{equation*}
    \begin{split}
    &z_0 + \sum_{i=1}^n \lin(z_ix_i) = \sum_{S\subseteq\{\cl{}n\}}\left(z^S_0+(\chi_S)_iz^S_i\right)\\
    &\qquad \ge \sum_{S\subseteq \{\cl{}n\}}\left[\lambda_S g_0\left(\dfrac{y^S}{\lambda_S}\right) + \sum_{i=1}^n (\chi_S)_i\lambda_S g_i\left(\dfrac{y^S}{\lambda_S}\right)\right] = \sum_{\mathclap{S\subseteq\{\cl{}n\}}}\ftilde\left(\lambda_S\chi_S,y^S,\lambda_S\right),
    \end{split}
  \end{equation*}
  where the inequality follows from \eqref{eq:RLTHghomogenize} and $(\chi_S)_i\ge 0$. Moreover, $y^S\in \lambda_S C$ by \eqref{eq:RLTHChomogenize}. Clearly, \eqref{eq:RLThypercube} is a valid relaxation of $(\AC)$ when $P=[0,1]^n$. Since we have shown that $(\lambda_S,y^S)$ satisfy the constraints in $(\DDP_\square)$ and the objective of \eqref{eq:RLThypercube} overestimates that of $(\DDP_\square)$, it follows that \eqref{eq:RLThypercube}, at level-$n$, is a valid convex reformulation of $(\AC)$.\qed
\end{proof}

The above result shows that RLT terminates finitely yielding the convex hull of $(\AC)$. This result is akin to the use of RLT to develop convex envelopes of multilinear functions over a unit hypercube~\cite{sherali1997convex} but differs in that $y$ lies in a general compact convex set and participates in the objective via arbitrary convex functions. 

Lemma~\ref{lemma:disjVrep} can be used to derive the convex hull when $(x_1,\ldots,x_n) \in \Lambda_{q_1}\times\cdots\times \Lambda_{q_n}$. It is shown in~\cite{oh2022algorithms,oh2024branchandbound} that when $P$ is of the form $\Lambda_{q_1}\times\cdots\times \Lambda_{q_n}$ RLT yields the convex hull at level-$n$. Our focus is on the structure of this RLT relaxation and that expanding the product factors results in a relaxation that is at least as strong. To see this, observe that the product factors in the relaxation are of the form $\prod_i x_{ij_i}$, where $j_i \in \{1\cl{}q_i\}$, and $x_{iq_i}$ is rewritten as $1 - \sum_{j=1}^{q_i-1} x_{ij}$ exploiting $\sum_{j=1}^{q_i}x_{ij} = 1$. After distributing products over addition, the monomials $\prod_i x_{ij_i}$, for $j_i \in \{0\cl{}(q_i-1)\}$, are linearized, where $x_{i0}$ denotes the constant $1$. We now present a counting argument showing that expanding the product form, in this way, does not weaken the relaxation and allows it to capture linear dependencies among the product factors.

\begin{proposition}\label{prop:expansionhelps}
Consider $(x_1,\ldots,x_n)\in (\Lambda_q)^n$. There are $\binom{n}{k}q^k$ product factors of the form $\prod_{i\in T} x_{ij_i}$ where $|T| \le k$. Replacing $x_{iq}$ with $1-\sum_{j=1}^{q-1} x_{ij}$ and expanding the expression reduces the number of nonlinear monomials to $\sum_{i=1}^k\binom{n}{i}(q-1)^i$. The resulting RLT relaxation exploits linear dependencies between the product-factors.
\end{proposition}
\begin{proof}
There are $\binom{n}{k}$ ways to select $T$, and to create a product factor, for each $i\in T$, we may select any element of $\{x_{i1},\ldots,x_{iq}\}$. This yields the count $\binom{n}{k}q^k$. When we replace $x_{iq_i}$ with $1-\sum_{j=1}^{q-1} x_{ij}$ and expand the product factor, the resulting monomials have degree $i\le k$ where each term in the monomial can take one of $q-1$ values yielding $\sum_{i=1}^k \binom{n}{i}(q-1)^i$ possible choices. Since
\begin{equation*}
\binom{n}{k}q^k  = \sum_{i=0}^k \binom{n}{k}\binom{k}{i}(q-1)^i 
= \sum_{i=0}^k \binom{n}{i}\binom{n-i}{k-i}(q-1)^i
\ge \sum_{i=0}^k \binom{n}{i}(q-1)^i,
\end{equation*}
where the last inequality holds because $i\le k\le n$. The linear dependencies exploited by the relaxation after expansion are seen by choosing distinct supersets $T$ and $T'$ of $S$ and observing that:
\begin{equation}\label{eq:affinerelation}
\prod_{i\in S}x_{ij_i} = \prod_{i\in S}x_{ij_i}\prod_{j\in T\backslash S} (x_{j1}+\cdots+x_{jq}) =  \prod_{i\in S}x_{ij_i}\prod_{j\in T'\backslash S} (x_{j1}+\cdots+x_{jq}).
\end{equation}
In other words, for a given $S$ with cardinality $i$, there are $\binom{n-i}{k-i}$ ways in which this monomial can be expressed linearly using product factors $\prod_{i\in T} x_{ij_i}$. \qed
\end{proof}

Proposition~\ref{prop:expansionhelps} shows that recognizing the affine relationship between product factors at level $\ell$, such as $\prod_{i\in S}x_{ij_i}$, where $|S|=\ell$, with those at a level $k > \ell$, affords an opportunity to tighten the relaxation. We will show later that such affine relationships between barycentric coordinates at adjacent levels can be developed for general polyhedral cones (see Theorem~\ref{thm:expressmulin}).

%\begin{remark}\label{rmk:linkagefactors}
%  Proposition~\ref{prop:RLThypercubetight} generalizes to the case where $P$ is a Cartesian product of simplices using the tightness of RLT for $(\AC)$ when $P$ is a simplex assuming implicit equalities of $P$ are known and utilized in constructing RLT relaxations \cite{oh2024branchandbound}. Consider $(x_1,\ldots,x_n) \in \Lambda_{q_1}\times\cdots\times \Lambda_{q_n}$. Then, the product factors are $\prod_{i} x_{ij_i}$, where $j_i\in \{\cl{}q_i\}$. To linearize, $x_{iq_i}$ is written as $1-\sum_{j=1}^{q_i-1} x_{ij}$ and then, after expanding, the monomials $\prod_{i}x_{ij_i}$ $j_i \in \{0\cl{}(q_i-1)\}$ are linearized, where $x_{i0}$ denotes the constant $1$. To compare the number of product-factors and the monomials, assume that, for all $i\in\{\cl{}n\}$, $q_i=q$. Then, at level-$k$, there are $\sum_{i=1}^k\binom{n}{i}(q-1)^i$ monomials while the number of product factors is $\binom{n}{k}q^k$. Since $\binom{n}{k}q^k  = \sum_{i=0}^k \binom{n}{k}\binom{k}{i}(q-1)^i \ge \sum_{i=0}^k \binom{n}{i}(q-1)^i$, the fewer monomials allow RLT to exploit linear dependencies between the product-factors. \qed
%\end{remark}

\section{Double description and a simple finite hierarchy}\label{sec:doubleandsimple}

For a polyhedron $P=\{x\mid Ax\le b\}$, where $A\in \Re^{m\times n}$, the double description (DD) procedure, as outlined in Algorithm~\ref{alg:ddwithmult}, derives its $V$-representation in $m$ iterations \cite{motzkin1953double}. Our finite hierarchy for $(\D)$ critically relies on this property of DD. We leverage the fact that Algorithm~\ref{alg:ddwithmult} creates increasingly tighter outer-approximations of $P$ with each iteration. By incorporating these approximations into the relaxation described in Lemma~\ref{lemma:disjVrep}, we construct a finite hierarchy of relaxations that terminates with the convex hull of $(\D)$. This basic relaxation hierarchy is derived using DP. %, the main technique used to prove Lemma~\ref{lemma:disjVrep}. 
However, its connection to RLT is less direct and will be explored in Section~\ref{sec:algebraic}. 

This connection relies on the barycentric coordinates, which we will next derive by adapting the DD procedure. While the DD procedure is typically viewed as a numerical algorithm for computing the $V$-representation of polyhedra, we have augmented it with symbolic calculations (Steps~\ref{algstep:definemu1}, \ref{algstep:definetheta1}, and \ref{algstep:definemu2}) to obtain Algorithm~\ref{alg:ddwithmult}. This modified procedure computes barycentric coordinates for each point $x\in P$. These coordinates represent each point in $P$ in terms of its vertices, rays, and lineality directions. Such a representation is guaranteed by the Minkowski-Weyl Theorem for polyhedra \cite{conforti2014integer}. While these coordinates are often treated as numerical values for each $x$, it is crucial for our purposes to regard them as functions of the point $x$ being represented. We now formally define the barycentric coordinates for polyhedral cones, which also allows us to extend this framework to polyhedra, since the latter can be homogenized to yield a polyhedral cone.

\begin{definition}\label{defn:barycentric}
    Given a cone $K$, a set of its rays $(r_i)_{i=1}^p$ and its lineality directions $(l_i)_{i=1}^q$, we say $\left(\left(\mu_{r_i}(x)\right)_{i=1}^{p},\left(\theta_{l_i}(x)\right)_{i=1}^{q}\right)$ are barycentric coordinates for $K$ if, for all $x\in K$, $x=\sum_{i=1}^{p} \mu_{r_i}(x)r_i + \sum_{i=1}^{q}\theta_{l_i}(x)l_i$, with $\mu_{r_i}(x) \ge 0$.
\end{definition}

In order for barycentric coordinates to exist, it follows easily that the set of rays must contain $\extray(K)$ and the lineality directions must be covered by $(l_i)_{i=1}^q$ and $(r_i)_{i=1}^p$. Although barycentric coordinates have been studied extensively in computer-aided-design and finite-element literature~\cite{warren1996barycentric,warren2003uniqueness,warren2007barycentric,wachspress2016rational}, it has not been realized that DD procedure produces these coordinates. It is known that the barycentric coordinates cannot be polynomials even for quadrilaterals in two dimensions~\cite{warren2003uniqueness,wachspress2016rational} and rational functions suffice for general polytopes \cite{warren1996barycentric}. We operationalize this insight by algorithmically constructing the barycentric coordinates as rational functions of $x$.

\begin{algorithm}
    \normalsize
    \caption{Double Description Algorithm with Barycentric Coordinates}\label{alg:ddwithmult}
    \begingroup
    \fontsize{8}{14}\selectfont
    \begin{algorithmic}[1]
      \State\label{algstep:initialize} $L^0=(0;\Id_{n})$;\; $\theta^0 = x$\; $R^0=e_0$\; $\mu^0=(x_0)$\; $E^0 = []$.
      %$E^0 = []$;\; $M^0 = []$;\; $\gamma^0 = \emptyset$.
      \For{$k$ from $1$ to $m$}\label{algstep:ddforloop}
        \State $p_{k-1} \leftarrow \text{ColumnDimension}(R^{k-1})$; $q_{k-1} \leftarrow \text{ColumnDimension}(L^{k-1})$
        \State\label{algstep:loopinit}  $E^{k} = E^k$\;;\; $\theta^k = \theta^{k-1}$;\; $L^k = L^{k-1}$.
        %$M^k = M^{k-1}$;\;$\gamma^k = \gamma^{k-1}$;\;
        \State\label{algstep:definealphabeta} $\alpha \leftarrow \A_{k,\cl}L^{k-1}$; $\beta \leftarrow \A_{k,\cl{}}R^{k-1}$
        \If{$q_{k-1} > 0$ and $\alpha \ne 0$}\label{algstep:ifcondition}
        \State\label{algstep:definexi} $\xi \leftarrow \arg\min\{i\mid \alpha_i\ne 0\}$
          \If{$\alpha_{\xi}<0$}
            \State\label{algstep:negateL} $L^{k-1}_{:,\xi} \leftarrow -L^{k-1}_{:,\xi}$;\; $\theta^{k-1}_{\xi} \leftarrow -\theta^{k-1}_{\xi}$;\; $\alpha_{\xi} = -\alpha_{\xi}$
          \EndIf
          \State\label{algstep:definemu1} $\mu^k \leftarrow \left(\theta^{k-1}_\xi+\sum_{j=\xi+1}^{q_{k-1}}\dfrac{\alpha_j\theta^{k-1}_j}{\alpha_\xi} + \sum_{j=1}^{p_{k-1}} \dfrac{\beta_j\mu^{k-1}_j}{\alpha_\xi}\;,\; \left(\dfrac{\mu^{k-1}_j}{\alpha_\xi}\right)_{j=1}^{p_{k-1}}\right)$
          \State\label{algstep:definetheta1} $\theta^k \leftarrow \left(\left(\theta^{k-1}_j\right)_{j=1}^{\xi-1}\;,\;\left(\dfrac{\theta^{k-1}_{j}}{\alpha_{\xi}}\right)_{j=\xi+1}^{q_k}\right)$
          \State\label{algstep:defineR1} $R^k\leftarrow \left(L^{k-1}_{:,\xi}\;,\; \left(\alpha_{\xi} R^{k-1}_{\cl{},j}-\beta_j L^{k-1}_{:,\xi}\right)_{j=1}^{p_{k-1}}\right)$
          \State\label{algstep:defineL1} $L^k\leftarrow \left(\left(L^{k-1}_{\cl{},j}\right)_{j=1}^{\xi-1}\;,\;\left(\alpha_\xi L^{k-1}_{\cl{},j}-\alpha_j L_{\cl{},\xi}\right)_{j=\xi+1}^{q_{k-1}}\right)$
          %\State\label{algstep:defineM1gamma1} $\gamma^k = \gamma^{k-1}$;\; $M^k = M^{k-1}$
        \Else
          \State \label{algstep:defineNsets}$N^- = \{i\mid \beta_i < 0\}$;\; $N^0 = \{i\mid \beta_i = 0\}$;\; $N^+ = \{i\mid \beta_i > 0\}$
          \State \label{algstep:defineNtot} $N_{\text{tot}} = \sum_{i\in N^+}\mu_i\beta_i$
          \State\label{algstep:definemu2} $\mu^k \leftarrow \left(\mu^{k-1}_{N^0}, \left(\mu^{k-1}_i+\sum\limits_{\mathclap{j\in N^-}}\dfrac{\mu^{k-1}_j\mu^{k-1}_i\beta_j}{N_{\text{tot}}}\right)_{i\in N^+}\mskip-10mu,\left(\dfrac{\mu^{k-1}_j\mu^{k-1}_i}{N_{\text{tot}}}\right)_{\myratio{i\in N^+}{j\in N^-}}\right)$
          \State\label{algstep:defineR2} $R^k\leftarrow  \left(R^{k-1}_{\cdot N^0}\;,\; R^{k-1}_{\cdot N^+}\;,\; \left(\beta_iR^{k-1}_{:,j}-\beta_jR^{k-1}_{:,i}\right)_{i\in N^+,j\in N^-}\right)$
          \State\label{algstep:defineE2} $E^k = \left(E^{k-1},k\right)$ if $N^+ = \emptyset$  
          %\State\label{algstep:definegamma2} $\gamma^k = \left(\gamma^{k-1}, (\mu^{k-1}_{j\in N^-})\right)$ if $N^+ = \emptyset$
          %\State\label{algstep:defineM2} $M^k = \left(M^{k-1},\left(R^{k-1}_{\cl{},j}\right)_{j\in N^-}\right)$ if $N^+ = \emptyset$     
        \EndIf
      \EndFor
    \end{algorithmic}
    \endgroup
  \end{algorithm}

We refer the reader to \cite{Fukuda1995,padberg2013linear} for a detailed treatment of acceleration and numerical stabilization techniques for DD algorithm, which include reordering inequalities, removing redundant points from the $V$-representation, and reducing the size of generated numbers. We do not need these strategies for our theoretical developments, but expect that they would help with scaling and numerical performance of our hierarchy. For concise notation, it is helpful at this stage to arrange the rays $r_i$ and lineality directions $l_i$ as columns of matrices $R$ and $L$ respectively that are updated at each iteration of Algorithm~\ref{alg:ddwithmult}.

In summary, Algorithm~\ref{alg:ddwithmult} computes the barycentric coordinates of a cone $K = \{(x_0;x)\in \Re^{n+1}\mid \A (x_0;x)\ge 0, x_0\ge 0\}$, where $\A\in \Re^{m\times (n+1)}$ and $x\in \Re^n$. This cone is obtained by homogenizing $P$ using the variable $x_0$ so that $\A = (b; -A)$. Apart from initialization, the main content of the algorithm is the loop beginning at Step~\ref{algstep:ddforloop}. Thus, by iteration $k$ or $k^{\text{th}}$ iterates of Algorithm~\ref{alg:ddwithmult}, we respectively mean the $k^\text{th}$ iteration of the loop or the quantities computed during that iteration. It is well-known that DD algorithm, at iteration $k$, produces a $V$-description of $K^k = \{(x_0;x)\in \Re^{n+1}\mid \A_{1\cl{}k,\cl{}}(x_0;x)\ge 0, x_0\ge 0\}$, which is obtained by restricting attention to the first $k$ inequalities describing $K$ \cite[Section 7.4]{padberg2013linear}. 
\iftoggle{ancilliarydiscussions}{
    Several intricacies of the DD algorithm are crucial for deriving the barycentric coordinates and proving their properties. To make our treatment self-contained, we include a short explanation of the algorithm and its correctness in Appendix~\ref{app:ddalgorithm}. The reader is referred to \cite{padberg2013linear,conforti2014integer} for additional background on the algorithm.}{\ddalgorithm}
At the final step, since $K^m = K$, the algorithm computes a $V$-description of $K$. At each iteration, this $V$-description, detailed below in \eqref{eq:pointrep}, uses the $k^\text{th}$ iterates, $R^k\in \Re^{k\times p_k}$ and $L^k\in \Re^{k\times q_k}$, where $R^k$ (resp. $L^k$) contains rays (resp. lineality directions) of $K^k$ as columns. We remark that Algorithm~\ref{alg:ddwithmult}, as stated, may produce feasible rays that are not extremal, so that the resulting description is not minimal. However, it can be made minimal later by pruning, as described in Remark~\ref{rmk:removeredundant}. As mentioned before, we keep track of the barycentric coordinates $\mu^k$ (resp. $\theta^k$) for $R^k$ (resp. $L^k$) via symbolic calculations. For all $(x_0;x) \in K^k$, this provides the following linear-algebraic version of the representation in Definition~\ref{defn:barycentric}:
\begin{equation}\label{eq:pointrep}
    R^k\mu^k + L^k\theta^k = (x_0;x), \quad \mu^k \ge 0.
\end{equation}
\long\def\ddalgorithm{%
The DD algorithm produces a $V$-representation of $K^k = \{\A(x_0;x)\ge 0, x_0\ge 0\}$ which lends itself naturally to an inductive argument. Here, we sketch the intuitive ideas behind the correctness of the algorithm. We remark that several modifications have been devised to the DD procedure to generate minimal descriptions \cite{Fukuda1995,padberg2013linear} so that the matrices $R^k$ contain only extremal rays. However, we do not utilize these extensions and do not detail them here. 

The base case begins with a $V$-representation of $K^0=\{(x_0;x)\mid x_0\ge 0\}$ by setting $L^0 = (0; \Id_{n})$, $\theta^0 = x$, $R^0=e_{0}$, $\mu^0 = (x_0)$, where $\Id_{n}$ denotes the identity matrix in $\Re^{n\times n}$. By the induction hypothesis, at iteration $k+1$, a $V$-representation is available for $(x_0;x)\in K^k$. Then, at iteration $k+1$, depending on the outcome of the test at Step~\ref{algstep:ifcondition}, Algorithm~\ref{alg:ddwithmult} executes either Steps~\ref{algstep:definexi}-\ref{algstep:defineL1} (Case 1) or Steps~\ref{algstep:defineNsets}-\ref{algstep:defineE2} (Case 2). The discriminating test checks whether $\A_{k+1,\cl{}}$, the $k+1^\text{st}$ row of $\A$, is orthogonal to the lineality space $L^{k}$ of $K^{k}$. If not, it executes Case 1, otherwise Case 2. We remark that, by rearranging rows, at iteration $k+1$, we may pick any row with index in $\{(k+1)\cl{}m\}$  that is not orthogonal to $L^{k}$. Doing so, Case 1 is executed in the first iterations until $L^m$ has been identified, and the remaining iterations only change $R^{k+1},\ldots,R^m$ by executing Case 2. Now, we consider the two cases in more detail.
\begin{casesp}
    \item Let $L^k_{\cl{},\xi}$, where $\xi\in[\cl{}q_k]$, be a column that is not orthogonal to $\A_{k+1,\cl{}}$. Let $\alpha  = A_{k+1,\cl{}}L^k$ and, by negating $L^k_{:,\xi}$ if necessary, assume that $\alpha_\xi > 0$ (see Step~\ref{algstep:negateL}). Then, Algorithm~\ref{alg:ddwithmult} constructs columns of $L^{k+1}$ (resp. $R^{k+1}$), to be orthogonal to $A_{k+1,\cl{}}$, by adding a suitably chosen multiple of $L^k_{\cl{},\xi}$ to other columns $L^k_{\cl{}j}$ (resp. $R^k_{\cl{}j}$), $j \in [\cl{}q_k]\backslash\{\xi\}$ (resp. $j\in [\cl{}p_k]$) of $L^k$ (resp. $R^k$). In addition, we prepend the column $L^k_{\cl{},\xi}$ to $R^{k+1}$. The key difference in the representation \eqref{eq:pointrep} at $k+1^{\text{st}}$ iteration is that the multiplier of $L^k_{\cl{},\xi}$ cannot be negative. The observation that supports this reclassification of $L^k_{\cl{},\xi}$ as a ray at $k+1^\text{st}$ iteration is that for any $(x_0;x)\in K^{k+1}$, $0\le \frac{1}{\alpha_\xi}A_{k+1,\cl{}} (x_0;x) = \mu^{k+1}_1$, and, so, the multiplier $\mu^{k+1}_1$ associated with $L^k_{:,\xi}$ (or $R^{k+1}_{\cl{},1}$) is non-negative. Here, the inequality follows since $\alpha_\xi \ge 0$ and $K^{k+1}$ satisfies $\A_{k+1,\cl{}}(x_0;x)\ge 0$, and the equality follows from \eqref{eq:pointrep} and that, by construction, the remaining columns of $L^{k+1}$ and $R^{k+1}$ are orthogonal to $\A_{k+1,\cl{}}$.
    \item Since the $k+1^\text{st}$ row of $\A$, $A_{k+1,\cl{}}$, is orthogonal to $L^k$, the lineality directions of $K^k$ continue to be those of $K^{k+1}$ as well. Nevertheless, some columns of $R^k$, those with indices in $N^-$, do not correspond to feasible points (see Step~\ref{algstep:defineNsets}). The remaining columns are either orthogonal to $A_{k+1,\cl{}}$ and belong to $N^0$, or are strictly feasible to the constraint and classified as $N^+$. This classification is performed easily by computing $\beta = A_{k+1,\cl{}}R^k$, which are the slacks for each column of $R^k$ on the $k+1^\text{st}$ constraint. By scaling $R^k$, if necessary, we assume that $|\beta_j| \le \mu_j$, for all $j\in N^+\cup N^-$. The algorithm then replaces the columns with indices in $N^-$ with additional feasible columns, one for each pair $(i,j)$ in $N^+\times N^-$, utilizing the following observation.  
    
    Consider reducing the multiplier for a single column, with an index $j$, that is in $N^-$. For $i\in N^+$, we simultaneously reduce $\mu_j$ by $-1/\beta_j$ and $\mu_i$ by $1/\beta_i$ by assigning a multiplier $-1/(\beta_i\beta_j)$ to a new ray $\beta_i R_{\cl{},j} - \beta_j R_{\cl{},i}$. This new ray barely satisfy the $k+1^{\text{st}}$ inequality and is feasible to $K^{k+1}$. The ray $i$ in $N^+$ has $\mu_i\beta_i$ such transfers available, while the ray $j$ requires $-\mu_j\beta_j$ transfers to reduce its multiplier to zero. Then, these transfers are determined by a transportation problem on a complete bipartite graph with nodes in $N^+\cup N^-$, where nodes in $N^+$ (resp. $N^-$) have a supply of $\mu_i\beta_i$ (resp. demand of $-\mu_j\beta_j$). This transportation problem is feasible since $0\le A_{k+1,\cl{}}(x_0;x) = \sum_{i\in N^+}\beta_i\mu_i + \sum_{j\in N^-}\beta_j\mu_j$, \textit{i.e.}\/, supply exceeds demand. Then, we use a solution where the transfer from each node $i\in N^+$ to every node $j\in N^-$ is $(-\mu_i\beta_i\mu_j\beta_j)/(\sum_{i'\in N^+}\mu_{i'}\beta_{i'})$. The feasibility of this solution follows by scaling the supply and demand so that total supply is one unit and shipping the product of supply and demand along each arc from source to destination.
    %observing that each arc connecting $i\in N^+$ to $j\in N^-$ can ship the product of supply and demand at its source and destination to meet all the demand (reduce the multipliers of $N^-$ to zero) while not exceeding supply (maintaining non-negativity of multipliers) for rays with indices in $N^+$. 
\end{casesp} 
} 

\paragraph{Finiteness of the basic hierarchy:} Before exploring the structure of barycentric coordinates, we argue, as was claimed at the beginning of this section, that a basic hierarchy, oblivious to the structure of these coordinates, is already finite. Since our hierarchy will tighten this basic hierarchy, it will inherit this finiteness property. To find the barycentric coordinates for $P$ instead of its homogenization, $K$, we will replace $x_0$ with $1$ and dehomogenize $R^k$ and $\mu^k$, by scaling those columns of $R^k$ that have positive entries in the first row.

\begin{definition}\label{defn:dehomogenize}
    Let $J=\{j\mid R_{1, j} > 0\}$. We say a pair $\left(V,\lambda(x)\right)$, is a \emph{dehomogenization} of $\left(R,\mu(x_0;x)\right)$ if $\left(V_{\cl{},j},\lambda_{j}(x)\right) = \left(R_{\cl{},j}/R_{1j},R_{1j}\mu_{j}(1;x)\right)$ for $j\in J$ and $\left(V_{\cl{},j},\lambda_{j}(x)\right) = \left(R_{\cl{},j},\mu_j(1;x)\right)$ otherwise. We say $\left(R(x),\mu(1;x)\right)$ are \emph{dehomogenized} if this operation is performed in-place.
\end{definition}

Since $P$ is a polytope, there exists a $k$ such that $R^k_{1,\cl{}}$ is a vector of all ones after $R^k$ has been dehomogenized. Moreover, as will be clear later, we can add a suitable inequality bounding slacks of $n$ linearly independent constraint expressions to ensure this occurs at the end of iteration $n+1$.

\begin{theorem}\label{thm:simplehierarchy}
    Let $\left(R^k,\mu^k(1;x)\right)$ be dehomogenized and $\bar{k}$ be the minimum index at which $P^{\bar{k}} = K^{\bar{k}}\cap\{(x_0;x)\mid x_0=1\}$ is bounded so that $R^k_{1,j} = 1$ for all $j\in[\cl{}p_k]$. For $k\in [\bar{k}\cl{}m]$, we relax $(\D)$ as follows:
  \begin{subequations}\label{eq:DD}
  \begin{alignat}{4}
    &(\DDR^k)&\quad&\min&&\sum_{i=1}^{p_k} \ftilde\left(R^k_{[2\cl{}],i}\mu^k_i(1;x),y\mu^k_i(1;x),\mu^k_i(1;x)\right)\label{eq:objDDk} \\
    &&&\text{s.t.}&&y\mu^k_i(1;x)\in \mu^k_i(1;x) C &\quad& i\in [:p_k]\label{eq:DDkyscale}\\
    &&&&&\mu^k_i(1;x)\ge 0 && i\in [:p_k]\label{eq:DDknonneg}\\
    &&&&&R^k \mu^k(1;x) = (1;x).\label{eq:DDklinprecision}
  \end{alignat}
  \end{subequations}
  Further, we relax $(\DDR^k)$ to a convex optimization problem $\underline{\DDR}^k$ by replacing $\mu^k(1;x)$ (resp. $y\mu^k_i(1;x)$) with a variable $\lambda^k\in \Re^{p_k}$ (resp. $y^{k,i}$).
  Then, $\opt_{\underline{DDR}^{\bar{k}}}\le \ldots\le \opt_{\underline{\DDR}^{m}} = \opt_D$. 
  \end{theorem}
  \begin{proof}
    Recall that Algorithm~\ref{alg:ddwithmult} yields the $V$-representation of $P^k = \{x\mid b_{1\cl{}k} - A_{1\cl{}k,\cl{}}x\}$ at the $k^\text{th}$ iteration. Clearly, $\bar{k}\le m$, since $P$, being a polytope, does not admit any rays. Moreover, since $k\ge \bar{k}$, $R^k_{[2\cl{}],i}$ is the $i^\text{th}$ vertex of $P^k$. To see that $(\DDR^k)$ is a relaxation of $(\D)$, we show that a solution of $(\D)$ can be transformed to a feasible solution of $(\DDR^k)$ with an objective value no higher than in $(\D)$. We choose a solution $(x,y)$ of $(\D)$ such that $x\in P$ and $y\in C$. Then, $x\in P\subseteq P^k$. Therefore, $(1;x) = R^k\mu^k(1;x)$. It follows that $(x,\mu^k(1;x),y)$ satisfies \eqref{eq:DDkyscale}-\eqref{eq:DDklinprecision}. Further,
    \begin{equation*}
        \begin{split}
        \sum_{i=1}^{p_k} \ftilde\left(R^k_{[2\cl{}],i}\mu^k_i(1;x),y\mu^k_i(1;x),\mu^k_i(1;x)\right) =  \sum_{i=1}^{p_k} \mu^k_i(1;x)f\left(R^k_{[2\cl{}],i},y\right)\\
        \le f\left(R^k_{[2\cl{}],\cl{}}\mu^k(1;x),y\right) = f(x,y),
        \end{split}
    \end{equation*}
    where the first equality is because $\ftilde$ is  a homogenization of $f$, the inequality is because $f$ is concave in $x$ for a given $y$ and, for $k\ge \bar{k}$, $\mu^k_i(1;x)\in \Lambda_{p_k}$, and the last equality is because of \eqref{eq:pointrep}.

    The linearization of $(\DDR^k)$ obtained by replacing $\mu^k(1;x)$ (resp. $y\mu^k(1;x)$) with $\lambda^k$ (resp. $y^{k,i}$) yields the relaxation in Lemma~\ref{alg:ddwithmult}. This is because the first constraint in \eqref{eq:DDklinprecision} reduces to $\sum_{i=1}^{p_k}\lambda_i = 1$, while the remaining constraints model $x = \sum_{i=1}^{p_k} R_{[2\cl{}],i}\lambda_i$. It follows from Lemma~\ref{alg:ddwithmult} that $\opt_{\underline{\DDR}^k}$ is the optimal value of $(\D)$ where $P$ has been relaxed to $P^k$. Since $P^k\supseteq P^{k+1}$, we have $\opt_{\underline{\DDR}^{k}} \le \opt_{\underline{\DDR}^{k+1}}$. Moreover, since $P^m = P$, $\opt_{\underline{\DDR}^m} = \opt_D$.
  \qed
  \end{proof}

For $(\AC)$, we use separability with respect to $x$ to construct relaxations at earlier iterations, even though $P^k$ is unbounded.
For simplicity, assume that, for $j\in \{\cl{}\varrho\}$,  $g_j(x)$ are affine functions, so that $P\subseteq\{x\mid x_i\ge 0 \text{ for } i\in \{(\varrho+1)\cl{}\}\}$.
First, we update the initialization step of Algorithm~\ref{alg:ddwithmult} so that:
\begin{equation}\label{eq:alternateinit}
  \begin{split}
&L^0 = (0_{1\times \varrho};\Id_{\varrho}; 0_{(n-\varrho)\times \varrho}),\;
 \theta^0 = x_{[\cl{}\varrho]},\\
&R^0=\left(e_1, (0_{(\varrho+1)\times (n-\varrho)},\Id_{n-\varrho})\right), \text{ and } \mu^0 = (x_0; x_{[\varrho+1:n]}).
  \end{split}
\end{equation}
This initialization starts as if $x_{i}\ge 0$ for $i\in [(\varrho+1)\cl{}]$ are already processed by Algorithm~\ref{alg:ddwithmult} and, so, $x_{(\varrho+1):n}$ are non-negative for each $P^k$. Now, assume $(R^k,\mu^k(1;x))$ is dehomogenized and $L^k$ is empty. Let $g(y) = \left(g_i(y)\right)_{i=0}^n$. Then,
\begin{equation}\label{eq:reformulateACobj}
  \begin{split}
g(y) (1;x) = g(y) R^k\mu^k(1;x) = \sum_{i=1}^{p_k} g(y)R^k_{\cl{},i} \mu^k_i(1;x)\\
 = \sum_{i=1}^{p_k} g\left(\dfrac{y\mu^k_i(1;x)}{\mu^k_i(1;x)}\right)R^k_{\cl{},i} \mu^k_i(1;x),
  \end{split}
\end{equation}
where the last equality uses $\mu^k_i(1;x)\ge 0$ and that homogenization of $g_i(\cdot)$ is zero when the homogenization variable is zero.
\begin{theorem}\label{thm:simpleAChierarchy}
  Let $\left(R^k,\mu^k(1;x)\right)$ be dehomogenized and $J=\left\{i\mid R^k_{1j} = 1\right\}$. Assume that for $j\in\{\cl{}\varrho\}$, $g_j(x)$ is affine and $P\subseteq\left\{x\;\middle|\; x_i\ge 0 \text{ for } i\in \{(\varrho+1)\cl{}\}\right\}$. Let $\bar{k}$ be the smallest index such that $L^{\bar{k}}$ is empty. For $k\in \{\bar{k}\cl{}m\}$, we relax $(\AC)$ as follows:
  \begin{subequations}\label{eq:DAC}
  \begin{alignat}{5}
    &(\ACD^k)&\quad&\min&\quad&\sum_{i=1}^{p_k} g\left(\dfrac{y\mu^k_i(1;x)}{\mu^k_i(1;x)}\right)R^k_{\cl{},i} \mu^k_i(1;x)\label{eq:ACDkobj}\\
    &&&\text{s.t.}&&\eqref{eq:DDkyscale}-\eqref{eq:DDklinprecision}.
  \end{alignat}
\end{subequations}
Further, we relax $(\ACD^k)$ into a convex optimization problem $\underline{\ACD}^k$ by replacing $\mu^k(1;x)$ (resp. $y\mu^k_i(1;x)$) with a variable $\lambda^k\in \Re^{p_k}$ (resp. $y^{k,i}$). We have $\opt_{\underline{\ACD}^{\bar{k}}}\le \ldots\le \opt_{\underline{\ACD}^m} = \opt_{\AC}$. 
\end{theorem}
\begin{proof}
Since \eqref{eq:reformulateACobj} rewrites the objective \eqref{eq:ACDkobj} in the form of \eqref{eq:objDDk}, it follows from Theorem~\ref{thm:simplehierarchy} that $(\ACD)^k$ is a relaxation of $(\AC)$ and $\opt_{\underline{\ACD}^m} = \opt_{\AC}$. To show that $\opt_{\underline{\ACD}^{k}}$ is increasing in $k$ for $k\ge \bar{k}+1$, we construct a solution to $\underline{\ACD}^{k-1}$ with the same objective value as a given solution $(\lambda^k,(y^{k,i})_{i=1}^{p_k})$ to $\underline{\ACD}^k$. Since $P^k\subseteq P^{k-1}$, we write $R^k = R^{k-1}\Gamma$ where $\Gamma \in \Re^{p_{k-1}\times p_k}_+$. Let $\lambda_i^{k-1}$ (resp. $\lambda_i^k$) be the linearization of $\mu_i^{k-1}(1;x)$ (resp. $\mu_i^{k}(1;x)$) and $Y^{k-1}$ (resp. $Y^k$) be a matrix whose columns are $y^{k-1,i}$ (resp. $y^{k,i}$). Then, for any feasible $(Y^k,\lambda_k)$ to $\underline{\ACD}^k$, we define $(Y^{k-1},\lambda^{k-1}) = (Y^{k}\Gamma^\intercal,\Gamma\lambda^k)$. The feasibility of the constructed solution to $\underline{\ACD}^{k-1}$ is easy to verify. Specifically, $Y^{k-1}_{\cl{},i} = Y^k\Gamma^\intercal_{\cl{},i}\in \sum_{j=1}^{p_k}\Gamma_{i,j}\lambda^k_j C = \lambda^{k-1}_i C$. Moreover, 
 \begin{equation}\label{eq:showACdominate}
  \begin{split}
  &\sum_{\ell=1}^{p_{k}} g_j\left(\dfrac{y^{k,\ell}}{\lambda^{k}_\ell}\right)R^{k}_{j,\ell}\lambda^{k}_\ell 
  %= \sum_{i=1}^{p_{k-1}}\sum_{\ell=1}^{p_{k}} g_j\left(\dfrac{y^{k,\ell}}{\lambda^{k}_\ell}\right)R^{k-1}_{j,i} \Gamma_{i\ell}\lambda^{k}_\ell 
  = \sum_{i=1}^{p_{k-1}}\sum_{\ell=1}^{p_{k}} g_j\left(\dfrac{y^{k,\ell}\Gamma_{i\ell}}{\lambda^{k}_\ell\Gamma_{i\ell}}\right)R^{k-1}_{j,i} \Gamma_{i\ell}\lambda^{k}_\ell\\
  &\ge \sum_{i=1}^{p_{k-1}} g_j\left(\dfrac{Y^k\Gamma^\intercal_{\cl{}i}}{\Gamma_{i,\cl{}}\lambda^{k}}\right)R^{k-1}_{j,i} \Gamma_{i,\cl{}}\lambda^{k} = \sum_{i=1}^{p_{k-1}} g_j\left(\dfrac{y^{k-1,i}}{\lambda^{k-1}_i}\right)R^{k-1}_{j,i} \lambda^{k-1}_i,
  \end{split}
 \end{equation}
 where the first equality is by writing $R^k_{\cl{},\ell} = \sum_{i=1}^{p_{k-1}}R^{k-1}_{\cl{},i}\Gamma_{i\ell}$. The first inequality is an equality for $j\in\{\cl{}\varrho\}$ because $g_j$ is affine. For $j\in \{0,\varrho+1,\ldots,n\}$, the inequality follows since the homogenization of $g$ is sublinear, $\Gamma_{i\ell}\ge 0$, and $R^{k-1}_{j,i}\ge 0$ because of the alternate initialization~\eqref{eq:alternateinit}. Then, by summing \eqref{eq:showACdominate} for $j\in[0\cl{}n]$, we have
 $\opt_{\underline{\ACD}^{k}}\ge \opt_{\underline{\ACD}^{k-1}}$. 
 \qed
\end{proof}

\section{The structure of barycentric coordinates}\label{sec:muproperties}

Theorems~\ref{thm:simplehierarchy} and~\ref{thm:simpleAChierarchy} leveraged the numerical aspects of Algorithm~\ref{alg:ddwithmult}, and we will now explore the symbolic calculations it performs. We begin by considering $P=[0,1]^n$ and show that the barycentric coordinates arising from Algorithm~\ref{alg:ddwithmult} match $x^{S,S'}$, where $S\cap S' = \emptyset$ and $S\cup S'= \{\cl{}k\}$.
% except that these factors are limited to those where $S\cup S' = T$, where $T\subseteq [\cl{}n]$, and is determined by the order in which constraints are processed. 

\begin{proposition}\label{prop:DDhypercube}
    Let $P=[0,1]^n$ and Algorithm~\ref{alg:ddwithmult} be initialized so that:
    \begin{equation}\label{eq:nonneginit}
      L^0 = [],\;\theta^0 = [],\; R^0=\Id_{n+1},\; \text{ and } \mu^0 = (x_0;x). 
    \end{equation}
    Let $T=[\cl{}k]$ and process the inequalities $x_i\le 1$ for $i\in T$ during the first $k$ iterations. Then, $R^k = [V^k;\bar{R}^k]$, where $V^k$ has $2^k$ columns, which will be indexed using subsets of $\{\cl{}k\}$ and $\bar{R}^k$ has $n-k$ columns which will be indexed $\left[1\cl{}(n-k)\right]$. These matrices $V^k$ and $\bar{R}^k$ are:
    \begin{equation}\label{eq:Boxkrays}
      V^k = \left(1;\chi_S;0_{n-k}\right)_{S\subseteq [:k]} \text{ and } \bar{R}^k = \left[e_{k+1},\ldots,e_n\right].
    \end{equation}
    Similarly, partition $\mu^{k}(x_0;x)$ as $[\nu^k(x_0;x);\omega^k(x_0;x)]$, where $\nu^k$ (resp. $\omega^k$) contain the first (resp. last) $2^k$ (resp. $n-k$) barycentric coordinates. Then, 
    \begin{equation}\label{eq:Boxkmultipliers}
      \nu^k_S = \dfrac{\prod_{i\in S}x_i\prod_{i\in \{\cl{}k\}\backslash S}(x_0-x_i)}{x_0^{k-1}} \text{ and }\omega^k_i = x_i.
    \end{equation} 
  \end{proposition}
  \long\def\propDDhypercubeproof{
    We prove the result by induction. For $k=0$, $V^k$ consists of a single column $e_0$ and $\nu^0_S = 1/x_0^{-1} = x_0$. Then, it can be checked that \eqref{eq:nonneginit} is a specialization of \eqref{eq:Boxkrays} and \eqref{eq:Boxkmultipliers} with $k=0$. Now, assume that the result is true at iteration $i$. We show that it holds also at iteration $i+1$. To see this, observe that the inequality processed at iteration $i+1$ is $x_0 - x_{i+1}\ge 0$. The only column of $R^i$ that does not satisfy this inequality is $\bar{R}^i_1$ and the columns in $V^i$ satisfy this inequality strictly. Since $\sum_{S\subseteq \{\cl{}i\}}\nu^i_S = x_0$, we set $N_{\text{tot}} = x_0$ at Step~\ref{algstep:defineNtot}. Then, the first $2^{i+1}$ columns of $R^{i+1}$ are $V^{i+1} = [V^i;W^i]$, where $W^i = (1;\chi_S;1;0_{n-i-1})_{S\subseteq \{\cl{}i\}}$ are obtained by adding $\bar{R}^i_1$ to the columns of $V^i$. Combining, $V^{i+1} = (1;\chi^S;0_{n-i-1})_{S\subseteq {\cl{}i+1}}$. Now, as in Step~\ref{algstep:definemu2}, we compute $\nu_S^{i+1}$ where $S\subseteq\{\cl{}(i+1)\}$ as a multiplier for $V^{i+1}_S$. We have
    \begin{equation*}
      \nu_S^{i+1} = \left\{
        \begin{alignedat}{2}
          &\nu_S^{i} - \dfrac{\nu_S^ix_{i+1}}{x_0} = \dfrac{\prod_{j\in S}x_j\prod_{j\in \{\cl{}(i+1)\}\backslash S}(x_0-x_j)}{x_0^{i}} &\quad& i+1\not\in S\\
          &\dfrac{\nu_{S\backslash\{i+1\}}^ix_{i+1}}{x_0} =  \dfrac{\prod_{j\in S}x_j\prod_{j\in \{\cl{}(i+1)\}\backslash S}(x_0-x_j)}{x_0^{i}} &&i+1\in S.
        \end{alignedat}\right.
    \end{equation*}
    The expressions for $\nu^{i+1}_S$ match regardless of whether $i+1\in S$ or not verifying the first formula in \eqref{eq:Boxkmultipliers}. Finally, $\omega^{i+1}_j = \omega^{i}_j = x_{j+i+1}$ for $j\in \left[1\cl{}(n-i-1)\right]$ since the corresponding columns are tight at this constraint.  
  }
  \iftoggle{allproofsinpaper}{%
  \begin{proof}
    \propDDhypercubeproof
    \qed
  \end{proof}
  }{}
  
  The expressions \eqref{eq:Boxkrays} for $(V^k, R^k)$ obtained at the $k^{\text{th}}$ iteration of DD algorithm provide a $V$-description of the set $X = \bigl\{x\mid 0 \le x_i \le 1 \text{ for } i\in \{\cl{}k\}, x_i\ge 0 \text{ for } i > k\bigr\}$. This property holds more generally for the DD-algorithm. When initiated with the $V$-description of a pointed polyhedral cone (\textit{i.e.}, the cone does not contain lineality directions), the algorithm generates at the $k^{\text{th}}$ iteration a $V$-description of the starting cone intersected with the first $k$ constraints (see, for example, Lemma~3.51 in \cite{conforti2014integer}). Our modified version in Algorithm~\ref{alg:ddwithmult} also returns the barycentric coordinates for this polyhedron~(see, for example, \eqref{eq:Boxkmultipliers} and Proposition~\ref{prop:correctmutheta}). For $X$, these coordinates are given by the expressions detailed in \eqref{eq:Boxkmultipliers}, where $x_0$ is set to $1$ to dehomogenize the representation. The following result follows directly from Proposition~\ref{prop:DDhypercube} due to the permutation-invariance of the constraint set with respect to $x_i$, $i\in\{\cl{}n\}$.

  \begin{corollary}\label{cor:allorders}
    Let $P=[0,1]^n$ and Algorithm~\ref{alg:ddwithmult} be initialized as in \eqref{eq:Boxkrays} and executed for $k$ steps for each $T\subseteq \{\cl{}m\}$ such that $|T|=k$. Then, after relabeling rows in $T$ as $1,\ldots,k$, the execution with $T$ returns matrices $V^{T}$ and $\bar{R}^{T}$ as in \eqref{eq:Boxkrays} with barycentric coordinates $\nu_S^T$ and $\omega_i^T$ given as in \eqref{eq:Boxkmultipliers}.
  \end{corollary}

  By Corollary~\ref{cor:allorders} and \eqref{eq:Boxkmultipliers}, 
  we have $\nu_S^T = x^{S,T\backslash S}$ when $x_0=1$. By defining $S' = T\backslash S$, it follows that all product-factors used in \eqref{eq:RLThypercube} (or, $\ACRLT^k$) can be obtained through various executions of the DD procedure. This also demonstrates that the linear dependencies in Proposition~\ref{prop:expansionhelps} are captured when we algebraically expand $\nu_S^T$. These affine dependencies will be explored in a more general setting in Theorem~\ref{thm:expressmulin}.

  %As we detail in Remark~\ref{alg:ddwithmult}, Proposition~\ref{prop:DDhypercube} shows that the basic hierarchy can be tightened by processing the constraints in different orders.
  
  %\begin{remark}\label{rmk:motivateorderings}
  %  Substituting $x_0=1$ in \eqref{eq:Boxkmultipliers}, we obtain product factors used in $\ACRLT^k$. Given a specific constraint order, the product-factors generated in \eqref{eq:Boxkmultipliers} are $x^{S,S'}$, where $S\cap S' =\emptyset$ and $S\cup S' = \{\cl{}k\}$. However, by reordering the constraints $x_i\le 1$, $i\in \{\cl{}n\}$, we can reindex any set $T$, $|T|=k$, as $[\cl{}k]$. Then, Algorithm~\ref{alg:ddwithmult} generates the product factors $x^{S,T\backslash S}$. By altering $T$, we obtain all the product factors used in $\ACRLT^k$. We will similarly use different constraint orders obtaining several barycentric coordinates in parallel. Expanding these functional forms will allow us to exploit linear dependencies between these product-factors just as in Proposition~\ref{prop:expansionhelps}.\qed
  %\end{remark}

  Even though the barycentric coordinates for $[0,1]^n$ were polynomials, this is not the case for general polytopes \cite{warren2003uniqueness,wachspress2016rational}. However, it is easy to see that the barycentric coordinates produced by Algorithm~\ref{alg:ddwithmult} are rational functions.

  \begin{proposition}\label{prop:rationalfunction}
    At iteration $k$, Algorithm~\ref{alg:ddwithmult} generates $\mu^k(x_0;x)$ and $\theta^k(x_0;x)$ that are ratios of homogenous polynomials. 
  \end{proposition}
\long\def\proprationalfunctionproof{
  The proof follows easily by induction. The base case follows directly from Step~\ref{algstep:initialize}. Assume, now that this condition holds after iteration $i$. Moreover, assume that the lineality directions are not orthogonal to $\A_{i+1,\cl{}}$ so that Case 1 of Algorithm~\ref{alg:ddwithmult} is executed. The coordinates $\mu^{i+1}(x_0;x)$ defined at Steps~\ref{algstep:definemu1} and \ref{algstep:definetheta1} are linear functions of $\mu^i(x_0;x)$. The property then follows from the induction hypothesis. Now, assume that the lineality directions are orthogonal to $\A_{i+1,\cl{}}$. Then, at Step~\ref{algstep:defineNtot}, $N_{\text{tot}}$ is computed as a conic combination of elements of $\mu^{i}(x_0;x)$. At Step~\ref{algstep:definemu2}, $\mu^{i+1}(x_0;x)$ is obtained either as $\mu^i_{j'}(x_0;x)\mu^i_{j}(x_0;x)/N_{\text{tot}}$ or as a linear combination of such quantities added to $\mu_j(x_0;x)$. Once again, the result follows from the induction hypothesis.
}
\iftoggle{allproofsinpaper}{%
\begin{proof}
 \proprationalfunctionproof\qed 
\end{proof}
}{}

  We now provide a concrete example. Here, we set $x_0=1$ to dehomogenize the coordinates.
  \long\def\showallmuexpr{%
  \begingroup
  \allowdisplaybreaks
  \renewcommand*{\arraystretch}{1.3}
  \begin{align*}
  %  &\mu_1 = \dfrac{\left(3-x_{1}-x_{2}-x_{3}\right) \left(2-2 x_{1}-x_{2}\right) \left(2-x_{1}-4 x_{2}\right)}{2 \left(2-x_{1}-x_{2}\right) \left(3-x_{1}-x_{2}\right)} \\
    &\mu_2 = \dfrac{2 \left(3-x_{1}-x_{2}-x_{3}\right) x_{2} \left(2-2 x_{1}-x_{2}\right)}{\left(2-x_{1}-x_{2}\right) \left(3-x_{1}-x_{2}\right)}
    \\
    &\mu_3 = \dfrac{\left(3-x_{1}-x_{2}-x_{3}\right) x_{1} \left(2-x_{1}-4 x_{2}\right)}{\left(2-x_{1}-x_{2}\right) \left(3-x_{1}-x_{2}\right)} 
    \\
    &\mu_4 = \dfrac{7 \left(3-x_{1}-x_{2}-x_{3}\right) x_{2} x_{1}}{2 \left(2-x_{1}-x_{2}\right) \left(3-x_{1}-x_{2}\right)} 
    \\
    &\mu_5 = \dfrac{x_{3} \left(2-2 x_{1}-x_{2}\right) \left(2-x_{1}-4 x_{2}\right)}{2 \left(2-x_{1}-x_{2}\right) \left(3-x_{1}-x_{2}\right)} 
    \\
    &\mu_6 = \dfrac{2 x_{3} x_{2} \left(2-2 x_{1}-x_{2}\right)}{\left(2-x_{1}-x_{2}\right) \left(3-x_{1}-x_{2}\right)} 
    \\
    &\mu_7 = \dfrac{x_{3} x_{1} \left(2-x_{1}-4 x_{2}\right)}{\left(2-x_{1}-x_{2}\right) \left(3-x_{1}-x_{2}\right)} 
    \\
    &\mu_8 = \dfrac{7 x_{3} x_{2} x_{1}}{2 \left(2-x_{1}-x_{2}\right) \left(3-x_{1}-x_{2}\right)}
\end{align*}
\endgroup}
\long\def\showonlyfirstmu{\begin{equation*}
\mu_1 = \dfrac{\left(3-x_{1}-x_{2}-x_{3}\right) \left(2-2 x_{1}-x_{2}\right) \left(2-x_{1}-4 x_{2}\right)}{2 \left(2-x_{1}-x_{2}\right) \left(3-x_{1}-x_{2}\right)},
\end{equation*}
where the remaining barycentric coordinates are detailed in Appendix~\ref{app:showremainingmu}.}
  \begin{example}\label{ex:nontrivialbarycentric}
    Consider the polytope for which we provide the $H$-description and $V$-description below:
    \begin{equation*}
    P = \left\{x\in\Re^3\;\middle|\; 
    \begin{aligned}
      &2 - x_1 -4x_2\ge 0, \\
      &2 - 2x_1 - x_2 \ge 0,\\
      &3 - x_1 - x_2 - x_3\ge 0\\
      &x\ge 0
    \end{aligned}\right\}\qquad
    \begingroup
    \renewcommand*{\arraystretch}{1.5}
    R=\begin{pmatrix}
      1 & 1 & 1 & 1 & 1 & 1 & 1 & 1 
      \\
       0 & 0 & 1 & \frac{6}{7} & 0 & 0 & 1 & \frac{6}{7} 
      \\
       0 & \frac{1}{2} & 0 & \frac{2}{7} & 0 & \frac{1}{2} & 0 & \frac{2}{7} 
      \\
       0 & 0 & 0 & 0 & 3 & \frac{5}{2} & 2 & \frac{13}{7} 
    \end{pmatrix}
    \endgroup
    \end{equation*}
    Using Algorithm~\ref{alg:ddwithmult}, we obtain 
    \iftoggle{allproofsinpaper}{\showallmuexpr}{\showonlyfirstmu}
    It can be verified that each $\mu_i$ is non-negative over the polytope and $\sum_{i=1}^8\mu_i = 1$. As expected, $\mu_j$ evaluates to $1$ at $R_{\cl{},j}$ and to $0$ at $R_{\cl{},i}$ for $i\ne j$. The polytope $P$ in this example is a simple polytope, \textit{i.e.}, an $n$-dimensional polytope, with $n=3$, where each vertex is adjacent to $n$ edges. For such polytopes, an explicit formula for barycentric coordinates is available~\cite{warren2007barycentric}. In our computation, we will denote the determinant of matrix $D$ as $|D|$. For the vertex indexed $1$, the set of tight inequalities is $\Id_{3}x\ge 0$. Since $\left|\Id_{3}\right| = 1$, and the inequalities which have a slack at this vertex are $2-x_1-4x_2\ge 0$, $2-2x_1-x_2\ge 0$, and $3-x_1-x_2-x_3\ge 0$, we obtain:
    \begin{equation*}
      \omega_1 = |\Id_{3}|\times (2-x_1-4x_2)(2-2x_1-x_2)(3-x_1-x_2-x3).
    \end{equation*}
    Similarly, we compute $\omega_2,\ldots,\omega_8$ and obtain $\sum_{i=1}^8 \omega_i = 2 \bigl(x_{1}+x_{2}-2\bigr) \bigl(x_{1}+x_{2}-3\bigr)$. Then, the barycentric coordinate for $(0,0,0)$ is $\omega_1/\sum_{i=1}^8 \omega_i$ \cite{warren2007barycentric} and matches the expression found using Algorithm~\ref{alg:ddwithmult}. \qed
  \end{example}

  We used Algorithm~\ref{alg:ddwithmult} to derive barycentric coordinates for specially structured polyhedral sets in Proposition~\ref{prop:DDhypercube} and Example~\ref{ex:nontrivialbarycentric}. That the algorithm produces such coordinates for all polyhedral cones will follow from the description in Appendix~\ref{app:ddalgorithm} and the formulae in Algorithm~\ref{alg:ddwithmult}. Next, we will verify this fact, albeit focusing on a transformation that reveals new inequalities for the relaxation hierarchy. These equalities generalize \eqref{eq:affinerelation} when $T\backslash S$ is a singleton to the setting of arbitrary polyhedra. Specifically, we will show that, as long as $N^+\ne \emptyset$ during the $k+1^{\text{st}}$ iteration, $\left(\theta^{k}(x_0;x);\mu^{k}(x_0;x)\right)$ is affinely related to $\left(\theta^{k+1}(x_0;x);\mu^{k+1}(x_0;x)\right)$ and the inequalities $\mu^{k+1}(x_0;x)\ge 0$ imply the inequalities $\mu^{k}(x_0;x) \ge 0$. We first introduce some notation and then use it to describe this affine transformation. Assume $L^k_{\cl{},\xi}$ is not orthogonal to $\A_{k+1,\cl{}}$, and, therefore, during iteration $k+1$, $q_k\ne 0$ and $\alpha\ne 0$. Then, we let:
\begin{equation}\label{eq:firstFGD}
\begin{aligned}
  &\NiceMatrixOptions{cell-space-limits = 2pt}\setlength\arraycolsep{0.5em}
  F^{k+1} = 
  \begin{pNiceMatrix}
    \Id_{(\xi-1)} & 0_{(\xi-1)\times (q_k-\xi)}\\
    0 & -\sign(\alpha_\xi)\alpha_{[(\xi+1)\cl{}q_k]}\\
    0_{(q_k-\xi)\times(\xi-1)} & |\alpha_\xi| \Id_{(q_k-\xi)}
  \end{pNiceMatrix}, 
  G^{k+1} = \begin{pNiceMatrix}
    0 & 0_{(\xi-1)\times (p_k)}\\
    1 & -\sign(\alpha_\xi)\beta\\
    0 & 0_{(q_k-\xi)\times(p_k)}  
  \end{pNiceMatrix}\\
  &D^{k+1} = 
  \begin{pmatrix}
    0_{1\times p_k} & |\alpha_\xi| \Id_{p_{k}}
  \end{pmatrix}.
\end{aligned}
\end{equation}
On the other hand, if  $L^k$ is empty or orthogonal to $\A_{k+1,\cl{}}$, we have $q_k=0$ or $\alpha=0$. Let $\beta$ be as computed in Step~\ref{algstep:definealphabeta}, and $N^0$, $N^-$, and $N^+$ classify columns of $R^k$ as at Step~\ref{algstep:defineNsets} during the $k+1^{\text{st}}$ iteration of Algorithm~\ref{alg:ddwithmult}. Then, we write
\begin{equation}\label{eq:secondFGD}
\NiceMatrixOptions{cell-space-limits = 2pt}
\begin{aligned}
  &F^{K+1} = \Id_{q_k}, G^{k+1}=0_{p_k\times \left(|N^0|+(|N^+|+1)\times|N^-|\right)}\\
  &\NiceMatrixOptions{cell-space-limits = 2pt}
  D^{k+1} = 
  \begin{pNiceArray}{!{\quad}c!{\quad}|c|ccc}
    \Id & \phantom{\quad}0\phantom{\quad} & 0 & 0 & 0\\\hline
    \Block{3-1}<\LARGE>{0} & \Block{3-1}<\LARGE>{\Id} & -\beta_{|N^0|+|N^+|+1\cl{}} & 0 & 0\\
    &&0 & \Ddots & 0\\
    &&0 & 0 & -\beta_{|N^0|+|N^+|+1\cl{}} \\\hline
    0&0& \beta_{|N^0|+1}\Id & \Cdots & \beta_{|N^0|+|N^+|}\Id
  \end{pNiceArray},
\end{aligned}
\end{equation}
where we assume for notational convenience that entries of $\beta$ and, correspondingly columns of $R^k$, are permuted such that indices are ordered as follows: first are those that belong to $|N^0|$; second, those in $|N^+|$; and finally, those in $|N^-|$. Correspondingly, for $D_{K+1}$, vertical lines separate column blocks of size $|N^0|$, $|N^+|$, and $|N^+|\times |N^-|$ respectively and horizontal lines separate row blocks of size $|N^0|$, $|N^+|$, and $|N^-|$ respectively.
\begin{theorem}\label{thm:expressmulin}
  Assume that at iteration $k+1$ of Algorithm~\ref{alg:ddwithmult}, either $q_k\ne 0$ and $\alpha\ne 0$ or $N^+\ne \emptyset$. Then, the symbolic expressions for the barycentric coordinates satisfy:
  \begin{equation}\label{eq:multipliersk+1tomultipliersk}
    \begin{pmatrix}\theta^k(x_0;x)\\\mu^k(x_0;x)\end{pmatrix} = (P^{k+1}_{\pi})^\intercal
    \begin{pmatrix} F^{k+1} & G^{k+1}\\ 0 & D^{k+1}\end{pmatrix} \begin{pmatrix} \theta^{k+1}(x_0;x)\\\mu^{k+1}(x_0;x)\end{pmatrix},
  \end{equation}
  where $P^{k+1}_\pi\in \Re^{(q_k+p_k)\times (q_k+p_k)}$ is a permutation matrix that does not resequence $\theta^k(x_0;x)$ but permutes the entries of $\mu^k(x_0;x)$ so that, at iteration $k+1$, the indices in $N^0$ are followed by those in $N^+$ and, finally, by $N^-$. If $N^+=\emptyset$ and either $q_k=0$ or $\alpha=0$, \eqref{eq:multipliersk+1tomultipliersk} still holds numerically. Specifically, using $P^{k+1}_\pi$, the expressions for $\mu^k_j(x_0;x)$, where $j\in N^-$, are dropped at iteration $k+1$, thereby enforcing that they evaluate to zero for all $(x_0;x)\in K^{k+1}$. Also, $F^{k+1}:\Re^{q_k\times q_{k+1}}$, $G^{k+1}:\Re^{q_k\times p_{k+1}}$, and $D^{k+1}:\Re^{p_k\times p_{k+1}}$ are constant matrices such that $D^{k+1}\ge 0$. Regardless of whether \eqref{eq:multipliersk+1tomultipliersk} holds symbolically,
  \begin{equation}\label{eq:LkRktoLk+1Rk+1}
    \begin{pmatrix}L^{k+1} & R^{k+1}\end{pmatrix} = \begin{pmatrix}L^k & R^k\end{pmatrix}(P^{k+1}_\pi)^\intercal \begin{pmatrix}F^{k+1} & G^{k+1}\\ 0 & D^{k+1}\end{pmatrix}.
  \end{equation}
\end{theorem}
\begin{proof}
  Let $\alpha$ and $\beta$ be as computed during iteration $k+1$ of Algorithm~\ref{alg:ddwithmult}.  First, assume that at iteration $k+1$ of Step~\ref{algstep:ddforloop}, $q_{k}> 0$ and $\alpha \ne 0$. We will temporarily index the columns of $\theta^{k+1}$ using the indices of the columns of $\theta^k$ from which they are  computed, \textit{i.e.},\/ $\{1,\ldots,\xi-1,\xi+1,\ldots,p_k\}$ and those of $\mu^{k+1}$ will be indexed $\{0,1,\ldots,p_{k}\}$, where index $0$ corresponds to the new ray $L^k_{\cl{},\xi}$. We will also hide the function arguments for $\theta$ and $\mu$ from either iteration. For $i< \xi$, $\theta^{k}_i = \theta^{k+1}_i$ and, for $i > \xi$, $\theta^k_i = \alpha_{\xi}\theta^{k+1}_i$. Step~\ref{algstep:negateL} ensures $\alpha_{\xi} > 0$ by switching the sign of $L^k_{\cl{}\xi}$, if required. Then, we have $\theta^k_{\xi} = \mu^{k+1}_0 - \sum_{i > \xi}\alpha_j\theta^{k+1}_{i} - \sum_{j=1}^{p_k} \beta_j\mu^{k+1}_j$. Finally, for $j\in\{\cl{}p_k\}$, $\mu^k_j = \alpha_{\xi}\mu^{k+1}_j$. These calculations are summarized in \eqref{eq:firstFGD} and \eqref{eq:multipliersk+1tomultipliersk}. It is easily checked that $L^{k+1}$ and $R^{k+1}$ computed in Steps~\ref{algstep:defineR1} and \ref{algstep:defineL1} match the calculations in \eqref{eq:firstFGD} and \eqref{eq:LkRktoLk+1Rk+1}.
  
  Now, assume that either $q_k=0$ or $\alpha = 0$. In this case, if $j\in N^0\cup N^+$, this column is retained in $R^{k+1}$ and will be temporarily indexed as $j$, while the newly created columns in $N^+\times N^-$ will be indexed $(i,j)$ for $i\in N^+$ and $j\in N^-$.
  For $j\in N^0$, $\mu^{k+1}_j = \mu^{k}_j$. Now, consider $j\in N^-$. Then, if $N^+\ne\emptyset$,
   \begin{equation*}\label{eq:mukjN-}
    0\le \sum_{i\in N^+}\beta_i\mu^{k+1}_{(i,j)} = \sum_{i\in N^+}\beta_i \frac{\mu^k_i\mu^k_j}{\sum_{i'\in N^+}\mu^k_{i'}\beta_{i'}} = \mu^k_j.
   \end{equation*}
   On the other hand, if $N^+ = \emptyset$, the columns in $\mu^k_j$, $j\in N^-$ are dropped. However, in this case, for $(x_0;x)\in K^{k+1}$, the following holds numerically: $0 \le \A_{k+1,\cl{}}R^k\mu^k_j(x_0;x) = \beta \mu^k(x_0;x)$. Since each term in the last calculation is non-positive, $\mu^k_j(x_0;x)$ must be zero if $j\in N^-$.
   Now, consider $i\in N^+$. Then,
   \begin{equation}\label{eq:muNplus}
   \begin{split}
    &0\le \mu^{k+1}_i - \sum_{j\in N^-}\beta_j\mu^{k+1}_{(i,j)} \\
    &\quad = \mu^k_i + \sum_{j\in N^-}\frac{\beta_j\mu^k_i\mu^k_j}{\sum_{i'\in N^+}\mu^k_{i'}\beta_{i'}} - \sum_{j\in N^-}\frac{\beta_j\mu^k_i\mu^k_j}{\sum_{i'\in N^+}\mu^k_{i'}\beta_{i'}} 
    = \mu^k_i.
   \end{split}
  \end{equation}
  These calculations are shown in \eqref{eq:secondFGD} and \eqref{eq:multipliersk+1tomultipliersk}. It can be easily seen that the calculations in Step~\ref{algstep:defineR2} are the same as those in \eqref{eq:secondFGD} and \eqref{eq:LkRktoLk+1Rk+1}.
   \qed
\end{proof}

Theorem~\ref{thm:expressmulin} plays an important role in our proposed hierarchy since it is typically the case that the expressions for $\left(\mu^k(x_0;x);\theta^k(x_0;x)\right)$ involve polynomials of smaller degree than those in $\left(\mu^{k+1}(x_0;x);\theta^{k+1}(x_0;x)\right)$. In other words, this result allows us to relate higher-degree polynomials terms with those having smaller degrees. We will show that a similar property does not hold for RLT in the polyhedral case. 

\begin{example}\label{ex:affineDDnotRLT}
Consider the polyhedron described by
\begin{equation*} 
P = \left\{(x_1,x_2)\;\middle|\; 
\begin{aligned}
3x_1-x_2\ge 0,\\
-x_1 + 4x_2\ge 0,\\ 
1 + 10x_1-10x_2\ge 0,\\
1 + x_1 - 3x_2\ge 0
\end{aligned}\right\}.
\end{equation*}
The DD algorithm at iteration 4 (resp. iteration 3) yields barycentric coordinates for $P$ (resp. polyhedron defined by its first three inequalities). The coordinates at iteration 4 (resp. iteration 3) are shown on the right-hand-side (resp. left-hand-side) of the following instantiation of \eqref{eq:multipliersk+1tomultipliersk} with $k=3$ for $P$,
\begin{equation}\label{eq:affinerel}
  \begin{aligned}
&\frac{1}{33}\begin{bmatrix}
33(1 -10 x_{2}+10 x_{1})
\\
3(4 x_{2} - x_{1})
\\
3x_{1}-x_{2} 
\end{bmatrix} \\
& {} = \frac{1}{d(x)}\begin{bmatrix}
1 & 0 & 87 & 0 
\\
 0 & 1 & 0 & 87 
\\
 0 & 0 & 1 & 12 
\end{bmatrix}
\begin{bmatrix}
33 \left(1 -10 x_{2}+10 x_{1}\right) \left(1 -3 x_{2}+x_{1}\right)
\\
3\left(4 x_{2}-x_{1}\right) \left(1 -3 x_{2}+x_{1}\right) 
\\
\left(3 x_{1}-x_{2}\right) \left(1 -10 x_{2}+10 x_{1}\right)
\\
33(3x_1-x_2)(4x_2-x1)
\end{bmatrix}.
\end{aligned}
\end{equation}
Here, $d(x) = 363 + 3243x_1 - 2046x_2 = 33\bigl(11(1+10x_1-10x_2) + 12(4x_2-x_1)\bigr)$. The equality~\eqref{eq:affinerel} exemplifies that $\mu^4\ge 0$ implies $\mu^3\ge 0$, since the constant matrix on the right-hand-side has non-negative entries. This in turn implies that the first three constraints defining $P$ are implied by $\mu^4\ge 0$ and the affine relation \eqref{eq:affinerel}. In contrast, first-level RLT relaxation of $P$ does not imply these inequalities. To see, write $P=\{x\mid M(1;x)\ge 0\}$ where
\begin{equation*} 
M = \begin{bmatrix}
0 & 3 & -1 
\\
 0 & -1 & 4 
\\
 1 & 10 & -10 
\\
 1 & 1 & -3 
\end{bmatrix}, \text{ and let }
X = \begin{bmatrix}
1 & x_1 & x_2\\
x_1 & x_1^2 & x_1x_2\\
x_2 & x_1x_2 & x_2^2
\end{bmatrix}.
\end{equation*}
The RLT inequalities are derived from $MXM^\intercal \ge 0$ after linearizing $x_1^2$, $x_1x_2$, and $x_2^2$ as $x_{11}$, $x_{12}$, and $x_{22}$ respectively. Then, the projection of this polyhedron to the space of $(x_1,x_2)$ variables is $\Re^2$. This demonstrates that RLT relaxation does not capture any of the structure of $P$. Specifically, the point $(x_1,x_2) = (-1,-1)$, although infeasible to $P$, can be extended by defining $(x_{11},x_{12},x_{22}) = (40,13,4)$ to satisfy all RLT inequalities. \qed
\end{example}

\begin{remark}\label{rmk:other_ineq}
Observe that if $N^+=\emptyset$ at the $k^{\text{th}}$ iteration, we drop the columns in $N^-$, or equivalently, we impose $\mu^{k-1}_j = 0$ for all $j\in N^{-}$. Then, these constraints, the affine relations in Theorem~\ref{thm:expressmulin}, and $\mu^{k}\ge 0$ imply $\A_{k\cl{}}R^{k-1}\mu^{k-1} \ge 0$. To see this, when $\mu^{k-1}_i > 0$ for some $i\in N^+$, observe that $N_{\text{tot}} > 0$ and
\begin{equation*}
0\le \mu^k_i = \mu^{k-1}_i + \sum_{j\in N^-}\mu_j^{k-1}\mu_i^{k-1}\beta_j = \mu^{k-1}_i \frac{\A_{k\cl{}}R^{k-1} \mu^{k-1}}{N_{\text{tot}}},
\end{equation*}
which implies that $\A_{k\cl{}}R^{k-1} \mu^{k-1} \ge 0$. If $\mu^{k-1}_i =0$ for all $i\in N^+$ and $N^+\ne \emptyset$, then \eqref{eq:mukjN-} implies that $\mu^{k-1}_j = 0$. If $N^+= \emptyset$, we require that $\mu^{k-1}_j = 0$. Therefore, in both cases, by definition of $N^0$, we have $\A_{k\cl{}}R^{k-1} \mu^{k-1} = \sum_{j\in N^0} \A_{k\cl{}}R^{k-1}_{\cl{}j} \mu^{k-1}_j = 0$. \qed
\end{remark}

The inequalities $\A_{t\cl{}}R^{k-1}\mu^{k-1} \ge 0$ for $t > k$ are clearly valid, but not implied since Algorithm~\ref{alg:ddwithmult} has not utilized inequalities indexed $k+1,\ldots,m$ until the $k^{\text{th}}$ iteration. 

\begin{proposition}\label{prop:correctmutheta}
    At iteration $k$,  $R^k$, $\mu^k(x_0;x)$, $L^k$, and $\theta^k(x_0;x)$ satisfy \eqref{eq:pointrep} for every $(x_0;x)\in K^k$. 
\end{proposition}
\long\def\propcorrectmuthetaproof{
  We use induction to show \eqref{eq:pointrep} holds. The base case follows by direct verification. Assume that \eqref{eq:pointrep} holds after iteration $i$. If $N^+\ne\emptyset$, by Theorem~\ref{thm:expressmulin},
  \begin{equation*}
      \begin{split}
      &(x_0;x) = \begin{pmatrix}L^i & R^i\end{pmatrix} \begin{pmatrix} \theta^i(x_0;x)\\ \mu^i(x_0;x)\end{pmatrix} \\
      &= \begin{pmatrix}L^i & R^i\end{pmatrix} (P^i_\pi)^\intercal \begin{pmatrix}F^{i+1} & G^{i+1}\\ 0 & D^{i+1}\end{pmatrix}\begin{pmatrix} \theta^{i+1}(x_0;x)\\\mu^{i+1}(x_0;x)\end{pmatrix} = \begin{pmatrix}L^{i+1} & R^{i+1}\end{pmatrix} \begin{pmatrix} \theta^{i+1}(x_0;x)\\\mu^{i+1}(x_0;x)\end{pmatrix},
      \end{split}
  \end{equation*}
  where the first equality is by induction hypothesis, the second is by \eqref{eq:multipliersk+1tomultipliersk} and the third is by \eqref{eq:LkRktoLk+1Rk+1}. If $N^+ = \emptyset$, the argument is similar once we observe that, for $(x_0;x)\in K^{i+1}$, we have $\mu^i_j(x_0;x) = 0$ for $j\in N^-$. This allows us to restrict attention to columns of $R^i$ with indices in $N^0$. 
  
  The non-negativity of $\mu^{i+1}(x_0;x)$ for $(x_0;x) \in K^{i+1}$ follows from Appendix~\ref{app:ddalgorithm}. We summarize the main ideas. If $L^i$ is not orthogonal to $\A_{i+1,\cl{}}$, all $\mu^{i+1}$ except $\mu^{i+1}_1$ are positive scalar multiples of $\mu^i$ elements and are, therefore, non-negative. The non-negativity of  $\mu^{i+1}_1$ follows since it equals $1/|\alpha_{\xi}| \A_{i+1,\cl{}}(x_0;x)$. If $L^i$ is orthogonal to $\A_{i+1,\cl{}}$ then it follows that (i) for $(j,j')\in N^+\times N^-$, $\mu^{i+1}_{(j,j')}(x_0;x) = \mu^{i}_j(x_0;x)\mu^{i}_{j'}(x_0;x)/N_{\text{tot}}(x_0;x)$ is non-negative since $N_{\text{tot}}(x_0;x), \mu^i(x_0;x)\ge 0$, and (ii) for $j\in N^+$, \begin{equation*}
    \begin{split}
    \mu^{i+1}_j(x_0;x) = \mu^{i}_j(x_0;x) + \dfrac{\sum_{j'\in N^-} \mu_{j'}^{i}(x_0;x)\mu_{j}^{i}(x_0;x)\beta_j}{N_{\text{tot}}(x_0;x)}\\ 
    = \mu^{i}_j(x_0;x)\dfrac{\A_{i+1,\cl{}}(x_0;x)}{N_{\text{tot}}(x_0;x)}\ge 0.
    \end{split}
  \end{equation*}
  Therefore, combining the cases, $\mu^{i+1}(x_0;x)\ge 0$. 
}
\iftoggle{allproofsinpaper}{%
\begin{proof}
  \propcorrectmuthetaproof\qed
\end{proof}}{}

We claim that Algorithm~\ref{alg:ddwithmult} expresses $P^k$ as the Minkowski sum of a polytope, a pointed cone, and a lineality space. The pointed cone is formed by those columns of $R^k$ whose first entry is $0$ and the vertices of the polytope are obtained by dehomogenizing the columns of $R^k$ with positive first entry. However, this requires justification, since rays of $K^k$ can hide lineality directions, for example, if $R^k$ contains two opposing rays. The next result shows that rays in $R^k$ produced by Algorithm~\ref{alg:ddwithmult} do not contain lineality directions. 

\begin{lemma}\label{lem:RpointedWithL}
    There does not exist $(\omega,\nu )\ne 0$ such that $\omega\in \Re^{q_k}$, $\nu \in \Re^{p_k}_+$, and $R^k\nu = L^k\omega$. 
    Equivalently, lineality directions of $K^k$ are not non-negative combinations of columns of $R^k$ and columns of $L^k$ are linearly independent.
  \end{lemma}
  \long\def\lemRpointedWithLproof{
    First, we show that the first and second statement are equivalent. Let $K'$ be the cone generated by non-negative combinations of columns of $R^k$. If $K'$ contains a lineality direction of $K^k$, there exists a $\gamma \in K'$, $\gamma\ne 0$, such that $-\gamma\in K^k$. However, by \eqref{eq:pointrep}, there is a $\nu'\in \Re^{p_k}_+$ and $\omega'\in \Re^{q_k}$ such that $-\gamma = R^k \nu' +  L^k\omega'$ and $\bar{\nu}\in \Re^{p_k}_+$, $\bar{\nu}\ne 0$, such that $\gamma = R^k\bar{\nu}$. This implies that $R^k(\nu'+\bar{\nu}) + L^k\omega' = 0$, contradicting the first statement. Similarly, if columns of $L^k$ are not linearly independent, there exists an $\omega\ne 0$ such that $L^k\omega = 0$, again contradicting the first statement. Therefore, the second statement follows from the first. On the other hand, assume that the first statement is not true and there exists $(\omega,\nu)\ne 0$ such that $R^k\nu = L^k\omega$. If $\nu\ne 0$, $\gamma  = R^k\nu\in K'$ and $-\gamma = -L^k\omega \in K^k$, contradicting the second statement. If $\nu = 0$, $L^k\omega = 0$ for some $\omega\ne 0$, and the linear independence of columns of $L^k$ is contradicted.

    We prove the first statement via induction. The result holds at initialization, since the columns of $R^0$ and $L^0$ are linearly independent. Now, assume that the result holds at iteration $i$ but not after iteration $i+1$. Let $(\omega,\nu)\ne 0$ contradict the first statement so that $(L^{i+1}, R^{i+1})(-\omega;\nu) = 0$. By \eqref{eq:LkRktoLk+1Rk+1}, this would contradict the induction hypothesis unless:
    \begin{equation}\label{eq:nuomega}
        \begin{pmatrix}F^{k+1} & G^{k+1}\\ 0 & D^{k+1}\end{pmatrix} \begin{pmatrix}-\omega\\\nu\end{pmatrix} = 0.
    \end{equation}
    If $L^k$ is not orthogonal to $\A_{i+1,\cl{}}$, \eqref{eq:nuomega} and \eqref{eq:firstFGD} imply that $\nu_{2\cl{}} = 0$ and $\omega=0$. Moreover, $0 = -F^{K+1}_{\xi,\cl{}}\omega + G^{k+1}_{\xi,\cl{}}\nu = \nu_1$, contradicting $(\omega,\nu)\ne 0$. On the other hand, if $L^k$ is orthogonal to $\A_{i+1,\cl{}}$, it follows from \eqref{eq:secondFGD} and \eqref{eq:nuomega} that $\omega = 0$ and $\nu_{N^0} = 0$. Moreover, for each $i\in N^+$, we have $\nu_{i} - \sum_{j\in N^-}\beta_j\nu_{(i,j)} = 0$, where the left-hand-side is a sum of non-negative entries, each of which must, therefore, be zero. In other words, $\nu=0$, which contradicts $(\omega,\nu)\ne 0$.
  }
  \iftoggle{allproofsinpaper}{%
  \begin{proof}
   \lemRpointedWithLproof\qed
  \end{proof}}{}

  A difficulty with rational functions as barycentric coordinates is that the denominators become zero over certain faces of $P$ \cite[Theorem 3]{warren2003uniqueness}. Nevertheless, we show that all the functions are well-defined in the interior of $K^k$. 

  \begin{proposition}\label{prop:positivemu}
    Consider a face $F$ of $K^k$ and a point $(\bar{x}_0;\bar{x})\in \ri(F)$. Then, if $R^k_{j,\cl{}}\in F$, the rational expression for $\mu^k_j$ evaluated at $(\bar{x}_0;\bar{x})$ is strictly positive.
  \end{proposition}
  \begin{proof}
    At the start of the algorithm, we have a single ray $R^0_{:,1}=e_0$. This ray belongs to $F=K^0 = \{x\mid x_0\ge 0\}$. If $(\bar{x}_0;\bar{x})\in \ri(F)$, $\mu^0_1 = \bar{x}_0 > 0$. Now, assume that the property holds after iteration $i$ of the algorithm. Consider a face $F$ of $K^{i+1}$ and a point $(\bar{x}_0;\bar{x})\in \ri(F)$. Since $(\bar{x}_0;\bar{x})\in K^i$, let $F'$ be the face of $K^i$ such that $(\bar{x}_0;\bar{x})\in \ri(F')$. By \eqref{eq:pointrep} and the induction hypothesis, let $(\bar{x}_0;\bar{x}) = R^i\mu^i + L^i\theta^i$ such that $\mu^i_j > 0$ for each column $R^i_{\cl{},j} \in F'$ and let $J=\{j\mid R^i_{\cl{},j} \in F'\}$. Construct a supporting hyperplane $a^\intercal (x_0;x) = 0$, exploiting that $F'$ is a face of $K^i$, such that $a^\intercal R^i_{\cl{},j} = 0$ but $a^\intercal (x_0;x) > 0$ for $(x_0;x) \in K^i\backslash F'$. We will show that the rays that belong to $F$ have a positive multiplier for $(x_0;x)$. To this end, we will consider two cases depending on whether $L^i$ is orthogonal to $A_{i+1,\cl{}}$ or not.
    
    \begin{casesp}
      \item Assume that $L^i$ is nonempty and $\A_{i+1,\cl{}}L^i \ne 0$. Then, $R^{i+1}$ and $L^{i+1}$ are defined at Steps~\ref{algstep:defineR1} and \ref{algstep:defineL1} of Algorithm~\ref{alg:ddwithmult} respectively.  Since $F'$ shares the lineality space of $K^i$, it follows that $a^\intercal L^i_{\cl{},\xi} = 0$. The rays in $R^{i+1}$ consist of $L^i_{\cl{},\xi}$ and those obtained by adding multiples of $L^i_{\cl{},\xi}$ to rays in $R^i$. We will analyze these two types of rays separately. 
      \begin{casesp} 
        \item A ray $r\in F$ is derived from $R^i_{\cl{},j}$ by adding a multiple of $L^i_{\cl{},\xi}$. Then, $R^i_{\cl{},j} \in F'$ since $r\in F'$ and $a^\intercal L^i_{\cl{},\xi} = 0$. The positivity of the multiplier associated with $r$ follows from the induction hypothesis since the multiplier of $R^i_{\cl{},j}$ was positive and Step~\ref{algstep:definemu1} scales it by a positive constant $1/|\alpha_\xi|$. 
        \item Now, consider $L^i_{\cl{},\xi}$. If $L^i_{\cl{},\xi}\in F$ we have, after negating $L^i_{\cl{},\xi}$ if necessary, that $\A_{i+1,\cl{}}L^i_{\cl{},\xi} > 0$. Since $L^i_{\cl{},\xi} \in F$, $(x_0;x) - \epsilon L^i_{\cl{},\xi}\in F$ for sufficiently small $\epsilon > 0$, which shows that $\A_{i+1,\cl{}}(\bar{x}_0,\bar{x}) > 0$. However, by \eqref{eq:pointrep}, this shows that the multiplier associated with $L^i_{\cl{},\xi}$ must be positive since all other columns of $R^{i+1}$ and those of $L^{i+1}$ are orthogonal to $\A_{i+1,\cl{}}$.
      \end{casesp}
    
     \item Assume that either $L^i$ is empty or $\A_{i+1,\cl{}}L^i = 0$. Then, by the induction hypothesis, $\mu^i_j > 0$ for $j\in J$ and, by \eqref{eq:pointrep}, $\mu^i_j = 0$ for $j\in \{\cl{}p_k\}\backslash J$. Observe that all rays derived at Step~\ref{algstep:defineR2} are positive combinations of rays in $R^i$. We argue that it suffices to consider rays with indices in $J$. Instead, consider a ray $r$ obtained by taking a positive combination of a ray $R^i_{\cl{},j}$ for some $j\not\in J$. Then, $a^\intercal r > 0$, which implies $r\not\in F'$ and, therefore, $r\not\in F$. So, we limit attention to $J$ and consider the three types of rays created at Step~\ref{algstep:defineR2}. 
     \begin{casesp}
     \item
    Consider a ray with index $j$ in $J\cap N^0$ that belongs to $F$. This ray remains unaltered and inherits its positive multiplier $\mu^i_j$ from the previous iteration. 
     \item Consider a ray $r$ that is a conic combination of $R^i_{\cl{},j}$ and $R^i_{\cl{},j'}$ where $j\in J\cap N^+$ and $j'\in J\cap N^-$. Then, by the induction hypothesis $\mu^i_j > 0$ and $\mu^i_{j'} > 0$ at $(\bar{x}_0;\bar{x})$. Moreover, since $J\cap N^+\ne \emptyset$, we have $N_{\text{tot}} > 0$ at $(\bar{x}_0,\bar{x})$. It follows that the multiplier for $r$, given by $\frac{\mu^i_{j}\mu^i_{j'}}{N_{\text{tot}}}$, is positive.
     \item 
    If $J\cap N^+ = \emptyset$, none of the other rays belong to $F$. On the other hand, if $J\cap N^+$ is non-empty, by the inductive hypothesis and $\beta_j > 0$ for $j\in N^+$, we have that $N_{\text{tot}}$ is positive over $F$. Now, if $\A_{i+1,\cl{}}(\bar{x}_0;\bar{x}) = 0$, we have $\A_{i+1,\cl{}}(x_0;x) = 0$ for all $(x_0;x)\in F$, in which case the ray with index in $J\cap N^+$ does not belong to $F$ and does not need to be considered. If $\A_{i+1,\cl{}}(\bar{x}_0;\bar{x}) > 0$, for any $j\in J\cap N^+$, by the induction hypothesis, $\mu^i_j > 0$. Then, the new multiplier associated with $j$ is  
    \begin{equation*}
      \frac{\mu^i_j}{N_{\text{tot}}}\left(N_{\text{tot}} + \sum_{j'\in N^-}\mu^i_{j'}\beta_{j'}\right) =\frac{\mu^i_j}{N_{\text{tot}}} \A_{i+1,\cl{}}(\bar{x}_0;\bar{x}) > 0.\tag*{\qed}
    \end{equation*}
  \end{casesp}
  \end{casesp}
  \iftoggle{createarxiv}{\renewcommand{\qedsymbol}{}}{}
    %This shows that the multipliers for all rays on $F$ are strictly positive\qed
  \end{proof}

Proposition~\ref{prop:positivemu} shows that the barycentric coordinates derived using Algorithm~\ref{alg:ddwithmult} are of a special type. For example, to express $(0.5,0.5) \in [0,1]^2$, we may choose the barycentric coordinates for $(0,1)$ and $(1,0)$ to be zero even though these vertices belong to $[0,1]^2$. However, Algorithm~\ref{alg:ddwithmult} does not produce these barycentric coordinates, and, instead, associates a positive multiplier with all the vertices of the face. This feature is due to the solution we chose for the transportation problem described in Appendix~\ref{app:ddalgorithm}. Even though Proposition~\ref{prop:positivemu} shows that all the associated rational expressions are strictly positive over $\ri(P^k)$, there are faces of $P^k$, where the denominators for some of these expressions become zero. This happens when, at iteration $k$, $N^{\text{tot}}(x_0;x)$ is zero at some $(x_0;x)$, and if $(x_0;x)\in K^{k+1}$, this point must belong to a face of $K^k$ generated by rays with indices in $N^0$.
%by Proposition~\ref{prop:positivemu} occurs when none of $j\in N^+$ belong to a face of $K^k$ that contains $(x_0;x)$. However, if $(x_0;x)\in K^{k+1}$, $\beta\mu^k = \A_{k+1,\cl{}}(x_0;x) \ge 0$, this implies that $\mu^k_{j} = 0$ for $j \in N^-$. In other words, $(x_0;x)$ belongs to a face of $K^k$ generated by rays in $N^0$. 
Since the multipliers $\mu^k_j$, for $j\in N^0$, are not divided by $N_\text{tot}$, the division does not affect the non-zero multipliers. The affected multipliers all converge to zero. Since we eventually linearize the nonlinear terms producing a closed set, we will not consider this issue and will derive inequalities for the relative interior of the face containing $P$ instead.  

\begin{example}\label{ex:megaexample}
  Consider $P$ given as follows:
  \begin{subequations}\label{eq:Pineq}
  \begin{alignat}{2}
    x_1,x_2,x_3\ge 0\label{eq:P1-3}\\
    x_{1}+4 x_{2}+x_{3}\le 7\label{eq:P4}\\
    2x_{1}+x_{2}+x_{3} \le 5\label{eq:P5}\\
    x_{1}+x_{2}+x_{3} \le 4.\label{eq:P6}
  \end{alignat}
\end{subequations}
  \begin{figure}
    \begin{center}
      \newcommand{\pscale}{0.56}
      \let\pscaleQ\pscale
      \let\pscaleR\pscale
      \let\pscaleS\pscale
      \begin{tabular}{cc}
      %\fbox{%
      \begin{subfigure}{0.42\textwidth}
      \begin{tikzpicture}%
      [x={(-0.387138cm, 0.347063cm)},
      y={(0.922022cm, 0.145724cm)},
      z={(0.000000cm, 0.926451cm)},
      scale=\pscaleR,
      back/.style={dotted, thick},
      edge/.style={color=blue!95!black, thick},
      intersectedge/.style={color=red!95!black, thick},
      facet/.style={fill=blue!95!black,fill opacity=0.300000},
      intersect/.style={fill=red!95!black,fill opacity=0.500000},
      plane/.style={fill=yellow!95!black,fill opacity=0.300000},
      vertex/.style={inner sep=1pt,circle,draw=green!25!black,fill=green!75!black,thick}]
    %
    %
    %% This TikZ-picture was produced with Sagemath version 10.2
    %% with the command: ._tikz_3d_in_3d and parameters:
    %% view = [-0.482900000000000, -0.726500000000000, -0.489000000000000]
    %% angle = 143.990000000000
    %% scale = 1
    %% edge_color = blue!95!black
    %% facet_color = blue!95!black
    %% opacity = 0.300000000000000
    %% vertex_color = green
    %% axis = True
    %%
    %% Drawing the interior
%%
\fill[plane] (0.50000, -1.00000, 5.00000) -- (-1.00000, 2.00000, 5.00000) -- (1.50000, 2.00000, 0.00000) -- (3.00000, -1.00000, 0.00000) -- cycle {};
    \node [right=0.1] at (1,2,2) {\eqref{eq:P5}};
%% Drawing the interior
%%
\fill[intersect] (1.85714, 1.28571, 0.00000) -- (0.00000, 0.66667, 4.33333) -- (0.00000, 0.00000, 5.00000) -- (2.50000, 0.00000, 0.00000) -- cycle {};
%%
%%
%% Drawing edges
%%
\draw[intersectedge] (0.00000, 0.00000, 5.00000) -- (0.00000, 0.66667, 4.33333);
\draw[intersectedge] (0.00000, 0.00000, 5.00000) -- (2.50000, 0.00000, 0.00000);
\draw[intersectedge] (0.00000, 0.66667, 4.33333) -- (1.85714, 1.28571, 0.00000);
\draw[intersectedge] (2.50000, 0.00000, 0.00000) -- (1.85714, 1.28571, 0.00000);
\node [right] at (0,0,7) {$\dot{x} = (0,0,7)$};
\node (A) [xshift=15, right] at (0,0,5) {$x' = (0,0,5)$};
\draw [->] (0,0,5) -- (A);
\node (B) [xshift=8,yshift=-3,right] at (0,2/3,13/3) {$\starred{x} = (0,2/3,13/3)$};
\draw [->] (0,2/3,13/3) -- (B);

    %% Drawing the axes
    \draw[color=black,thick,->] (0,0,0) -- (1,0,0) node[anchor=north east]{$x_1$};
    \draw[color=black,thick,->] (0,0,0) -- (0,1,0) node[anchor=north west]{$x_2$};
    \draw[color=black,thick,->] (0,0,0) -- (0,0,2) node[anchor=south]{$x_3$};
    %% Coordinate of the vertices:
    %%
    \coordinate (0.00000, 0.00000, 0.00000) at (0.00000, 0.00000, 0.00000);
    \coordinate (0.00000, 0.00000, 7.00000) at (0.00000, 0.00000, 7.00000);
    \coordinate (0.00000, 1.75000, 0.00000) at (0.00000, 1.75000, 0.00000);
    \coordinate (7.00000, 0.00000, 0.00000) at (7.00000, 0.00000, 0.00000);
    %%
    %%
    %% Drawing edges in the back
    %%
    \draw[edge,back] (0.00000, 0.00000, 0.00000) -- (0.00000, 0.00000, 7.00000);
    %%
    %%
    %% Drawing vertices in the back
    %%
    %%
    %%
    %% Drawing the facets
    %%
    \fill[facet] (7.00000, 0.00000, 0.00000) -- (0.00000, 0.00000, 0.00000) -- (0.00000, 1.75000, 0.00000) -- cycle {};
    \fill[facet] (7.00000, 0.00000, 0.00000) -- (0.00000, 0.00000, 7.00000) -- (0.00000, 1.75000, 0.00000) -- cycle {};
    %%
    %%
    %% Drawing edges in the front
    %%
    \draw[edge] (0.00000, 0.00000, 0.00000) -- (0.00000, 1.75000, 0.00000);
    \draw[edge] (0.00000, 0.00000, 0.00000) -- (7.00000, 0.00000, 0.00000);
    \draw[edge] (0.00000, 0.00000, 7.00000) -- (0.00000, 1.75000, 0.00000);
    \draw[edge] (0.00000, 0.00000, 7.00000) -- (7.00000, 0.00000, 0.00000);
    \draw[edge] (0.00000, 1.75000, 0.00000) -- (7.00000, 0.00000, 0.00000);
    %%
    %%
    %% Drawing the vertices in the front
    %%
    \node[vertex] at (0.00000, 0.00000, 0.00000)     {};
    \node[vertex] at (0.00000, 0.00000, 7.00000)     {};
    \node[vertex] at (0.00000, 1.75000, 0.00000)     {};
    \node[vertex] at (7.00000, 0.00000, 0.00000)     {};
 
      \end{tikzpicture}
      \caption{Polytope bounded by \eqref{eq:P1-3}, \eqref{eq:P4}}\label{fig:Polytopea}
    \end{subfigure}&
    %}
    %\fbox{%
      \begin{subfigure}{0.51\textwidth}
      \begin{tikzpicture}%
        [x={(-0.387138cm, 0.347063cm)},
        y={(0.922022cm, 0.145724cm)},
        z={(0.000000cm, 0.926451cm)},
        scale=\pscaleR,
        back/.style={dotted, thick},
        edge/.style={color=blue!95!black, thick},
        intersectedge/.style={color=red!95!black, thick},
        facet/.style={fill=blue!95!black,fill opacity=0.300000},
        intersect/.style={fill=red!95!black,fill opacity=0.500000},
        plane/.style={fill=yellow!95!black,fill opacity=0.300000},
        vertex/.style={inner sep=1pt,circle,draw=green!25!black,fill=green!75!black,thick}]
      %
      %
      %% This TikZ-picture was produced with Sagemath version 10.2
      %% with the command: ._tikz_3d_in_3d and parameters:
      %% view = [-0.459500000000000, -0.763300000000000, -0.454100000000000]
      %% angle = 149.420000000000
      %% scale = 1
      %% edge_color = blue!95!black
      %% facet_color = blue!95!black
      %% opacity = 0.300000000000000
      %% vertex_color = green
      %% axis = True
      %%
      %% Draw hyperplane
      %% Drawing the axes
\draw[color=black,thick,->] (0,0,0) -- (1,0,0) node[anchor=north east]{$x_1$};
\draw[color=black,thick,->] (0,0,0) -- (0,1,0) node[anchor=north west]{$x_2$};
\draw[color=black,thick,->] (0,0,0) -- (0,0,2) node[anchor=south]{$x_3$};
%% Coordinate of the vertices:
%%
\coordinate (-1.00000, 4.00000, 1.00000) at (-1.00000, 4.00000, 1.00000);
\coordinate (2.00000, 1.00000, 1.00000) at (2.00000, 1.00000, 1.00000);
\coordinate (2.00000, -3.00000, 5.00000) at (2.00000, -3.00000, 5.00000);
\coordinate (-1.00000, 0.00000, 5.00000) at (-1.00000, 0.00000, 5.00000);
%%
%%
%% Drawing the interior
%%
\fill[plane] (-1.00000, 0.00000, 5.00000) -- (-1.00000, 4.00000, 1.00000) -- (2.00000, 1.00000, 1.00000) -- (2.00000, -3.00000, 5.00000) -- cycle {};

\node [right=0.1] at (-1,2.5,2.5) {\eqref{eq:P6}};
\node [right=0.1] at (0,0,5) {$(0,0,5)$};
\node [left=0.1] at (5/2,0,0) {$(5/2,0,0)$};
\node [below=0.1] at (0,0,0) {$(0,0,0)$};
\node [right=0.1] at (0,7/4,0) {$(0,7/4,0)$};
\node (A) [xshift=20, yshift=1, right] at (13/7,9/7,0) {$(13/7,9/7,0)$};
\draw [->] (13/7,9/7,0) -- (A);
\node (B) [xshift=-5, left] at (1,0,3) {$\check{x} = (1,0,3)$};
\node (C) [xshift=-17,yshift=-7,left] at (1,9/13,30/13) {$\phantom{(1,9/13,30/13)}\mathllap{\mathring{x} = (1,9/13,30/13)}$};
\draw [->] (1,9/13,30/13) -- (C);
\node (D) [xshift=-10,yshift=2,left] at (0,0,4) {$\hat{x} = (0,0,4)$};
\draw [->] (0,0,4) -- (D);
\node (E) [xshift=10,yshift=2,right] at (0,7/13,45/13) {$\bar{x} = (0,7/13,45/13)$};
\draw [->] (0,7/13,45/13) -- (E);
\draw [->] (1,0,3) -- (B);
%% Coordinate of the vertices:
%%
\coordinate (0.00000, 0.00000, 4.00000) at (0.00000, 0.00000, 4.00000);
\coordinate (0.00000, 1.00000, 3.00000) at (0.00000, 1.00000, 3.00000);
\coordinate (1.00000, 0.00000, 3.00000) at (1.00000, 0.00000, 3.00000);
\coordinate (1.00000, 1.00000, 2.00000) at (1.00000, 1.00000, 2.00000);
%%
%%
%% Drawing the interior
%%
\fill[intersect] (1.00000, 1.00000, 2.00000) -- (0.00000, 1.00000, 3.00000) -- (0.00000, 0.00000, 4.00000) -- (1.00000, 0.00000, 3.00000) -- cycle {};
%%
%%
%% Drawing edges
%%
\draw[intersectedge] (0.00000, 0.00000, 4.00000) -- (0.00000, 1.00000, 3.00000);
\draw[intersectedge] (0.00000, 0.00000, 4.00000) -- (1.00000, 0.00000, 3.00000);
\draw[intersectedge] (0.00000, 1.00000, 3.00000) -- (1.00000, 1.00000, 2.00000);
\draw[intersectedge] (1.00000, 0.00000, 3.00000) -- (1.00000, 1.00000, 2.00000);
      %% Coordinate of the vertices:
      %%
      \coordinate (0.00000, 0.00000, 0.00000) at (0.00000, 0.00000, 0.00000);
      \coordinate (0.00000, 0.00000, 5.00000) at (0.00000, 0.00000, 5.00000);
      \coordinate (0.00000, 0.66667, 4.33333) at (0.00000, 0.66667, 4.33333);
      \coordinate (0.00000, 1.75000, 0.00000) at (0.00000, 1.75000, 0.00000);
      \coordinate (2.50000, 0.00000, 0.00000) at (2.50000, 0.00000, 0.00000);
      \coordinate (1.85714, 1.28571, 0.00000) at (1.85714, 1.28571, 0.00000);
      %%
      %%
      %% Drawing edges in the back
      %%
      \draw[edge,back] (0.00000, 0.00000, 0.00000) -- (0.00000, 0.00000, 5.00000);
      %%
      %%
      %% Drawing vertices in the back
      %%
      %%
      %%
      %% Drawing the facets
      %%
      \fill[facet] (1.85714, 1.28571, 0.00000) -- (0.00000, 0.66667, 4.33333) -- (0.00000, 0.00000, 5.00000) -- (2.50000, 0.00000, 0.00000) -- cycle {};
      \fill[facet] (1.85714, 1.28571, 0.00000) -- (0.00000, 1.75000, 0.00000) -- (0.00000, 0.00000, 0.00000) -- (2.50000, 0.00000, 0.00000) -- cycle {};
      \fill[facet] (1.85714, 1.28571, 0.00000) -- (0.00000, 0.66667, 4.33333) -- (0.00000, 1.75000, 0.00000) -- cycle {};
      %%
      %%
      %% Drawing edges in the front
      %%
      \draw[edge] (0.00000, 0.00000, 0.00000) -- (0.00000, 1.75000, 0.00000);
      \draw[edge] (0.00000, 0.00000, 0.00000) -- (2.50000, 0.00000, 0.00000);
      \draw[edge] (0.00000, 0.00000, 5.00000) -- (0.00000, 0.66667, 4.33333);
      \draw[edge] (0.00000, 0.00000, 5.00000) -- (2.50000, 0.00000, 0.00000);
      \draw[edge] (0.00000, 0.66667, 4.33333) -- (0.00000, 1.75000, 0.00000);
      \draw[edge] (0.00000, 0.66667, 4.33333) -- (1.85714, 1.28571, 0.00000);
      \draw[edge] (0.00000, 1.75000, 0.00000) -- (1.85714, 1.28571, 0.00000);
      \draw[edge] (2.50000, 0.00000, 0.00000) -- (1.85714, 1.28571, 0.00000);
      %%
      %%
      %% Drawing the vertices in the front
      %%
      \node[vertex] at (0.00000, 0.00000, 0.00000)     {};
      \node[vertex] at (0.00000, 0.00000, 5.00000)     {};
      \node[vertex] at (0.00000, 0.66667, 4.33333)     {};
      \node[vertex] at (0.00000, 1.75000, 0.00000)     {};
      \node[vertex] at (2.50000, 0.00000, 0.00000)     {};
      \node[vertex] at (1.85714, 1.28571, 0.00000)     {};
      \end{tikzpicture}
      \caption{Polytope bounded by \eqref{eq:P1-3}-\eqref{eq:P5}}\label{fig:Polytopeb}
    \end{subfigure}\\[3ex]
    %}
    %\fbox{
    \multicolumn{2}{c}{\begin{subfigure}{0.6\textwidth}
      \begin{tikzpicture}%
        [x={(-0.387138cm, 0.347063cm)},
        y={(0.922022cm, 0.145724cm)},
        z={(0.000000cm, 0.926451cm)},
        scale=\pscaleS,
        back/.style={dotted, thick},
        edge/.style={color=blue!95!black, thick},
        facet/.style={fill=blue!95!black,fill opacity=0.300000},
        vertex/.style={inner sep=1pt,circle,
        draw=green!25!black,fill=green!75!black,thick},
        myvertex/.style={inner sep=0.1pt,circle,
        fill=red!75!black}
        ]
      %
      %
      %% This TikZ-picture was produced with Sagemath version 10.2
      %% with the command: ._tikz_3d_in_3d and parameters:
      %% view = [-0.459500000000000, -0.763300000000000, -0.454100000000000]
      %% angle = 149.420000000000
      %% scale = 1
      %% edge_color = blue!95!black
      %% facet_color = blue!95!black
      %% opacity = 0.300000000000000
      %% vertex_color = green
      %% axis = True
      %%
      %% Drawing the axes
      \draw[color=black,thick,->] (0,0,0) -- (1,0,0) node[anchor=north east]{$x_1$};
      \draw[color=black,thick,->] (0,0,0) -- (0,1,0) node[anchor=north west]{$x_2$};
      \draw[color=black,thick,->] (0,0,0) -- (0,0,2) node[anchor=south]{$x_3$};
      \node [above right] at (0,0,4) {$(0,0,4)$};
      \node [below] at (0,0,0) {$(0,0,0)$};
      \node (B) [xshift=20,yshift=5,right] at (13/7,9/7,0) {$\bigl(13/7,9/7,0\bigr)$};
      \node (A) [right=2] at (1,1,2) {$\bigl(1,1,2\bigr)$};
      \node [left] at (5/2,0,0) {$(5/2,0,0)$};
      \node [left] at (1,0,3) {$(1,0,3)$};
      \node [right] at (0,1,3) {$(0,1,3)$};
      \node [right] at (0,7/4,0) {$(0,7/4,0)$};
      \draw [->] (1,1,2) -- (A);
      \draw [->] (13/7,9/7,0) -- (B);
      \node (C) [yshift=5, xshift=20, right] at (0, 7/13, 45/13) {$\bar{x} = (0, 7/13, 45/13)$};
      \draw[->] (0, 7/13, 45/13) -- (C);
      \node (D) [yshift=12,xshift=-12,left] at (0, 8/15, 52/15)  {$\tilde{x} = (0, 8/15, 52/15)$};
      \draw[->] (0, 8/15, 52/15) -- (D);
      %% Coordinate of the vertices:
      %%
      \coordinate (0.00000, 0.00000, 0.00000) at (0.00000, 0.00000, 0.00000);
      \coordinate (0.00000, 0.00000, 4.00000) at (0.00000, 0.00000, 4.00000);
      \coordinate (0.00000, 1.00000, 3.00000) at (0.00000, 1.00000, 3.00000);
      \coordinate (0.00000, 1.75000, 0.00000) at (0.00000, 1.75000, 0.00000);
      \coordinate (1.00000, 0.00000, 3.00000) at (1.00000, 0.00000, 3.00000);
      \coordinate (1.00000, 1.00000, 2.00000) at (1.00000, 1.00000, 2.00000);
      \coordinate (2.50000, 0.00000, 0.00000) at (2.50000, 0.00000, 0.00000);
      \coordinate (1.85714, 1.28571, 0.00000) at (1.85714, 1.28571, 0.00000);
      %%
      %%
      %% Drawing edges in the back
      %%
      \draw[edge,back] (0.00000, 0.00000, 0.00000) -- (0.00000, 0.00000, 4.00000);
      %%
      %%
      %% Drawing vertices in the back
      %%
      %%
      %%
      %% Drawing the facets
      %%
      \fill[facet] (1.85714, 1.28571, 0.00000) -- (0.00000, 1.75000, 0.00000) -- (0.00000, 1.00000, 3.00000) -- (1.00000, 1.00000, 2.00000) -- cycle {};
      \fill[facet] (1.85714, 1.28571, 0.00000) -- (1.00000, 1.00000, 2.00000) -- (1.00000, 0.00000, 3.00000) -- (2.50000, 0.00000, 0.00000) -- cycle {};
      \fill[facet] (1.85714, 1.28571, 0.00000) -- (0.00000, 1.75000, 0.00000) -- (0.00000, 0.00000, 0.00000) -- (2.50000, 0.00000, 0.00000) -- cycle {};
      \fill[facet] (1.00000, 1.00000, 2.00000) -- (0.00000, 1.00000, 3.00000) -- (0.00000, 0.00000, 4.00000) -- (1.00000, 0.00000, 3.00000) -- cycle {};
      %%
      %%
      %% Drawing edges in the front
      %%
      \draw[edge] (0.00000, 0.00000, 0.00000) -- (0.00000, 1.75000, 0.00000);
      \draw[edge] (0.00000, 0.00000, 0.00000) -- (2.50000, 0.00000, 0.00000);
      \draw[edge] (0.00000, 0.00000, 4.00000) -- (0.00000, 1.00000, 3.00000);
      \draw[edge] (0.00000, 0.00000, 4.00000) -- (1.00000, 0.00000, 3.00000);
      \draw[edge] (0.00000, 1.00000, 3.00000) -- (0.00000, 1.75000, 0.00000);
      \draw[edge] (0.00000, 1.00000, 3.00000) -- (1.00000, 1.00000, 2.00000);
      \draw[edge] (0.00000, 1.75000, 0.00000) -- (1.85714, 1.28571, 0.00000);
      \draw[edge] (1.00000, 0.00000, 3.00000) -- (1.00000, 1.00000, 2.00000);
      \draw[edge] (1.00000, 0.00000, 3.00000) -- (2.50000, 0.00000, 0.00000);
      \draw[edge] (1.00000, 1.00000, 2.00000) -- (1.85714, 1.28571, 0.00000);
      \draw[edge] (2.50000, 0.00000, 0.00000) -- (1.85714, 1.28571, 0.00000);
      %%
      %%
      %% Drawing the vertices in the front
      %%
      \node[vertex] at (0.00000, 0.00000, 0.00000)     {};
      \node[vertex] at (0.00000, 0.00000, 4.00000)     {};
      \node[vertex] at (0.00000, 1.00000, 3.00000)     {};
      \node[vertex] at (0.00000, 1.75000, 0.00000)     {};
      \node[vertex] at (1.00000, 0.00000, 3.00000)     {};
      \node[vertex] at (1.00000, 1.00000, 2.00000)     {};
      \node[vertex] at (2.50000, 0.00000, 0.00000)     {};
      \node[vertex] at (1.85714, 1.28571, 0.00000)     {};
      \node[myvertex] at (0, 7/13, 45/13)     {};
      \node[myvertex] at (0, 9/13, 30/13)     {};
      \node[myvertex] at (0, 8/15, 52/15)     {};
      \node[myvertex] at (0, 2/5, 13/5)     {};
      \end{tikzpicture}
      \caption{Polytope bounded by \eqref{eq:P1-3}-\eqref{eq:P6}}\label{fig:Polytopec}
    \end{subfigure}
    %}
    }
  \end{tabular}
      
  \end{center}
  \caption{Stages of relaxation construction. Figure~\ref{fig:Polytopea} introduces \eqref{eq:P5} on the simplex bounded by \eqref{eq:P1-3} and \eqref{eq:P4}. The points
  obtained as convex combinations from $\dot{x}$ are shown. Figure~\ref{fig:Polytopeb} shows the effect of introducing \eqref{eq:P6} and the points $\hat{x}$, $\check{x}$, $\bar{x}$, and $\mathring{x}$ introduced to capture the multiplier of $x'$.}
  \label{fig:polytopeP}
  \end{figure}
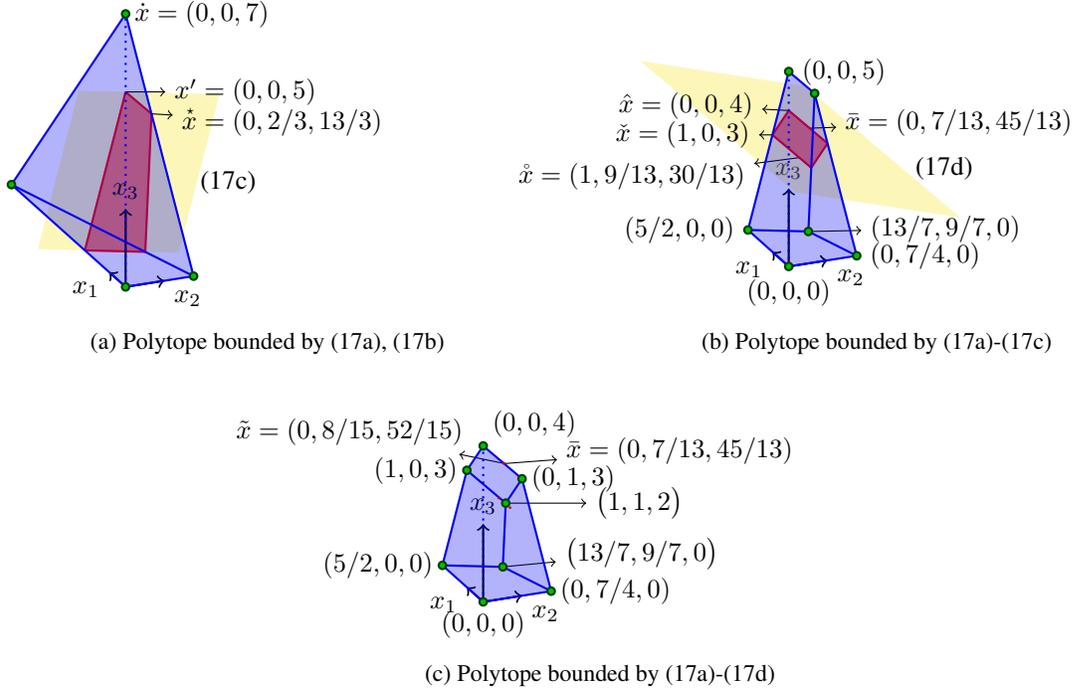%
  For an illustration of the polytope, see Figure~\ref{fig:Polytopec}, where the inequalities are introduced in the order they are specified above. Figure~\ref{fig:Polytopea} shows the polytope obtained with \eqref{eq:P1-3} and \eqref{eq:P4}, while Figure~\ref{fig:Polytopeb} illustrates when \eqref{eq:P5} has also been added. $P$ is a simple polytope and, therefore, we can obtain expressions for $\mu$ via explicit formulae \cite{warren2007barycentric}. In this case, we write the following dehomogenized matrix depicting the vertices of $P$:
  \begingroup
  \renewcommand*{\arraystretch}{1.5}
  \begin{equation*}
  R = \begin{pmatrix}
    1 & 1 & 1 & 1 & 1 & 1 & 1 & 1 
    \\
     0 & \frac{5}{2} & 0 & 1 & 0 & \frac{13}{7} & 0 & 1 
    \\
     0 & 0 & 0 & 0 & \frac{7}{4} & \frac{9}{7} & 1 & 1 
    \\
     0 & 0 & 4 & 3 & 0 & 0 & 3 & 2 
  \end{pmatrix}.
\end{equation*}
\endgroup
  We process the inequalities in the order $[1\cl{}6]$. For ease of reference, instead of indexing $\mu$ in the order listed in $R$, we index $\mu$ using the points themselves. Thus, the barycentric coordinate for $\hat{x} = (0,0,4)$ is denoted as $\mu_{\hat{x}}$.
  First, we show that the barycentric coordinates from Algorithm~\ref{alg:ddwithmult} differ from those calculated using \cite{warren2007barycentric} when redundant rays are present in the former. Since the absolute value of the determinant of the normal vectors at $(0,0,4)$ is $1$, the calculation of the barycentric coordinate using \cite{warren2007barycentric} yields, 
  \begin{equation}\label{eq:fullmu3example}
    \mu_{\hat{x}} = \dfrac{x_{3} \left(7-x_{1}-4 x_{2}-x_{3}\right) \left(5-2 x_{1}-x_{2}-x_{3}\right)}{d(x)},
  \end{equation}
  where the numerator consists of inequalities that are not tight at $(0,0,4)$ multiplied with the determinant of the normal vectors, which is $1$. Here, $d(x)$ is the sum of numerators of $\mu_i$, $i\in\{1\cl{}8\}$, and is given by:
  \begin{equation}\label{eq:dxformula}
    \begin{split}
    d(x) = 
    5 x_{1}^{2}+12 x_{1} x_{2}+9 x_{3} x_{1}+7 x_{2}^{2}+11 x_{3} x_{2}\\
    {}+4 x_{3}^{2}-55 x_{1}-63 x_{2}-48 x_{3}+140.
    \end{split}
  \end{equation}
  Now, we show that $\mu_{\hat{x}}$ differs from the coordinate computed by Algorithm~\ref{alg:ddwithmult}. The algorithm produces the following redundant points that are not in $\extpt(P)$: 
  \begin{equation*}
    (0,7/13,45/13), (0,9/13,30/13), (0,8/15,52/15), \text{ and } (1,2/5,13/5).
  \end{equation*}
  In particular, $\bar{x} = (0,7/13,45/13) = 6/13(0,0,4)+7/13(0,1,3)$ and $\tilde{x} = (0,8/15,52/15) = 7/15(0,0,4)+8/15(0,1,3)$ (see Figure~\ref{fig:Polytopec}). (The two points $\bar{x}$ and $\tilde{x}$ are too close to be distinguishable in Figure~\ref{fig:Polytopec}). 
  %Then, if we denote the barycentric coordinates for $\bar{x}$, $\tilde{x}$, and $\hat{x} = (0,0,4)$ obtained using Algorithm~\ref{alg:ddwithmult} as $\mu_{\bar{x}}$, $\mu_{\tilde{x}}$, and $\mu_{\hat{x}}$ respectively, 
  It turns out that $\mu_{\hat{x}}= \mu^3_{\hat{x}} + 6/13\mu^3_{\bar{x}} + 7/15\mu^3_{\tilde{x}}$, which highlights a source of difference between \eqref{eq:fullmu3example} and the formulae generated by Algorithm~\ref{alg:ddwithmult}. There are additional differences. First, Algorithm~\ref{alg:ddwithmult} works for non-simple polytopes and polyhedra. Second, Algorithm~\ref{alg:ddwithmult} is iterative in nature, which makes it well-suited for constructing a relaxation hierarchy. Third, the denominators of the resulting expressions only approach zero as the coordinates converge to zero. Finally, as detailed in Theorem~\ref{thm:expressmulin}, the barycentric coordinates are affinely related across iterations. We illustrate this last point next. Specifically, $\mu^3_{\hat{x}}$ computed by Algorithm~\ref{alg:ddwithmult} is:
  \begin{equation}\label{eq:mu3definition}
    \mu^3_{\hat{x}} = \dfrac{5 x_{3} \left(x_{1}+4 x_{2}+x_{3}-7\right)^{2} \left(-5+x_{3}+x_{2}+2 x_{1}\right)}{d(x)\left(-35+5 x_{3}+7 x_{2}+5 x_{1}\right)}.
  \end{equation}
  We backtrack to the step where we have not yet used the inequality~\eqref{eq:P6} (see Figure~\ref{fig:Polytopeb}). Let us consider the derivation of $\mu^3_{\hat{x}}$. The polytope has vertices $(0,0,0)$ and $(0,0,5)$ such that $(0,0,4)$ lies on the edge connecting them. There are four vertices $(5/2,0,0)$, $(0,0,0)$, $(13/7,9/7,0)$, and $(0,7/4,0)$ that satisfy \eqref{eq:P6}. 
  We use these to write:
  \begin{equation*}
    N_{\text{tot}} = (\beta_+)^\intercal \mu^3 = \dfrac{-x_{3}^{2}+\left(-x_{1}-x_{2}+12\right) x_{3}+5 x_{1}+7 x_{2}-35}{5 x_{1}+7 x_{2}+5 x_{3}-35}.
  \end{equation*}
  Then, $(0,0,4) = 4/5(0,0,5)+1/5(0,0,0)$. Moreover, ${x'} = (0,0,5)$ is also used in convex combinations with other vertices to produce $\bar{x} = (0,7/13,45/13)$, $\check{x} = (1,0,3)$, and $\mathring{x} = (1,9/13,30/13)$ with its multiplier being $9/13$, $3/5$, and $6/13$, respectively. Then, 
  \begin{equation}\label{eq:mulinN-example}
    \mu^2_{x'} = 4/5 \mu^3_{\hat{x}} + 9/13\mu^3_{\bar{x}} + 3/5\mu^3_{\check{x}} + 6/13\mu^3_{\mathring{x}}
  \end{equation}
  and has the expression:
  \begin{equation}\label{eq:exactm2xprime}
    \mu^2_{x'} = \dfrac{x_{3} \left(-7+x_{3}+4 x_{2}+x_{1}\right)}{-35+5 x_{3}+7 x_{2}+5 x_{1}}.
  \end{equation}
 We will leverage \eqref{eq:mulinN-example}, an example of the affine relation \eqref{eq:multipliersk+1tomultipliersk}, for our relaxations. It is useful to note that the denominator in \eqref{eq:mu3definition} is significantly more complicated than the one in \eqref{eq:exactm2xprime}. In other words, \eqref{eq:mulinN-example} can be visualized as a way to cancel the factor $d(x)$ from the denominator. Algorithm~\ref{alg:ddwithmult} discovers this relation by rewriting $4-x_1-x_2-x_3\ge 0$ as 
  \begin{equation}\label{eq:Pbetaineq}
    \begin{split}
    \beta\mu^2 = 4\mu^2_{(0,0,0)} + 9/4\mu^2_{(0,7/4,0)} - 1\mu^2_{(0,0,5)} - 1\mu^2_{(0,2/3,13/3)}\\
     + 3/2\mu^2_{(5/2,0,0)} + 6/7\mu^2_{(13/7,9/7,0)} \ge 0.
    \end{split}
  \end{equation}
  This expression is $[4,-1,-1,-1] R^2$ computed in Step~\ref{algstep:definealphabeta} of Algorithm~\ref{alg:ddwithmult}. The coefficients in this equality are then the slacks for each point with respect to $4-x_1-x_2-x_3\ge 0$, where, for example, the slack for $(0,0,0)$ is $4$.
  We notice that $(0,0,5) \in N^-$, and its barycentric coordinate $\mu^2_{x'}$ is computed from  $\mu^3$ as a non-negative combination of multipliers $\mu^3_x$ where $x$ is a vertex strictly feasible to \eqref{eq:P6}, \textit{i.e.}, its index belongs to $N^+$. If $i$ is the index of such a vertex, its non-negative multiplier is $(\beta_i)/(\beta_i + 1)$, where the denominator $\beta_i+1$ is used to express $R^3$ in dehomogenized form. This results in the expression \eqref{eq:mulinN-example}.
  At the previous iteration, the outer-approximation is formed by the simplex defined by \eqref{eq:P1-3} and \eqref{eq:P4}, and its barycentric coordinates are affine functions of the problem variables. In this case, let $\dot{x} = (0,0,7)$, and we have that 
  \begin{equation}\label{eq:level1mu}
    \mu^1_{\dot{x}}=\frac{x_3}{7} = \frac{5}{7} \mu^2_{x'} + \frac{13}{21}\mu^2_{\starred{x}},
  \end{equation}
  where $\starred{x} = {(0,2/3,13/3)}$. 
  We consider another example of \eqref{eq:multipliersk+1tomultipliersk}, relating $\mu^2_{(0,0,0)}$ to $\mu^3$, where the latter is computed when \eqref{eq:P6} is introduced. Since $(0,0,0)$ satisfies \eqref{eq:P6} strictly, its index is in $N^+$ and its barycentric coordinate is:
  \begin{equation}\label{eq:mulinN+example}
    \mu^2_{(0,0,0)} = \mu^3_{(0,0,0)} + \frac{1}{5} \mu^3_{\bar{x}} + \frac{1}{5}\mu^3_{(0,8/15,52/15)}.
  \end{equation}
  Unlike the previous case, where a new point was obtained as a convex combination of prior vertices, Equation~\eqref{eq:mulinN+example} illustrates the case where the multiplier for the same point changes from one level to another. 
  \qed
\end{example}

\begin{remark}\label{rmk:removeredundant}
  Double-description can be modified so that it does not generate redundant rays \cite{Fukuda1995,padberg2013linear}. However, we do not use these modifications since the resulting barycentric coordinates do not have closed-form expressions. Instead, we use the following constructive approach illustrated in Example~\ref{ex:megaexample}. For any ray, $r$, that is redundant and occurs as column $j$ in $R^i$, we solve a linear feasibility problem to express it as $r=R^i_{\cl{},J}\nu$, where $J=\{\cl{}p_i\}\backslash\{j\}$. Then, we rewrite \eqref{eq:pointrep} as $(x_0;x) = R_{\cl{},J} (\mu^i_J +\mu^i_j\nu) + L^i\theta^i$ and observe that $\mu^i_J + \mu^i_j\nu$ is also a rational expression. Since the multipliers of the retained columns do not reduce, Proposition~\ref{prop:positivemu} continues to hold. \qed
\end{remark}

\section{An algebraic relaxation hierarchy}\label{sec:algebraic}

In this section, we develop the algebraic underpinnings for our hierarchy. Specifically, (i) we utilize the algebraic structure of barycentric coordinates, and (ii)  we invoke Algorithm~\ref{alg:ddwithmult} on multiple constraint orders simultaneously. We refer the reader to Proposition~\ref{prop:expansionhelps} and Corollary~\ref{cor:allorders} for motivation regarding these two prominent features of our relaxation schemes. The next remark provides a concrete illustration of the benefit of the second feature. 

\begin{remark}\label{rmk:singleineqRLT}
  We investigate the relaxation strength that arises from creating various outer-approximations of $P$ simultaneously. Consider the extreme case where $P$ is outer-approximated using one constraint at a time. We argue that such an outer-approximation yields all the constraints in the first-level RLT relaxation. Consider executing Algorithm~\ref{alg:ddwithmult} on the $i^{\text{th}}$ constraint defining $P$. For ease of notation, we express this constraint as $a_{i0} + a^\intercal_i x\ge 0$. Assume without loss of generality that $a_{ij_i}\ne 0$ for some $j_i\in \{1\cl{}n\}$. At Step~\ref{algstep:definealphabeta}, we compute $\alpha = a$ and $\beta=a_0$. Then, Steps \ref{algstep:definemu1}-\ref{algstep:defineL1} compute:
  \begin{subequations}\label{eq:oneineq}
    \begin{align}
       &L^{i,1} = \left(0_{1\times n}; -a_{ij}\sign(a_{ij_i})e_{j_i} + |a_{ij_i}| e_j\right)_{j\in \{1\cl{}n\}\backslash{j_i}}\\
       &\theta^{i,1}(1;x) = \left(x_j/|a_{ij_i}|\right)_{j\in \{1\cl{}n\}\backslash{j_i}}\\
       &R^{i,1} = \left(e_{j_i}; e_0-a_{i0}e_{j_i}\sign(a_{ij_i})\right)\\
       &\mu^{i,1}(1;x) = \left((a_{i0}+a^\intercal_i x)/|a_{ij_i}|,1/|a_{ij_i}|\right),\label{eq:oneineqmu}
    \end{align}
  \end{subequations}
  where we have used the superscript $i$ to track the constraint index. Now, consider Constraint \eqref{eq:DDkyscale}. Since the first element of $\mu^{i,1}$ is $(a_{i0}+a^\intercal x)/|a_{ij_i}|$, this constraint reduces to $(a_{i0}+a^\intercal_i x)y \in (a_{i0}+a^\intercal_i x) C$, which is precisely the constraint derived using RLT. When the set of lineality directions is not empty, the relaxation proposed in Theorem~\ref{thm:simpleAChierarchy} is not as useful if only one outer-approximation of $P$ is used. This is because the constraint $y\in C$ cannot be scaled by $\theta$ whose sign is undetermined. However, with multiple outer-approximations, the introduced variables are $x_iy_j$ for $i\in \{1\cl{}n\}$ and $j\in\{1\cl{}n_y\}$ and the same variables are used to express \eqref{eq:DDkyscale} regardless of which constraint of $P$ is used to outer-approximate $P$. As we have shown, using $\mu^{i,1}$ for each $i$ to scale \eqref{eq:DDkyscale} produces all RLT inequalities at the first-level of the hierarchy. 

  We provide a simple counting argument to better understand the source of this strength in the relaxation. To convexify \eqref{eq:DDkyscale}, where $C$ is scaled with $a_{i0}+a^\intercal_i x$ for each $i$, we introduce $nn_y$ variables, one for each $x_iy_j$. However, when $\mu^{i,1}y$ variables are linearized we introduce $mn_y$ variables. Since $P$ is bounded, $m > n$ and the latter linearization introduces more variables. Additionally, observe that by defining $\mu^{i,1}$ using \eqref{eq:oneineqmu}, and similarly defining $\mu^{i,1}y^\intercal$ as $\frac{1}{|a_{ij_i}|}(a_{i0}; a_i) (1;x)y^\intercal$, the resulting point satisfies the inequalities obtained by scaling \eqref{eq:DDkyscale} using $\mu^{i,1}$.
\qed
\end{remark}

Remark~\ref{rmk:singleineqRLT} strengthens the DP based relaxations proposed in Theorems~\ref{thm:simplehierarchy} and \ref{thm:simpleAChierarchy} by incorporating algebraic expressions for $\mu$. This parallels the relationship between lift-and-project/DP hierarchy and RLT relaxations in 0-1 programs. Specifically, the contrast in lift-and-project and RLT is captured in Proposition~\ref{prop:RLThypercubetight} and Corollary~\ref{cor:allorders}: the former preserves the original constraint order, while the latter selects an arbitrary set of $k$ variables to process their upper-bounding constraints. Similarly, the barycentric coordinates used in Theorem~\ref{thm:simplehierarchy} correspond to the sequence of outer-approximations $P^k$ derived after fixing the constraint order while the hierarchy proposed in this section allows for all possible constraint orders.

As mentioned before, rational functions are required to represent barycentric coordinates for arbitrary polytopes, whereas RLT relaxations rely solely on polynomial expressions. This distinction has significant implications. In fact, it is known that any non-negative polynomial can be written as a sum of squares of rational functions, but not necessarily as a sum of squares of polynomials. This demonstrates that our representations are more expressive than what is considered in RLT relaxations. The next result shows that considering such functions is not only crucial for representation but also leads to the discovery of new inequalities that are not captured in the standard RLT hierarchy. 

\begin{proposition}\label{prop:twoineq}
  Consider $a, b\in \Re^n_+$ and 
  \begin{equation*}
    P = \left\{x\;\middle|\; x\ge 0, 1 - a^\intercal x \ge 0, 1 - b^\intercal x \ge 0\right\}.
  \end{equation*}
  Let $N^> = \{i\mid a_i > b_i\}$, $N^< = \{i\mid a_i < b_i\}$, $N^= =\{i\mid a_i = b_i\}$, and $d(x) = 1 - \sum_{j\in N^<} a_jx_j - \sum_{j\in N^>} b_jx_j$. Then, after processing all the rows, the barycentric coordinates are:
  \begin{subequations}\label{eq:tworow}
  \begin{alignat}{2}
    &\mu(a_ie_0 + e_i) = x_i &\quad& i\in N^= \label{eq:tworowN0} \\
    &\mu(e_0) = (1-a^\intercal x)(1-b^\intercal x)/d(x) \label{eq:tworowNplus1}\\
    &\mu(a_ie_0 + e_i) = x_i(1-b^\intercal x)/d(x) && i \in N^> \label{eq:tworowNplus2} \\
    &\mu(b_ie_0 + e_i) = x_i(1-a^\intercal x)/d(x) && i \in N^< \label{eq:tworowNminus1} \\
    &\begin{aligned}
      &\mu\left((a_ib_j - a_jb_i)e_0 + (b_j-a_j)e_i+(a_i-b_i)e_j\right)\\
      &\quad = x_ix_j/d(x)
    \end{aligned} && i\in N^>, j\in N^<,\label{eq:tworowNminus2}
  \end{alignat}
\end{subequations}
where the argument of $\mu$ is the ray of $P$ corresponding to the barycentric coordinate and the iteration superscript has been dropped for notational convenience.
\end{proposition}
\long\def\proptwoineqproof{
  We initialize Algorithm~\ref{alg:ddwithmult} so that $x\ge 0$ are already processed by setting $L^0=[]$ and $\theta^0=[]$, $R^0 = \Id_{n+1}$, and $\mu^0 = (x_0;x)$. Then, it is easy to verify that after processing, $1-a^\intercal x\ge 0$, we have
  \begin{equation*}
  R^{1} = \begin{pmatrix}
      1 & a^\intercal\\
       0 &\Id_n
  \end{pmatrix} \text{ and }
  \begin{pmatrix}
      x_0-a^\intercal x 
      \\
      x
\end{pmatrix}.
\end{equation*}
During the last step, $N^{\text{tot}} = d(x)$. Then, $\beta = [1; a^\intercal - b^\intercal]$, $N^+ = N^> \cup \{1\}$, $N^- = N^<$, and $N^0= N^{=}$. It can be verified that barycentric coordinates in \eqref{eq:tworowN0} are obtained by indices in $N^0$, \eqref{eq:tworowNplus1} and \eqref{eq:tworowNplus2} from indices in $N^+$, and \eqref{eq:tworowNminus1} and \eqref{eq:tworowNminus2} from indices in $N^+\times N^-$, where we obtained \eqref{eq:tworowNplus1} and \eqref{eq:tworowNminus1} by choosing $e_0$ from $N^+$ while \eqref{eq:tworowNplus2} and \eqref{eq:tworowNminus2} were obtained by choosing indices in $N^>$ which also belong to $N^+$.
}
\iftoggle{allproofsinpaper}{%
\begin{proof}
\proptwoineqproof\qed
\end{proof}}{}

Proposition~\ref{prop:twoineq} and Theorem~\ref{thm:expressmulin} find relations that do not have their counterparts in RLT. Observe that the numerators of $\mu$ in Proposition~\ref{prop:twoineq} are a product of two constraint expressions. Therefore, Constraint \eqref{eq:DDkyscale}, where we scale $y\in C$ using barycentric coordinates,  bears similarity to RLT except that it is scaled by $1/d(x)$. This scaling along with \eqref{eq:multipliersk+1tomultipliersk} yields nontrivial inequalities.

\begin{remark}\label{rmk:advantageoverRLT}
  %From Proposition~\ref{prop:twoineq} and \eqref{eq:multipliersk+1tomultipliersk}, 
  Let $\mu_{ij}$ denote $\mu\left((a_ib_j - a_jb_i)e_0 + (b_j-a_j)e_i+(a_i-b_i)e_j\right)$ for $i\in N^>$ and $j\in N^<$ as computed in \eqref{eq:tworowNminus2}. Then,
  \begin{subequations}\label{eq:expressxisimplex+1}
    \begin{align}
    \begin{aligned}
    &x_i = x_i(1-b^\intercal x)/d(x) + \sum_{j\in N^<} (b_j-a_j)x_ix_j/d(x)\\
    &\quad\quad= \mu(a_ie_0+e_i)+ \sum_{j\in N^<} (b_j-a_j)\mu_{ij}
    \end{aligned} & \text{ for } i\in N^{>}\label{eq:xN>}\\
    \begin{aligned}
    &x_j = x_j(1-a^\intercal x)/d(x) + \sum_{i\in N^>}(a_i-b_i)x_ix_j/d(x)\\
    &\quad\quad=\mu(b_je_0+e_j)+ \sum_{j\in N^>} (a_j-b_j)\mu_{ij}
    \end{aligned} & \text{ for } j \in N^{<}\label{eq:xN<},
    \end{align}
  \end{subequations}
  where $d(x)$ is as in Proposition~\ref{prop:twoineq}. The expression $x^\alpha y_\ell/d(x)$ for $\alpha = (\alpha_i)_{i\in [\cl{}n]}$ and $\ell\in\{\cl{}n_y\}$ is linearized as $w_{d,\alpha,\ell}$ and $\mu y$ is then written as a linear function of these introduced variables. For example, given $i\in N^>$ and $j\in N^<$, we have 
  \begin{equation}\label{eq:expressmuy}
    y\mu(a_ie_0+e_i) = w_{d,e_i,\cl{}} - \sum_{j=1}^n b_j w_{d,e_i+e_j,\cl{}} \text{ and } y\mu_{ij} = w_{d,e_i+e_j,\cl{}},\\
  \end{equation}
  which using \eqref{eq:xN>} yields
  \begin{equation}\label{eq:expressxiyell}
    x_iy_\ell = w_{d,e_i,\ell} - \sum_{j=1}^n b_j w_{d,e_i+e_j,\ell} + \sum_{j\in N^<}(b_j-a_j) w_{d,e_i+e_j,\ell}.\\ 
  \end{equation}
  RLT scales $y\in C$  using the numerators of the expressions in \eqref{eq:tworow}. However, the way we cancel $d(x)$ to express $x_iy_{\ell}$ in terms of $x_ix_jy_{\ell}/d(x)$ using the expressions above is different. Since \eqref{eq:expressmuy} and \eqref{eq:expressxiyell} express $x_iy_\ell$ as a conic combination of $\mu y_\ell$, it follows that requiring \eqref{eq:expressxiyell} for all $\ell$ enforces $x_iy \in x_i C$ even though this containment is not imposed directly. \qed
\end{remark}

Observe that \eqref{eq:tworow} does not include all constraint products. For instance, the term $x_i(1 - a^\intercal x)/d(x)$ for $i \in N^>$ is not used.
%This is because the non-negativity of this product follows from the expressions in \eqref{eq:tworow}. 

\begin{example}\label{ex:megacontinues}
  We continue Example~\ref{ex:megaexample} to illustrate the linearization techniques and  inequalities used in our relaxation hierarchy. Assume \eqref{eq:Pineq} is a strict subset of inequalities describing $P$. Then, at level-$3$, we compute $\mu^3_3$ (see \eqref{eq:mu3definition}) and replace $\prod_{i=1}^3 x_i^{\alpha_i}/d_3(x)$, where $d_3(x) = d(x) (-35+5x_3+7x_2+5x_1)$, with $w_{3,\alpha_1,\alpha_2,\alpha_3}$, thereby linearizing $\mu^3_3$. Here, the first index of $w$ tracks the index of the denominator used in the relaxation hierarchy. Then, we obtain:
  \begin{equation}\label{eq:linearmu3}
  \begin{split}
    \mu^3_3 = -1225 w_{3,0,0,1}+840 w_{3,1,0,1}+1645 w_{3,0,1,1}+595 w_{3,0,0,2}\\
    -165 w_{3,2,0,1}
    -830 w_{3,1,1,1}-260 w_{3,1,0,2}-680 w_{3,0,2,1}\\
    -550 w_{3,0,1,2}
    -95 w_{3,0,0,3}
    +10 w_{3,3,0,1}+85 w_{3,2,1,1}\\
    +25 w_{3,2,0,2}
    +200 w_{3,1,2,1}
    +130 w_{3,1,1,2}+20 w_{3,1,0,3}\\
    +80 w_{3,0,3,1}+120 w_{3,0,2,2}+45 w_{3,0,1,3}+5 w_{3,0,0,4},
  \end{split}
  \end{equation}
  Equation~\eqref{eq:linearmu3} linearizes \eqref{eq:mu3definition}. The algebraic representation of $\mu$ also reveals other inequalities such as the one below:
  \begin{equation*}
    \begin{split}
  \frac{(4-x_1-x_2-x_3)x_{3}}{-35+5 x_{3}+7 x_{2}+5 x_{1}} \ge 0, \text{ or }
    4w_{2,0,0,1} -w_{2,1,0,1} - w_{2,0,1,1} - w_{2,0,02} \ge 0.
    \end{split}
  \end{equation*}
  This valid inequality exploits the denominator in \eqref{eq:exactm2xprime} and is also useful when Constraint \eqref{eq:P6} has not yet been processed by Algorithm~\ref{alg:ddwithmult}, \textit{i.e.},\/ the $V$-representation is obtained for the polytope defined by \eqref{eq:P1-3}-\eqref{eq:P5}.
  \qed
\end{example}

We now explore the structure of barycentric coordinates in terms of the constraint expressions. When the if condition in Step~\ref{algstep:ifcondition} is true, $\mu^k$ and $\theta^k$ are defined as affine functions of $\mu^{k-1}$ and $\theta^{k-1}$. However, when this condition is false, $\mu^k$ is obtained as nonlinear expressions of $\mu^{k-1}$ in Step~\ref{algstep:definemu2}. Therefore, to understand the structure of the barycentric coordinates, we will focus on the calculations in Step~\ref{algstep:definemu2}. This will be made simpler by executing the algorithm in two phases. During the first phase, the algorithm rearranges the constraints so that the if condition in Step~\ref{algstep:ifcondition} is true for the initial steps and then remains false during the rest of the algorithm. We describe the first phase of the algorithm in more detail below.

In Phase 1, we execute Algorithm~\ref{alg:ddwithmult} so that, if any row indexed $[(i+1)\cl{}m]$ is not orthogonal to $L^i$ at iteration $i+1$ then it is reindexed to row $i+1$. This way the first iterations of the algorithm identify the lineality space by executing Steps~\ref{algstep:definexi}-\ref{algstep:defineL1} while later iterations execute Steps~\ref{algstep:defineNsets}-\ref{algstep:defineE2}. Then, Algorithm~\ref{alg:ddwithmult} can be interpreted as initially performing a procedure closely related to Gaussian elimination on $(\Id_{n+1};e_0;\A)$. Let $\rho+1$ be the number of linearly independent columns of $(e_0;\A)$. We can assume without loss of generality that these are indexed $[0\cl{}\rho]$ by reordering variables if necessary. We will retain $x_0$ as the column with index $0$. This is possible since the first column of $(e_0;\A)$, being the only column with a $1$ in the first row, is linearly independent of the other columns. Then, for notational simplicity we write 
\begin{equation}\label{eq:simplifyinversion}
  \begin{split}
  &\begin{pmatrix}
    e_{\rho+1;0} & 0_{1\times (n-\rho)}\\
    \A_{[\cl{}\rho],[0\cl{}\rho]} & \A_{[\cl{}\rho],[\rho+1\cl{}]}\\
    \A_{[\rho+1\cl{}],[0\cl{}\rho]} & \A_{[\rho+1\cl{}],[\rho+1\cl{}]}
  \end{pmatrix} = 
  \begin{pmatrix}
    B & N\\
    \Upsilon & \Psi
  \end{pmatrix}
  ,\text{~}
  \begin{pmatrix}
    x_{[0:\rho]}\\
    x_{[\rho+1:n]}
  \end{pmatrix} = 
  \begin{pmatrix}
    x_B\\
    x_N
  \end{pmatrix}
  ,\\
  &\text{ and }
  \begin{pmatrix}
    B & N\\
    0 & \Id_{n-\rho}
  \end{pmatrix} 
  \begin{pmatrix}
    x_B\\
    x_N
  \end{pmatrix} =
  \begin{pmatrix}
    \sigma\\
    \psi
  \end{pmatrix}, 
\end{split}
\end{equation}
where $B\in \Re^{(\rho+1)\times (\rho+1)}$, $\Upsilon\in \Re^{(m-\rho)\times (\rho+1)}$, $N\in \Re^{(n-\rho)\times (\rho+1)}$, and $\Psi\in \Re^{(m-\rho)\times(\rho+1)}$.
The first $\rho$ iterations of Algorithm~\ref{alg:ddwithmult} can be visualized as performing the following change of variables:
\begin{equation}\label{eq:phase1end}
  \begin{pmatrix}
    \Id_{\rho+1} & 0\\
    0 & \Id_{n-\rho}\\
    B & N\\
    \Upsilon & \Psi
  \end{pmatrix}
  \begin{pmatrix}
  x_B\\
  x_N
  \end{pmatrix}
  = \begin{pmatrix}
  B^{-1} & -B^{-1}N\\
  0 & \Id_{n-\rho}\\
  \Id_{\rho+1} & 0\\
  \Upsilon B^{-1} & 0
\end{pmatrix}
\begin{pmatrix}
  \sigma\\
  \psi
\end{pmatrix},
\end{equation}
where we have used linear dependence of the last $n-\rho$ columns of $\A$ to write $\Upsilon B^{-1}N$ as $\Psi$. The above matrix operation is interpreted as follows. Consider the matrix on the right-hand-side of \eqref{eq:phase1end}. Its first $n+1$ rows are $(R^\rho; L^\rho)$, where $R^{\rho} = (B^{-1}; 0)$ and $L^\rho = (-B^{-1}N; \Id_{n-\rho})$. Moreover, $(\mu^\rho; \theta^\rho) = (\sigma;\psi)$. Then \eqref{eq:DDknonneg} enforces $x_0\ge 0$ and $\A_{[\cl{}\rho],\cl{}}(x_0;x)\ge 0$ because $0\le \mu^\rho = \sigma = Bx_B + Nx_N$. In other words, this phase of the algorithm writes $K^\rho$ in the space of $(\mu^\rho,\theta^\rho)$ as the Minkowski sum of the non-negative orthant and a lineality space; see Lemma~\ref{lem:RpointedWithL}. 
%It follows from \eqref{eq:simplifyinversion} that pointed cone is formed as the non-negative orthant using the slack variables associated with the $\rho$ linearly independent constraints. 
Thus, Phase 1 generalizes the treatment of one constraint in~Remark~\ref{rmk:singleineqRLT}. 
%Secondly, it provides a geometric understanding of Theorem~\ref{thm:RLTsimplex} because this process has the effect of standardizing the simplex, with the exception of variable scaling. For a full-dimensional simplex, this is carried out in a manner similar to Remark~\ref{rmk:axandx} except that, $N^{+}=\{0\}$, so that the multipliers for $j\in N^-$ get scaled and the multiplier $\mu^{\rho+1}_0$ the last inequality defining the simplex. Indeed, the barycentric coordinates in this case are uniquely defined.

Now, we consider the steps after the completion of the first phase, where the lineality space is identified. Observe that Steps~\ref{algstep:defineNsets}-\ref{algstep:defineE2} do not alter or utilize $\theta$ and $L$. The cone is reformulated as $\{(x_0;x)\mid (x_0;x) = R^\rho\sigma + L^\rho\psi, \sigma \ge 0, \Upsilon B^{-1}\sigma \ge 0\}$ and, the algorithm focuses on $\{\sigma\mid \sigma\ge 0, \Upsilon B^{-1}\sigma \ge 0\}$ in the second phase. Therefore, executing Phase 1 before initialization, we can write the cone as $\{(\bar{x}_0;\bar{x})\mid (\bar{x}_0;\bar{x})\ge 0, \tilde{A} (\bar{x}_0,\bar{x})\ge 0\}$, where $(\bar{x}_0;\bar{x}) = \sigma$ and $\tilde{A} = \Upsilon B^{-1}$. Then, the algorithm only runs Phase 2 and computes $\bar{R}$ and $\bar{\mu}(\sigma)$ such that $\sigma = \bar{R}\bar{\mu}(\sigma)$ and  
\begin{equation}\label{eq:changeafterphase2}
  \begin{pmatrix}x_B\\x_N\end{pmatrix} = 
  \begin{pmatrix}B^{-1}\bar{R}\\0\end{pmatrix} \bar{\mu}(Bx_B + Nx_N) + \begin{pmatrix} -B^{-1}N\\\Id_{n-\rho}\end{pmatrix}x_N,
\end{equation}
which gives the barycentric coordinates in the original space.
Moreover, since we focus on polytopes, Phase 1 will compute $L^\rho=[]$ and we will henceforth write $(\bar{x}_0;\bar{x}) = \bar{R}^k\bar{\mu}^k$ when lineality directions can be ignored.

  The reader may have suspected based on expressions in Proposition~\ref{prop:twoineq} and Examples~\ref{ex:nontrivialbarycentric} and \ref{ex:megaexample} that the barycentric coordinates are related to the product of constraint expressions. We next argue that this intuition indeed holds. To see this, we first interpret $\beta\mu^{k-1}(x_0;x)$, computed at iteration $k$, in terms of the original constraint expression $\A_{k,\cl{}}(x_0;x)\ge 0$. This relationship is easy to see when the polytope is full-dimensional, since $N^+\ne \emptyset$ at any of the iterations. Otherwise, all the rays following this iteration would satisfy $\A_{k,\cl{}}(x_0;x) = 0$ violating the assumed full-dimensionality. Then, the two expressions are equivalent as follows easily from Theorem~\ref{thm:expressmulin}. Specifically,
  \begin{equation}\label{eq:lifttobeta}
    \begin{split}
    \beta^\intercal\mu^{k-1} = \A_{k\cl{}} R^{k-1} \mu^{k-1} = \A_{k\cl{}} \Id_{n} (P^1_\pi)^\intercal D^1\cdots (P^{k-1}_\pi)^\intercal D^{k-1} \mu^{k-1}\\ = \A_{k\cl{}}(x_0;x),
    \end{split}
  \end{equation}
  where $D^1,\ldots,D^{k-1}$ are as defined in Theorem~\ref{thm:expressmulin}. However, if the polytope is not full-dimensional, the two expressions are not the same, even though they evaluate to the same value for every $(x_0;x)\in P$. Nevertheless, we argue that each barycentric coordinate takes the following form.

  \begin{definition}\label{defn:nonnegratio}
    Let a constraint expression refer to any linear expression whose non-negativity is one of the constraints defining $K$. We say a function is a constraint product ratio if it can be expressed as a ratio of two functions, each a non-negative combination of products of constraint expressions. \qed 
  \end{definition}

  The class of constraint product ratios is closed under positive combinations and products and, therefore, under the pruning operation in Remark~\ref{rmk:removeredundant}.

  \begin{proposition}\label{prop:mustructure}
    Each barycentric coordinate in Algorithm~\ref{alg:ddwithmult} is a constraint product ratio. If the number of executions of Step~\ref{algstep:definemu2} are $i$, then the degree of the numerator (resp. denominator) of the barycentric coordinate is bounded from above by $\frac{3^i}{2}+\frac{1}{2}$ (resp. $\frac{3^i}{2}-\frac{1}{2}$).
  \end{proposition}
  \long\def\propmustructureproof{
    At the end of iteration $i$, we augment Algorithm~\ref{alg:ddwithmult} to rewrite $\A_{j,\cl{}} (x_0;x)$, for all $j\ge i$, as $(\beta^{i,j})^\intercal \mu^{i}(x_0;x)$ so that both these expressions have the same evaluation at each point $(x_0;x)\in K^i$. This is done by calculating $(\beta^{i,j})^\intercal$ as $(\beta^{i-1,j})^\intercal (P^i_\pi)^\intercal D^i$. We first analyze the change of variables performed at the end of Phase 1. Assume that we have an expression that is expressed as a constraint product ratio of $(\bar{x}_0,\bar{x})$ after Phase 2. This ratio uses expressions
    $\tilde{A}_{j,\cl{}}(\bar{x}_0;\bar{x})$ (resp. $\bar{x}_j$), for some $j$, which, using $\Upsilon B^{-1}N = \Psi$ and \eqref{eq:simplifyinversion}, become $\Upsilon_{j,\cl{}} B^{-1} (Bx_B + Nx_N) = \Upsilon_{j,\cl{}}x_B + \Psi_{j,\cl{}} x_N$ (resp. $\bar{x}_i = Bx_B+Nx_N$). In other words, each is a constraint product ratio in the original variable space after the change of variables. Therefore, we can limit our attention to Phase 2.
    
    We will suppress the function arguments for $\mu$ for notational convenience.
    We now show that at the end of iteration $k$, $\mu^k$ is a constraint product ratio. Moreover, there exists $\phi^{i,\ell}(\bar{x}_0;\bar{x})$, a constraint product ratio, such that $(\beta^{i,\ell})^\intercal\mu^i - \phi^{i,\ell}(\bar{x}_0;\bar{x})$ simplifies to $\tilde{A}_{\ell,\cl{}}(\bar{x}_0;\bar{x})$. Observe that this implies, after simplification, that $(\beta^{i,\ell})^\intercal \mu^i =  \tilde{A}_{\ell,\cl{}}(\bar{x}_0;\bar{x}) + \phi^{i,\ell}(\bar{x}_0;\bar{x})$, which shows that $(\beta^{i,\ell})^\intercal \mu^i$ is also a non-negative combination of inequalities satisfied by $K$. Since the first two equalities in \eqref{eq:lifttobeta} show that $\beta^{j,j}$ is the same as $\beta$ computed at iteration $j$, we obtain that $\beta^\intercal \mu^j$ is a constraint product ratio. We consider the base case at iteration $0$. Since $R^0 = \Id$ and $\mu^0 = (\bar{x}_0,\bar{x})$, it follows that $\beta^{i,j}(\bar{x}_0;\bar{x}) = \tilde{A}_{j,\cl{}}(\bar{x}_0,\bar{x})$.
    %Therefore, it suffices to focus on Phase 2 to show that barycentric coordinates are constraint product ratios. so that $\mu^0 = \sigma$, and we write the inequalities as $\Upsilon B^{-1}\sigma \ge 0$, which, by \eqref{eq:simplifyinversion} is $\Upsilon x_B\ge 0$. Therefore, $\phi^{1,\ell}(\bar{x}_0;\bar{x}) = 0$ and $(\beta^{0,j})^\intercal \mu^0$ is the $j^{\text{th}}$ inequality describing $K$. After the initialization, we change the variables and assume that $(\bar{x}_0;\bar{x}) = \sigma$ so that the inequalities defining $K$ are $(\bar{x}_0;\bar{x})\ge 0$ and $\A(\bar{x}_0;\bar{x})\ge 0$. This is admissible because, if a quantity is expressed as a constraint product ratio after variable changes, the original constraints are obtained when $\sigma$ is substituted by $B^{-1}x_B$ in each constraint expression.
  
    Assume now that the result is true at iteration $i$, and we prove that it holds for iteration $i+1$. For $j\in N^0$, $\mu^{i+1}_j = \mu^i_j$, and so, the result follows by induction. For $j\in N^+$, we have $\mu^{i+1}_j = \mu^i_j (\beta^\intercal \mu^{i})/(\beta^\intercal_+ \mu^i)$, where $\beta^\intercal_+$ is the coordinate-wise maximum of $0$ and $\beta$. Then, by induction, $\beta^\intercal \mu^i$, $\mu^i_j$, and $\beta^\intercal_+ \mu^i$ are constraint product ratios and the result follows. Now, consider $(j,j')$ where $j\in N^+$ and $j'\in N^-$. Then, the result holds for $\mu^{i+1}_{j,j'} = \mu^{i}_j\mu^{i}_{j'} /(\beta^\intercal_+ \mu^i)$ because $\mu^i_j$, $\mu^i_{j'}$, and $\beta^\intercal_+ \mu^i$ are constraint product ratios. 
    
    We show that $(\beta^{i+1,\ell})^\intercal \mu^{i+1}$ is of the specified form. If $N^+ \ne \emptyset$, we write $(\beta^{i,\ell})^\intercal \mu^i = (\beta^{i,\ell})^\intercal (P^{i+1}_\pi)^\intercal D^{i+1}\mu^{i+1} = \beta^{i+1,\ell}\mu^{i+1}$. In other words, since each $\mu^i$ can be written as an affine combination of $\mu^{i+1}$, it follows that $(\beta^{i+1,\ell})^\intercal \mu^{i+1}$ and  $(\beta^{i,\ell})^\intercal \mu^i$ share the same expression. Therefore, $\phi^{i+1,\ell}(\bar{x}_0;\bar{x}) = \phi^{i,\ell}(\bar{x}_0;\bar{x})$. Now, assume that $N^+ = \emptyset$. Then, 
    \begin{equation*}(\beta^{i,\ell})^\intercal \mu^i = \sum_{j\in N^0} (\beta^{i,\ell})^\intercal \mu^i_j + \sum_{j\in N^-}(\beta^{i,\ell})^\intercal \mu^i_j = (\beta^{i+1,\ell})^\intercal \mu^{i+1} + \sum_{j\in N^-}(\beta^{i,\ell})^\intercal \mu^i_j,
    \end{equation*} 
    where the second equality is because $\mu^{i+1}$ retains the entries $(\mu^i_j)_{j\in N^0}$ unaltered and drops the entries $(\mu^i_j)_{j\in N^-}$. We then define 
    \begin{equation*}\phi^{i+1,\ell}(\bar{x}_0;\bar{x}) = \phi^{i,\ell}(\bar{x}_0;\bar{x}) - \sum_{j\in N^-}(\beta^{i,\ell})^\intercal \mu^i_j(\bar{x}_0;\bar{x})
    \end{equation*}
    and show that $\phi^{i+1,\ell}(\bar{x}_0;\bar{x})$ is a constraint product ratio. If $\beta^{i,\ell}_j < 0$, the $j^{\text{th}}$ term in the summation is retained without change. If $\beta^{i,\ell}_j > 0$, we instead express $-\mu^{i}_j(\bar{x}_0;\bar{x})$ as a constraint product ratio which allows us to replace $\beta^{i,\ell}_j$ with $-\beta^{i,\ell}$, which is negative. After this change, it follows from the inductive hypothesis that $\phi^{i+1,\ell}(\bar{x}_0;\bar{x})$ is a constraint product ratio being the sum of constraint product ratios. 
  
    It remains to show that, for $j\in N^-$, $-\mu^i_j$ can be expressed as a non-negative combination of inequalities defining $K$. Clearly, $\mu^{i}_j = 0$ for every point in $K^{i+1}$ because $\A_{i+1,\cl{}}$ is satisfied at equality by every ray in $R^{i+1}$ and, so, every point $(\bar{x}_0;\bar{x}) \in K^{i+1}$ satisfies $\tilde{A}_{i+1,\cl{}}(\bar{x}_0;\bar{x}) = 0$. Let $T$ be the affine hull of $K^{i+1}$. By the induction hypothesis, the numerator of $\mu^{i}_j$ is a non-negative combination of product of constraint expressions. Since each of these terms is non-negative over $K^{i+1}$, they are all zero on $T$. This implies that, for each term, at least one of the constraint expressions is zero over $T$. To see this, consider any linear expression that is not zero over $T$. Then, it is zero over a linear subspace of $T$ of strictly lower dimension. In other words, we can focus on the product of the remaining expressions, which must be zero over $T$ since this product is continuous over $T$. Eliminating linear expressions one-by-one, we identify a linear expression in the product that is zero over $T$. Since $T$ is expressible as an intersection of linear equalities implied in $\tilde{A}_{1:i+1,\cl{}}(\bar{x}_0;\bar{x})\ge 0$, the identified expression is a linear combination of the expressions defining these linear equalities.  It follows that the negative of this linear combination is also a linear combination and can be used to express $-\mu^i_j(\bar{x}_0;\bar{x})$ as a non-negative combination of inequalities defining $K$.

    Finally, we argue that the degree of the polynomials in the numerator and denominator of $\mu^k$ are bounded from above by $\frac{3^i}{2}+\frac{1}{2}$ and $\frac{3^i}{2}-\frac{1}{2}$ respectively, where $i$ is the number of iterations where the if condition in Step~\ref{algstep:ifcondition} is false. At $i=0$, the numerator has a degree at most $1$ and denominator has a degree $0$. At Step~\ref{algstep:definemu1} of Algorithm~\ref{alg:ddwithmult}, we transform $\mu$ linearly. Therefore, the degree of the polynomials does not change. Since there are $n$ lineality directions when the algorithm starts and this cardinality reduces by $1$ each time the if condition at Step~\ref{algstep:ifcondition} is true, there are exactly $n$ such iterations. The remaining $m-n$ iterations execute Steps~\ref{algstep:defineNsets}-\ref{algstep:defineE2}. We are interested in Step~\ref{algstep:definemu2} during these iterations. Since the numerators of $\mu^{k-1}$ have degree less than $\frac{3^n}{2}+\frac{1}{2}$ and the denominator of $N_{\text{tot}}$ has degree at most $\frac{3^i}{2}-\frac{1}{2}$, the numerator of $\mu^k$ has degree at most $3^i + 1 + \frac{3^i}{2} -\frac{1}{2} = \frac{3^{i+1}}{2} + \frac{1}{2}$. Similarly, the denominator of $\mu^{k-1}$ has degree no more than $\frac{3^i}{2}-\frac{1}{2}$ and the numerator of $N^{\text{tot}}$ has degree no more than $\frac{3^i}{2}+\frac{1}{2}$. Therefore, the degree of the denominator is no more than $3^i - 1 + \frac{3^i}{2} + \frac{1}{2} = \frac{3^{i+1}}{2}-\frac{1}{2}$.
  }
  \iftoggle{allproofsinpaper}{%
  \begin{proof}
    \propmustructureproof\qed
  \end{proof}}{}

  The above result has important consequences for certificates of optimality for concavo-convex programs. We elaborate this in the context of DBP.

  \begin{theorem}\label{thm:representationresult}
    If $x^\intercal Q y\ge \delta$ over $\bigl\{(x,y)\mid P\times P^y\bigr\}$, where $P$ and $P^y$ are non-empty polytopes, then
    \begin{equation}\label{eq:bilrepresentation}
    z(x)(x^\intercal Q y -\delta) =  q(x,y),
    \end{equation}
    where $z(x)$ is a non-negative combination of products of constraints in $P$. Moreover, $q(x,y)$ is a non-negative combination of products of constraints in $P$ where each such product may additionally be multiplied with a constraint in $P^y$. Moreover, both $z(x)$ and $q(x,y)$ have a finite degree and $z(x) > 0$ for  $x\in \ri(P)$. These polynomials are derived in closed form from the dual solution to \eqref{eq:DBPhull} and Algorithm~\ref{alg:ddwithmult}. 
  \end{theorem}
  \begin{proof}
  We consider minimizing $x^\intercal Q y$ over $P\times P^y$. By Proposition~\ref{prop:mustructure}, the polynomials used to express $\mu^m$ as a constraint product ratio have a degree at most $\frac{3^{m-n}}{2}+\frac{1}{2}$. Then, the expressions for the convex multipliers $\lambda$ in \eqref{eq:DBPhull} associated with the vertices of $P$ can be found by scaling $\mu^m(x_0;x)$ using the entries in $R^m_{1\cl{}}$ (see Definition~\ref{defn:dehomogenize}). The feasible region of \eqref{eq:DBPhull} is bounded. This is because $Y^k_{\cl{}i} \in \conv\bigl((0,0)\cup \{(1,y)\mid y\in P^y\}\bigr)$. Moreover, the feasible region is non-empty since $P\times P^y$ is non-empty and each such point can be lifted to become feasible to \eqref{eq:DBPhull}. Then, by strong duality of linear programming, it follows that there is a non-negative matrix $S\in \Re^{m_y\times p}_+$ and $\gamma \in \Re^p_+$ such that:
  \begin{equation}\label{eq:expressTrace}
    \begin{split}
  \Trace(QYV^\intercal) &= \Trace\left(S,\begin{pmatrix}b^y  &- A^y\end{pmatrix}\begin{pmatrix}\lambda^\intercal\\ Y\end{pmatrix}\right) + \delta \lambda^\intercal e + \gamma^\intercal \lambda\\
  &=\Trace\left(S,\begin{pmatrix}b^y  &- A^y\end{pmatrix}\begin{pmatrix}\lambda^\intercal\\ Y\end{pmatrix}\right) + \delta + \gamma^\intercal \lambda,
    \end{split}
  \end{equation}
  where the second equality is because the barycentric expressions satisfy $\lambda^\intercal e = 1$. Then, by Proposition~\ref{prop:mustructure} each $\lambda$ is a constraint product ratio whose denominator is bounded by a degree $\frac{3^{m-n}}{2}-\frac{1}{2}$ polynomial. We refer to this polynomial as $z(x)$. By writing $Y$ as $y\lambda^\intercal$, we multiply \eqref{eq:expressTrace} by $z(x)$ to derive the desired expression in the statement of the theorem. The positivity of $z(x)$ follows from the positivity of $\mu$ over $\ri(P)$ shown in Proposition~\ref{prop:positivemu}.
  \qed
  \end{proof}

  Observe that the representation~\eqref{eq:bilrepresentation} shows that $x^\intercal Q y \ge \delta$ at $(x,y)\in \ri(P)\times P^y$ because $z(x) > 0$ over $\ri(P)$.
  Then, since $x^\intercal Q y$ is continuous, it is not smaller than $\delta$ over $P\times P^y$. Similar proofs can be constructed for VXP except that the dual certificates will need to extracted using Rockafellar-Fenchel duality theorem from $\DDP_\V$ \cite{Rockafella1970}. For the case of a simple polytope, the calculation of the associated polynomials is easier. This is because there is an explicit formula~\cite{warren2007barycentric} which expresses the numerator of a barycentric coordinate as a non-negative scalar multiple of the product of those constraint expressions that are strictly positive at the corresponding vertex. The denominator is then the sum of the expressions in the numerator. We illustrate the procedure in the proof of Theorem~\ref{thm:representationresult} using the following example.
  \begin{example}\label{ex:DBPproof}
    Consider the following disjoint bilinear program:
    \begin{equation}\label{eq:DBPex}
    \begin{alignedat}{3}
      &\min &\;& 
      \rlap{$\displaystyle {}-27x_1y_1 - 108x_1y_2 + 90x_2y_1 - 32x_2y_2$}\\
      &&&+ 180x_1 - 180x_2 - 180y_1 + 204y_2\\
      &\text{s.t.}&& x\in P := 
       \left\{
          \begin{aligned}
          &x_1-x_2\ge -2\\
          &-3x_1+2x_2\ge -6\\
          &-3x_1-4x_2\ge -15\\
          &x_1\ge 0
          \end{aligned}
       \right\}, &\;&
       y \in P^y := 
       \left\{
          \begin{aligned}
          &y_1-y_2\ge -2\\
          &-3y_1+2y_2\ge -6\\
          &-3y_1-4y_2\ge -15\\
          &y_1,y_2\ge 0
          \end{aligned}
       \right\}.
    \end{alignedat}
  \end{equation}
    There are several optimal solutions for \eqref{eq:DBPex}. One of these optimal solutions is:
    \begin{equation*} (x_1,x_2,y_1,y_2) = (0,2,0,0)
    \end{equation*}
    which has an optimal value of $-360$. 
    When solved with Gurobi 12.0.1 \cite{gurobi}, the solver explores 84170 nodes \cite{gurobi} to prove optimality within $0.0086\%$ tolerance. The optimal solution for the first-level RLT relaxation is:
    \begin{equation*}
    \begin{split}
    (x_1,x_2,y_1,y_2,xy_{11},xy_{12},xy_{21},xy_{22}) 
    = \frac{1}{23}(12,30,18,12,12,36,-36,18),
    \end{split}
    \end{equation*}
    with an objective value of $-524\frac{8}{23}$ exhibiting significant relaxation gap. Instead, as shown in Lemma~\ref{lemma:disjVrep}, \eqref{eq:DBPhull} achieves the optimal value. Now, we extract an algebraic proof from its solution using Theorem~\ref{thm:representationresult}.
    To do so, we use Algorithm~\ref{alg:ddwithmult} to compute:
    \begin{equation}\label{eq:lambdaYex}
    \begin{bmatrix}
    \lambda^\intercal\\
    Y
    \end{bmatrix} = 
    \frac{1}{150 - 40x_2 + 33x_1}
    \begin{bmatrix}
    1\\
    y_1\\
    y_2\\
    \end{bmatrix}
    \begin{bmatrix}
      (-15 + 3x_1 + 4x_2)(-6 + 3x_1 - 2x_2)\\
      -6(2 + x_1 - x_2)(-15 + 3x_1 + 4x_2)\\
      -7x_1(-6 + 3x_1 - 2x_2)\\
      18x_1(2 + x_1 - x_2)
    \end{bmatrix}^\intercal.
    \end{equation}
    Then, we retrieve the optimal dual solution for \eqref{eq:DBPhull}. The constraints with non-zero dual values are:
    \begin{alignat*}{3}
      &\text{Constraint in \eqref{eq:DBPhull}} &\quad&& \text{Dual Value}\\
      &6\lambda_1 - 3 Y_{12} + 2 Y_{22} \ge 0 &&& 150\\
      &15\lambda_4 - 3Y_{14} - 4Y_{24}\ge 0 &&& 42\\
      &Y_{13}\ge 0 &&& 63\\
      &Y_{21}\ge 0 &&&140\\
      &\lambda_1 + \lambda_2 + \lambda_3 + \lambda_4 = 1 &&& 360
    \end{alignat*}
    Since the last constraint is implicit in the definition of barycentric coordinates, it is not needed for the proof. We multiply each constraint with its dual value (except for the last one) and use the definitions of $\lambda$ and $Y$ from \eqref{eq:lambdaYex} to obtain
    \begin{equation*}
    \frac{(150-40x_2+33x_1)(\text{objective of \eqref{eq:DBPex}} + 360)}{(150-40x_2+33x_1)} \ge 0.
    \end{equation*}
    Moreover, $150-40x_2+33x_1 > 0$ over $P$. This gives the algebraic identity \eqref{eq:bilrepresentation} claimed in Theorem~\ref{thm:representationresult} proving that the minimum value in \eqref{eq:DBPex} is $-360$.
    \qed
  \end{example}
  
  Algorithm~\ref*{alg:ddwithmult} derives several such denominator expressions. Although this follows from Theorem~\ref{thm:representationresult}, we detail next how the relaxations for polynomial functions~(see \cite{lasserre2002semidefinite,sherali2013reformulation}) can be used to develop inequalities relating monomials divided by such expressions. 

  \begin{remark}\label{rmk:nonlinineq}
    Let $\Omega \subseteq \Z^n_+$ and $d_\ell(x)$ be expressible as an affine expression in $(x^\alpha)_{\alpha\in \Omega}$. Define $\tilde{x} = (x^\alpha)_{\alpha\in\Omega}$, 
    \begin{align*}
      &\mathcal{F} = 
      \left\{
        \left(\begin{pmatrix}1 & y^\intercal\\\tilde{x} &\tilde{x}y^\intercal\end{pmatrix},
        d_{\ell}(x)\begin{pmatrix}1 & y^\intercal\\ \tilde{x} & \tilde{x}y^\intercal\end{pmatrix}\right)\;\middle|\; (x,y)\in P\times C\right\}, \text{ and}\\
      &\mathcal{G} =
      \left\{
        \left(\dfrac{1}{d_\ell(x)}\begin{pmatrix}1 & y^\intercal\\\tilde{x} &\tilde{x}y^\intercal\end{pmatrix},
        \begin{pmatrix}1 & y^\intercal\\ \tilde{x} & \tilde{x}y^\intercal\end{pmatrix}\right)\;\middle|\; (x,y)\in P\times C\right\},
    \end{align*}
    so that $d_\ell(x)>0$ over $P$ and $P\ne\emptyset$. For any set $S$, let $0\closure\conv(\mathcal{S})$ be the recession cone of $\closure\conv(\mathcal{S})$. Then, it has been shown that \citep[Proposition 1]{he2023convexification}
    \begin{align*}
      &\closure\conv(\mathcal{F}) = 
      \left\{
        \left(W,W^{d_\ell}\right)\;\middle|\; 
          \left(W,W^{d_\ell}\right)\in \rho\closure\conv(\mathcal{G}), W_{1,1}=1,
          \rho\ge 0
      \right\} \\    
      &\closure\conv(\mathcal{G}) = 
      \left\{
        \left(\overline{W}^{d_\ell},\overline{W}\right)\;\middle|\; 
          \left(\overline{W}^{d_\ell},\overline{W}\right)\in \sigma\closure\conv(\mathcal{F}), \overline{W}_{1,1}=1,
          \sigma\ge 0
      \right\}.
    \end{align*}
    It follows that inequalities for $\mathcal{F}$ are in correspondence with the inequalities for $\mathcal{G}$. We may thus utilize polynomial inequalities to find valid inequalities for $\mathcal{G}$. For example, we may require that $1/d_{\ell}(x) \left(1;(x^\alpha)_{\alpha\in \Omega}\right)\left(1,(x^\alpha)_{\alpha\in \Omega}^\intercal\right)$ is a positive-semidefinite matrix. The inequalities for $\conv(\mathcal{G})$ are useful in our context. This is because we leverage Proposition~\ref{prop:mustructure} to express barycentric products as a sum of constraint product ratios where each ratio consists of a product of constraint expressions as its numerator. Then, we expand the product in the numerator as a linear combination of $x^\alpha/d_{\ell}(x)$, for some $\alpha$ as in the definition of $\conv(\mathcal{G})$. Moreover, we multiply these barycentric coordinates with $y$ in \eqref{eq:DDkyscale}, which after expansion result in the terms $\tilde{x}y^\intercal/d_{\ell}(x)$ in $\mathcal{G}$. \qed
  \end{remark}

  \paragraph{Hierarchy:} Our discussions have shown various ways to tighten the hierarchies for $(\D)$ and $(\AC)$ that were presented in Theorems~\ref{thm:simplehierarchy} and Theorem~\ref{thm:simpleAChierarchy} by using the algebraic structure of the barycentric coordinates. We summarize these ideas without attempting to be comprehensive. We say $\Xi_k$ is a set of constraint orders, where each constraint order $\varsigma\in \Xi$ is a sequence of $k$ indices in $\{1\cl{}m\}$ without repetitions. To keep track of the constraint orders, the iteration superscripts will be replaced by $\varsigma$ so that \eqref{eq:pointrep} will be written as: $(x_0;x) = R^\varsigma \mu^\varsigma + L^\varsigma \theta^\varsigma$, $\mu^\varsigma\ge 0$. We will denote subsequences of $\varsigma$ consisting of $i^{\text{th}}$ to $j^{\text{th}}$ processed constraint as $\varsigma[i\cl{}j]$. The denominator associated with a barycentric coordinate $\mu^\varsigma_j$ will be denoted as $d_{\varsigma,j}(x_0;x)$. At level $k$, let $\eta$ be the degree of the largest degree monomial in the numerators for $\mu^\varsigma$. Here, we assume that $L^\varsigma = []$. As mentioned earlier, this can be achieved at the initialization step by performing Gaussian elimination and redefining variables. Then, we construct the following relaxation:
  \begingroup
  \allowdisplaybreaks
  \begin{subequations}\label{eq:tighthierarchy}
  \begin{alignat}{4}
    &(\DE^k)&\quad&\min&\;& z\\
    &&&\text{s.t.}&&z \ge \sum_{i=1}^{p_\varsigma} \ftilde\left(R^k_{[2\cl{}],i}\mu^\varsigma_i(1;x),y\mu^\varsigma_i(1;x),\mu^\varsigma_i(1;x)\right)&\quad&\varsigma\in \Xi\label{eq:objDEk} \\
    &&&&&y\mu^\varsigma_i(1;x)\in \mu^\varsigma_i(1;x) C &\quad& \begin{aligned}&i\in \{\cl{}p_\varsigma\},\\&\varsigma\in \Xi\end{aligned}\label{eq:DEkyscale}\\
    &&&&&\mu^\varsigma_i(1;x)\ge 0 &&\begin{aligned}&i\in \{\cl{}p_\varsigma\},\\&\varsigma\in \Xi\end{aligned}\label{eq:DEknonneg}\\
%    &&&&&R^\varsigma \mu^\varsigma(1;x) = (1;x)&&\varsigma\in\Xi\label{eq:DEklinprecision}\\
    &&&&&\mu^{\varsigma[1\cl{}t-1]} = (P^\varsigma_\pi)^\intercal \begin{pmatrix}F^{\varsigma[1\cl{}t]} & G^{\varsigma[1\cl{}t]}\\ 0 & D^{\varsigma[1\cl{}t]}\end{pmatrix}\mu^{\varsigma[1\cl{}t]} &&\begin{aligned}&t\le k,\\&\varsigma\in \Xi\end{aligned}\label{eq:DEkmurecurse}\\
    &&&&&\mu^{\varsigma[1\cl{}t-1]}y^\intercal = (P^\varsigma_\pi)^\intercal \begin{pmatrix}F^{\varsigma[1\cl{}t]} & G^{\varsigma[1\cl{}t]}\\ 0 & D^{\varsigma[1\cl{}t]}\end{pmatrix}\mu^{\varsigma[1\cl{}t]}y^\intercal &&\begin{aligned}&1\le t\le k,\\ &\varsigma\in \Xi\end{aligned}\label{eq:DEkmuy}\\
    &&&&&\dfrac{\prod_{j\in \Theta} \A_{j,\cl{}}(1;x)}{d_{\varsigma[1:t],j}(1;x)} \ge 0 &&\begin{aligned}&|\Theta|\le \eta,\\&1\le t\le k\\&\varsigma\in \Xi.\end{aligned}\label{eq:DEkprodcons}
  \end{alignat}
\end{subequations}
\endgroup
We expand expressions and linearize $x^\alpha/d_{\varsigma,j}(1;x)$ and $yx^\alpha/d_{\varsigma,j}(1;x)$ ensuring that the denominators which evaluate to the same expression are linearized using the same variables. We use Remark~\ref{rmk:nonlinineq} to introduce inequalities relating $x^\alpha/d_{\varsigma,j}(1;x)$ and $yx^\alpha/d_{\varsigma,j}(1;x)$ for various $\alpha$ with one another. There are further opportunities to tighten the relaxation. First, we can disaggregate \eqref{eq:DEkyscale} replacing $\mu^\varsigma(1;x)$ with the left-hand-side of \eqref{eq:DEkprodcons} using Proposition~\ref{prop:mustructure}. 
%Second, we can pick a $\varsigma\in \Xi$ to enforce that $\mu^\varsigma = \mu^{\varsigma'}$ for each permutation $\varsigma'$ of $\varsigma$. This constraint is redundant if $\mu^\varsigma$ and $\mu^{\varsigma'}$ have the same expressions since it is implied by linearization. This constraint is valid because the optimal solution is attained at a vertex of $P$ at which these conditions are satisfied. We choose a specific $\varsigma$ since we cannot simultaneously require that that the solution belongs to the vertices of two different polytopes. 
%Second, the optimality of vertex solution for $\varsigma$ also implies that we can require that the $\mu^{\varsigma'}$ corresponding to any redundant rays are zero, where $\varsigma'$ is a permutation of the chosen $\varsigma$. 
Second, as in Remark~\ref{rmk:nonlinineq}, we may require that $1/d_{\varsigma,j}(1;(x^\alpha)_{\alpha\in \mathcal{M}})(1,(x^\alpha)^\intercal_{\alpha \in \mathcal{M}})$ is a positive semidefinite matrix, where $\mathcal{M}$ is a set of monomial powers. We can also multiply these matrix constraints with those defining $P$ and/or $C$ \cite{lasserre2002semidefinite}. Third, if the objective $f(x,y)$ is written as a sum of terms, we can disaggregate \eqref{eq:objDEk} for each term in the summation. For example, this is the case for $(\AC)$, where we replace the objective with $\sum_{j=1}^n z_j$ and \eqref{eq:objDEk} with
\begin{equation*}
  z_j \ge \sum_{i=1}^{p_k} g_j\left(\dfrac{y\mu^\varsigma_i(1;x)}{\mu^\varsigma_i(1;x)}\right)R^\varsigma_{j,i} \mu^\varsigma_i(1;x).
\end{equation*}
It is useful to consider $(\DBP)$, where $g_j(y) = Q_{j,\cl{}} y$. We introduce $z_j \ge \sum_{i=1}^{p_k} Q_{j,\cl{}} y\mu^\varsigma_i(1;x)R^\varsigma_{j,i} = Q_{j,\cl{}}y \left(\mu^\varsigma(1;x)\right)^\intercal R^\varsigma_{j,\cl{}} = Q_{j,\cl{}}y (1;x)^\intercal$, where the last equality follows from \eqref{eq:DEkmuy}. Even though we required that the lineality space was empty for Theorem~\ref{thm:simpleAChierarchy}, this is no longer required here, and a meaningful relaxation can still be constructed if there are sufficient constraint orders processed by Algorithm~\ref{alg:ddwithmult} as shown in Remark~\ref{rmk:singleineqRLT}. Then, we simply replace $R^\varsigma \mu(1;x) = (1;x)$ with $R^\varsigma \mu(1;x) + L^\varsigma \theta(1;x)$ instead to obtain 
\begin{equation*}
  z_j \ge Q_{j,\cl{}}y \left(\left(\mu^\varsigma(1;x)\right)^\intercal R^\varsigma_{j,\cl{}} + \left(\theta^\varsigma(1;x)\right)^\intercal L^\varsigma_{j,\cl{}}\right).
\end{equation*}

\begin{theorem}\label{thm:tighthierarchy}
The proposed hierarchy~\eqref{eq:tighthierarchy} is at least as tight as \eqref{eq:DD}.
\end{theorem}
\begin{proof}
  The result follows since \eqref{eq:tighthierarchy} includes the relationships in \eqref{eq:DD}, except \eqref{eq:DDklinprecision}. However, this equality also follows because
  \begin{equation*}
    \begin{split}
      (1;x) = 
  \mu^{\varsigma[1\cl{}0]} =(P^{\varsigma[1:1]}_\pi)^\intercal \begin{pmatrix}F^{\varsigma[1:1]} & G^{\varsigma[1:1]}\\ 0 & D^{\varsigma[1:1]}\end{pmatrix}\mu^{\varsigma}\cdots (P^\varsigma_\pi)^\intercal \begin{pmatrix}F^{\varsigma} & G^{\varsigma}\\ 0 & D^{\varsigma}\end{pmatrix}\mu^{\varsigma} = R^\varsigma \mu^\varsigma,
    \end{split}
  \end{equation*}
  where the first equality is by initialization of Algorithm~\ref{alg:ddwithmult}, the second equality is by \eqref{eq:DEkmuy}, and the last equality is because of Theorem~\ref{thm:expressmulin} and $L^\varsigma = []$. \qed
\end{proof}

For $\AC$, by multiplying \eqref{eq:DEkprodcons} with $z\ge \sum_{i=1}^{p_k} g\left(\dfrac{y\mu^k_i(1;x)}{\mu^k_i(1;x)}\right)R^k_{\cl{},i} \mu^k_i(1;x)$ and  $y\in C$, we recover the RLT inequalities. These arise naturally in our hierarchy when $d_{\varsigma[1:t],j}(1;x) = 1$, as is the case during Phase 1 of Algorithm~\ref{alg:ddwithmult}.

We mention that the expressions for $\mu^\varsigma$ and $\mu^{\varsigma'}$, where $\varsigma'$ is a permutation of $\varsigma$, often turn out to be the same, especially after pruning as in Remark~\ref{rmk:removeredundant}. This is because rational expressions for barycentric coordinates are unique, when the $V$-description is minimal, as long as the associated polynomials are of sufficiently low degree~\cite{warren2003uniqueness}. These are also the same in another case of interest as discussed in the following result and Remark~\ref{rmk:barycentricproduct}.

  \begin{proposition}\label{prop:barycentricforsums}
    Consider polytopes $P_1$, $P_2$, and $P$ in $\Re^n$ such that $P=P_1+P_2$. For $x\in P_1$, let $x=\sum_{i=1}^{p_1}\gamma_i(x) v^{i,1}$ and, for $x\in P_2$, let $x=\sum_{i=1}^{p_2}\varphi_i(x) v^{i,2}$, where $\gamma_i$ (resp. $\varphi_i$) are barycentric coordinates for $P_1$ (resp. $P_2$) and $(v^{i,1})_{i=1}^{p_1}$ (resp. $(v^{i,2})_{i=1}^{p_2}$) are vertices of $P_{1}$ (resp. $P_{2}$). Then, for $x_1\in P_1$ and $x_2\in P_2$: 
    \begin{equation}\label{eq:barycentricsum}
      x_1 + x_2 =\sum_{i=1}^{p_1}\sum_{j=1}^{p_2}\gamma_i\varphi_{j} \left(v^{i,1}+v^{j,2}\right).
  \end{equation}
  \end{proposition}
  \long\def\propbarycentricforsumsproof{
    The result follows since 
    \begin{equation*}
      \begin{split}
      \sum_{i=1}^{p_1}\sum_{j=1}^{p_2}\gamma_i(x_1)\varphi_{j}(x_2) \left(v^{i,1}+v^{j,2}\right) = \sum_{i=1}^{p_1} \gamma_i(x_1) v^{i,1} \sum_{j=1}^{p_2}\varphi_j(x_2) \\
      + \sum_{i=1}^{p_1} \gamma_i(x_1) \sum_{j=1}^{p_2}\varphi_j(x_2) v^{j,2} = x_1+x_2,
      \end{split}
    \end{equation*}
    where the last equality is because $\sum_{j=1}^{p_1}\gamma_j(x_1) = 1$ and $\sum_{j=1}^{p_2}\varphi_j(x_2) = 1$.
  }
  \iftoggle{allproofsinpaper}{%
  \begin{proof}
    \propbarycentricforsumsproof\qed
  \end{proof}}{}

  \begin{remark}\label{rmk:barycentricproduct}
  A special case of Proposition~\ref{prop:barycentricforsums} is when $P=P_1\times P_2$, where $P_i\in \Re^{n_i}$. The barycentric coordinates for $P_1\times P_2$ are then $\gamma(x_1)\varphi(x_2)$. This implies that Algorithm~\ref{alg:ddwithmult} can be run in parallel for $P_1$ and $P_2$. 
  %Moreover, the form of barycentric coordinates does not depend on the order in which constraints in $P$ are processed, as long as those in $P_1$ and those in $P_2$ are processed in a given order. 
  Since barycentric coordinates for $[0,1]$ are $\left(x;1-x\right)$, we obtain the barycentric coordinates for $[0,1]^n$ as in Proposition~\ref{prop:DDhypercube} which are used as product factors in $(\ACRLT)$. \qed
  \end{remark}

  It follows easily that the convex hull of barycentric coordinates $\mu(1;x)$ over the corresponding polytope $P$ is a simplex. 
  
  \begin{proposition}\label{prop:convexhullmu}
    Let $(v_i)_{i=1}^p$ be vertices of a polytope $P\in \Re^n$. Assume that $\left(\mu_i(x)\right)_{i=1}^p$ are the barycentric coordinates associated with the vertices of $P$. Let $X=\left\{\left(\mu_i(x)\right)_{i=1}^p\;\middle|\; x\in P\right\}$. Then, 
    \begin{equation}\label{eq:convhullmu}
      \conv(X) = \left\{\mu\;\middle|\; \sum_{i=1}^p \mu_i = 1, \mu_i\ge 0, i=1,\ldots,p\right\}.
    \end{equation}
  \end{proposition}
  \long\def\propconvexhullproof{
    First, observe that $e_{p;i}\in X$ for $i=1,\ldots,p$. This is because $e_{p;i} = \left(\mu_i(v_j)\right)_{j=1}^p$. Let $S$ denote the simplex on the right-hand-side of \eqref{eq:convhullmu}. Then, it follows that $\conv(X) \supseteq S$. To see the reverse inclusion, observe that $\mu(x)$ satisfy all the inequalities defining $S$.
  }
  \iftoggle{allproofsinpaper}{%
  \begin{proof}
    \propconvexhullproof\qed
  \end{proof}}{}

  \begin{remark}\label{rmk:overmu}
  Consider the following variant of $(\D)$, where the feasible region is: 
  \begin{equation}\label{eq:muxfeasible}
    \left\{(\mu,x,y)\in \Re^p\times\Re^n\;\middle|\; y\in C, (\mu,x)\in X'\right\},
  \end{equation}
  with $X'=\{(\mu,x)\mid \mu = \mu(1;x), x\in P\}$. We minimize $f(\mu,x,y)$, which is concave in $(\mu,x)$ for a fixed $y\in C$ and convex in $y$ for a fixed $(\mu,x)\in \conv(X')$. 
  Then, by Proposition~\ref{prop:convexhullmu} and because an optimal solution for $(\D)$ is such that $(\mu,x)$ is extremal, we may relax $X'$ with its convex hull in \eqref{eq:muxfeasible}. Since $x$ is affinely related to $\mu$ via $(1;x) = R\mu$ and because affine transformations commute with convexification, it follows from Proposition~\ref{prop:convexhullmu} that $\conv(X') = \{(\mu,x)\mid \mu\in \conv(X), (1;x) = R\mu\}$. In other words, our hierarchy allows construction of convex envelope of $f(\mu,x,y)$ over $X'\times C$. Specifically, this shows that our hierarchy also applies to the case where $P=P_1\times\cdots\times P_{n_p}$, $x_i\in P_i\subseteq \Re^{n_i}$, and the objective depends on multilinear expressions of $(x_i)_{i=1}^{n_p}$ since we can minimize $f(\mu,y)$ and $\prod_{i=1}^{n_p} x_{i\ell_i}$ for $\ell_i\in \{\cl{}n_i\}$ is expressible as an affine transform of $\mu$ by Remark~\ref{rmk:barycentricproduct} and Theorem~\ref{thm:expressmulin}. \qed
  \end{remark}

  We have focused our attention on lifting $P$ to the nonlinear set $X'$ in Remark~\ref{rmk:overmu}. However, the same construction also applies to $P^y$ in $(\DBP)$ so that our Algorithm~\ref{alg:ddwithmult} can be used to lift $P^y$ as well, thereby identifying additional valid inequalities for $(\DBP)$. Since, by Remark~\ref{rmk:barycentricproduct}, barycentric coordinates for $P\times P^y$ are products of those for $P$ and $P^y$, $(\DBP)$ optimizes a linear function of the barycentric coordinates for $P\times P^y$. Even if $C$ is not polyhedral, tighter relaxations for $(\D)$ and $(\AC)$ can be obtained by outer-approximating $C$ with a polytope $P^y$ and using our lifting procedure on $P^y$ besides $P$. 

  \section{Facial Disjunctive Programs}\label{sec:fdp}

  In this section, we develop a hierarchy that terminates finitely with the convex hull of $\FD$. The hierarchy is first constructed using disjunctive programming arguments and then additional algebraic relations are developed using barycentric coordinates. 
  %For now, we relax the requirement that $F_{ij}$ and $F_{ij'}$ contain distinct vertices for $j\ne j'$. 
  %and $\mathcal{P}(\{\cl{}j_i\})$ denote the power set of $\{\cl{}j_i\}$. 
  %\begin{equation*}
  %  \phi_{S,s}(i) =
  %    \begin{cases}
  %    \{s_i\}, & i\in S,\\
  %    \{\cl{}j_i\}, & i\notin S.
  %    \end{cases}
  %\end{equation*}
  Let $\gamma_{ij}$ be the indicator of $F_{ij}$. For any $S \subseteq \{1,\ldots,n_p\}$ and any selection vector 
  $s=(s_i)_{i\in S}$ with $s_i\in \{\cl{}j_i\}$, define
  \begin{equation*}
  F^{S,s} = \left\{x\;\middle|\; 
  \begin{aligned}
  x_i\in F_{is_i} &\text{ for } i \in S\\
  x_i\in \bigcup_{j=1}^{j_i} F_{ij} &\text{ for } i \notin S\\
  \end{aligned}\right\}.
  \end{equation*}
  For example, %since $\phi_{\emptyset,()}$ maps $i$ to $\{\cl{}j_i\}$,  
  $F^{\emptyset,()} = \{x\mid x_i \in \bigcup_{j=1}^{j_i} F_{ij}\}$.

  \begin{assumption}\label{assume:vertexPartition}
  $F_{ij}$ and $F_{ij'}$ do not share vertices when $j\ne j'$.
  \end{assumption}

  \begin{theorem}\label{thm:FDPhierarchy}
    Let $J^S$ denote $\prod_{i\in S}\{\cl{}j_i\}$. Let $(\tau^{ij}, \pi^{ij})\in \Re^{n_i+1}$ be such that $\tau^{ij} - (\pi^{ij})^\intercal x_i \ge 0$ is valid for $P_i$ but $F_{ij} = \{x_i\mid x_i \in P_i, \tau^{ij} - (\pi^{ij})^\intercal x_i \le 0\}$. The set $\FDR^k$ below defines the level $k$ nonlinear relaxation of $\FD$:
    \begin{subequations}\label{eq:FDPhierarchy}
      \begin{alignat}{3}
        &\gamma^{S,s} \left(\begin{pmatrix}-c & B\end{pmatrix}(1;x) + Dy\right) \le_{\bar{K}} 0, &\quad&\forall S \subseteq \{\cl{}n_p\},\; |S|=k,\;s\in J^S  \label{eq:fdpconeconstraint}\\
        &\sum_{s \in J^S} \gamma^{S,s} (1;x;y) = (1;x;y)&&\forall\, S \subseteq \{\cl{}n_p\},\; |S|=k\label{eq:fdpsumconstraint}\\
        &\gamma^{S,s}x_{i'}\in \gamma^{S,s}P_{i'} &&
        \forall S \subseteq \{\cl{}n_p\},\; |S|=k,\;s\in J^S,\;i'\in[\cl{}n_p] \label{eq:fdpscaleP}\\
        &\gamma^{S,s}\bigl(\tau^{is_i} - (\pi^{is_i})^\intercal x_i\bigr)\le 0&&\forall i\in S \subseteq \{\cl{}n_p\},\; |S|=k,\;s\in J^S \label{eq:fdpfacedefinition}\\
        &\gamma^{S,s}\ge 0 &&\forall S \subseteq \{\cl{}n_p\},\; |S|=k,\;s\in J^S,
      \end{alignat}
    \end{subequations}
  With Assumption~\ref{assume:vertexPartition}, $\gamma^{S,s}$ can be restricted to be the indicator of $F^{S,s}$. Let the linear relaxation obtained by substituting $\gamma^{S,s}x$ with $u^{S,s}$ and $\gamma^{S,s}y$ with $w^{S,s}$ be denoted as $\underline{\FDR}^k$. Then, $\underline{\FDR}^1\supseteq \cdots \supseteq \underline{\FDR}^{n_p} = \conv(\FD)$. 
  \end{theorem}
  \begin{proof}
  We show that \eqref{eq:FDPhierarchy} is a valid relaxation when $\gamma^{S,s}$ is defined as $\prod_{i\in S}\gamma_{is_i}$ so that $\gamma^{S,s}$ is the indicator of $F^{S,s}$. For each $x_i$ that belongs to multiple faces, $F_{ij}$, $j\in J_i$, the indicator for the smallest indexed face in $J_i$ as $1$ while the rest are set to $0$. With Assumption~\ref{assume:vertexPartition}, this qualification is not needed since each $x_i$ belongs to a unique face. If not, the smallest face containing $x_i$ is a subset of each $F_{ij}$, $j\in J_i$, contradicting that faces do not share vertices.
  
  First, \eqref{eq:fdpsumconstraint} follows since $x\in F^{\emptyset,()}$ implies by Assumption~\ref{assume:vertexPartition} that for each $S$, there is exactly one $s\in J^S$ such that $\gamma^{S,s} = 1$. Second, \eqref{eq:fdpfacedefinition} follows because $\gamma^{S,s}=1$ implies $x_i\in F^{S,s}$ and $\tau^{is_i} - (\pi^{is_i})^\intercal x_i\le 0$. The rest of the inequalities follow since $\gamma^{S,s}$ is non-negative and multiplied with one of the inequalities valid for $P_1\times\cdots\times P_{n_p}$. Now, we show that the $\underline{\FDR}^k\supseteq \underline{\FDR}^{k+1}$ for $k < n_p$. To see this, let $S\subseteq \{\cl{}n_p\}$ such that $|S| = k$. Choose $\ell\not\in S$. For a given $s = (s_i)_{i\in S}$ and $\psi \in \{\cl{}j_\ell\}$, define $\bar{s}(\psi)$ as follows:
  \begin{equation*}
  \bar{s}(\psi) = \begin{cases}
     s_i & i \in S\\
     \psi & i = \ell.
    \end{cases}
  \end{equation*}
  Define 
  \begin{equation}\label{eq:fdpsumoverell}
    (\gamma^{S,s},u^{S,s}, w^{S,s}) = \sum_{\psi \in \{\cl{}j_\ell\}} \left(\gamma^{S\cup\{\ell\},\bar{s}(\psi)},u^{S\cup\{\ell\},\bar{s}(\psi)},w^{S\cup\{\ell\},\bar{s}(\psi)}\right).
  \end{equation} 
  The proposed solution satisfies each constraint in $\underline{\FDR}^k$ since a constraint with $S\subseteq\{\cl{}n_p\}$ and $s\in J^S$ is an aggregation of the constraints in $\underline{\FDR}^{k+1}$ for different $\psi$ in $\{\cl{}j_\ell\}$, with subset $S\cup\{\ell\}$ and selection vector $\bar{s}(\psi) \in J^{S\cup\ell}$. Finally, we show that $\underline{\FDR}^{n_p} = \conv(\FD)$. Consider a feasible solution $(\gamma, u, w)$ to $\underline{\FDR}^{n_p}$. By \eqref{eq:fdpscaleP}, \eqref{eq:fdpfacedefinition}, and that $(\tau^{ij}, \pi^{ij})$ defines $F_{ij}$, for any $s\in J^{\{\cl{}n_p\}}$, it follows that $u^{\{\cl{}n_p\},s}\in \gamma^{\{\cl{}n_p\},s}\prod_{i=1}^{n_p} F_{is_i}$. Since each $P_i$ is bounded, so is $\prod_{i=1}^{n_p} F_{is_i}$. When $\gamma^{\{\cl{}n_p\},s} = 0$, we have $u^{\{\cl{}n_p\},s} = 0$, since $\prod_{i=1}^{n_p} F_{is_i}$ has no recession directions. If $\gamma^{\{\cl{}n_p\},s} > 0$, using \eqref{eq:fdpconeconstraint}, we have that $(u^{\{\cl{}n_p\},s},w^{\{\cl{}n_p\},s})\in \gamma^{\{\cl{}n_p\},s} \FD$. Otherwise, $Dw^{\{\cl{}n_p\},s} \in \bar{K}$. In other words, $(0, w^{\{\cl{}n_p\},s})$ belongs to the recession cone of $\FD\cap F^{\{\cl{}n_p\},s}$. Since all disjunctive sets $\FD\cap F^{\{\cl{}n_p\},s}$ for $s\in J^{\{\cl{}n_p\}}$ share the same recession directions, it follows by disjunctive programming that $\underline{\FDR}^{n_p} = \conv(\FD)$ \cite[Corollary 9.8.1]{Rockafella1970}. \qed
  \end{proof}

  Recall that FDP generalizes 0-1 programs where $j_i=2$, $F_{i1} = \{x_i\mid x_i\in P_i, x_i=0\}$, and $F_{j2} = \{x_i\mid x_i\in P_i, x_i = 1\}$. Moreover, $(\tau_{i1}, \pi^{i1}) = (0,-1)$ and $(\tau_{i2},\pi^{i2}) = (1,1)$ and Assumption~\ref{assume:vertexPartition} is satisfied. Then, \eqref{eq:fdpfacedefinition} reduces to:
  \begin{equation*}
  x^{S,S'}x_i \le 0\; \forall i\in S' \text{ and } x^{S,S'}(1-x_i)\le 0\; \forall i\in S,
  \end{equation*}
  which imply each of these quantities is $0$ since $x^{S,S'}x_i\ge 0$ and $x^{S,S'}(1-x_i)\ge 0$ follow from \eqref{eq:fdpscaleP}. Specifically, \eqref{eq:fdpfacedefinition} and \eqref{eq:fdpscaleP} imply the equalities $x_i^2 = x_i$ used in RLT relaxations. Nevertheless, $\underline{\FDR}^k$ relaxation obtained by linearizing \eqref{eq:FDPhierarchy} is not as tight as the RLT relaxation for 0-1 programs. To see this consider Equation \eqref{eq:fdpsumoverell}. Utilizing that $(u^{T,t}, w^{T,t})$ linearizes $\gamma^{T,t}(x^{T,t},y^{T,t})$, \eqref{eq:fdpsumoverell} follows by restricting $\gamma^{T,t}$, for all $(T,t)$, to represent the indicator of $F^{T,t}$. 
  Since the left-hand-side of \eqref{eq:fdpsumoverell} is independent of $\ell$, the right-hand-side is invariant with $\ell$ as well.  For 0-1 programs, setting $S=\{1\}$ and $\ell=i$, the following provides an example of \eqref{eq:fdpsumoverell}: 
  \begin{equation*}
  (x_1,x_1x,x_1y) = (x_ix_1, x_ix_1x, x_ix_1y) + \bigl((1-x_i)x_1,(1-x_i)x_1x, (1-x_i)x_1y\bigr).
  \end{equation*}
  Algebraic calculation shows that the right-hand-side of the above expression is independent of $i$, and this relation is implicit in RLT but not in \eqref{eq:FDPhierarchy}. The discrepancy arises in part because, unlike RLT, \eqref{eq:FDPhierarchy} does not express $\gamma^{S,s}$ as functions of $x$. We address this issue next. 

  Let $P_i$ have $p_i$ vertices, $V_{ir_i}$ be the $r_i^{\text{th}}$ vertex of $P_i$, and $\lambda_{ir_i}(x)$ be barycentric coordinate associated with $V_{ir_i}$. For each face $F_{ij}$, let $E_i(j)$ represent the index set of vertices of $P_i$ that are contained in $F_{ij}$. Therefore, by Assumption~\ref{assume:vertexPartition}, for each $r$ there is at most one $j$ such that $r\in E_i(j)$.

  Now, we define the barycentric indicator function associated with a face as the sum of barycentric coordinates of the vertices contained in the face.
  
  \begin{definition}\label{defn:barycentricIndicator}
     For each $i\in \{\cl{}n_p\}$ and $j\in \{\cl{}j_i\}$, we define \begin{equation*}
     \eta_{ij}(x_i) = \sum_{r_i\in E_i(j)} \lambda_{ir_i}(x_i)
     \end{equation*} to be the \emph{barycentric indicator function} of $F_{ij}$. For $S\subseteq\{\cl{}n_p\}$ and $s = (s_i)_{i\in S}$, we define 
     \begin{equation*}
     \eta^{S,s}(x) = \prod_{i\in S} \eta_{is_i}(x_i)
     \end{equation*} to be the \emph{barycentric indicator} associated with $F^{S,s}$.
  \end{definition}
    
We motivate the above definition next. By Assumption~\ref{assume:vertexPartition}, $\bigl(E_i(j)\bigr)_{j\in \{\cl{}j_i\}}$ are disjoint and $F_{ij} = \conv\bigl(\bigcup_{r_i\in E_i(j)} V_{ir_i}\bigr)$. It follows easily that $\eta_{ij}(x_i)$ is $1$ when $x_i\in F_{ij}$ and $0$ if $x_i\in F_{ij'}$ for any $j'\ne j$. (We recall that the barycentric coordinate associated with points outside the affine hull of $P_i$, \textit{i.e.} points in $N^-$ when $N^+=\emptyset$ at Step~\ref{algstep:definemu2} in Algorithm~\ref{alg:ddwithmult}, are forced to be zero.) In other words, $\eta_{ij}(x_i)$ acts as an indicator function of $F_{ij}$ for feasible $x$ even though it is strictly positive for $x$ in the relative interior of $P_i$. By Remark~\ref{rmk:barycentricproduct}, the barycentric coordinates of $P=P_1\times\cdots\times P_{n_p}$ are given by $\prod_{r_i\in\{\cl{}p_i\}}\lambda_{ir_i}$. Then, the barycentric $\eta^{S,s}$ is similarly the sum of the barycentric coordinates associated with all the vertices of $P$ in $F^{S,s}$. It can be written in terms of barycentric coordinates as:
  \begin{equation}\label{eq:rewriteeta}
     \eta^{S,s}(x) = \prod_{i\in S} \sum_{r_i\in E_i(s_i)} \lambda_{ir_i}(x_i)= \sum_{r\in \prod_{i\in S} E_i(s_i)}\prod_{i\in S}\lambda_{ir_i}(x).
  \end{equation} 
  \begin{proposition}\label{prop:fdpnonlin}
  Let $\eta_{S,s}(x)$ be the barycentric indicator associated with face $F^{S,s}$ of $P$. If Assumption~\ref{assume:vertexPartition} holds, substituting $\eta^{S,s}(x)$ for $\gamma^{S,s}$ in \eqref{eq:FDPhierarchy} gives a nonlinear relaxation of $\FD$. Given $r\in \prod_{i\in S}\{\cl{}p_i\}$, let 
  \begin{equation*}
    \chi_S(i,r,x) = \begin{cases}
    V_{ir_i} & i\in S\\
    x_i & \text{otherwise.}
    \end{cases}
  \end{equation*}
  Then, $\forall S \subseteq \{\cl{}n_p\}$, $|S|=k$, $s\in J^S := \prod_{i\in S}\{\cl{}j_i\}$, we may write
  \begin{equation}\label{eq:uformulation}
    u^{S,s} = 
  \sum_{r\in \prod_{i\in S} E_i(s_i)}\left(\prod_{i\in S}\lambda_{ir_i}(x_i) 
    \begin{bmatrix}
      1\\
      \chi_S(1,r,x)\\
      \vdots\\
      \chi_S(n_p,r,x)
    \end{bmatrix}\right).
  \end{equation}
  In particular, for a given $r\in \prod_{i\in S} \{\cl{}p_i\}$, 
  \begin{equation}\label{eq:simplifyprod}
    \eta^{S,s}(x)\prod_{i\in S}\lambda_{ir_i}(x_i) = 
    \begin{cases}\prod_{i\in S}\lambda_{ir_i}(x_i) & \text{ if } r_i\in E_i(s_i) \text{ for all } i\in \{\cl{}n_p\}\\
      0 & \text{otherwise.}
    \end{cases}
  \end{equation} 
  \end{proposition}
  \begin{proof}
  If $x\in F^{S,s}$, it follows $\sum_{r_i\in E_i(s_i)}\lambda_{ir_i}(x_i) = 1$ for all $i\in S$. Therefore, by Equation~\eqref{eq:rewriteeta}, $\eta^{S,s}(x)= 1$. If $x\in \FD\backslash F^{S,s}$, then there exists an $i\in S$ such that $x_i\in F_{ij}$ with $j\ne s_i$. Therefore, $\sum_{r_i\in E_i(s_i)} \lambda_{ir_i(x_i)} = 0$ since, by Assumption~\ref{assume:vertexPartition}, $F_{ij}$ and $F_{is_i}$ do not share any vertices. It follows that for $x\in \FD$, $\eta^{S,s}(x)$ is an indicator of $F^{S,s}$. Therefore, by Theorem~\ref{thm:FDPhierarchy}, it follows that $\gamma^{S,s}$ may be replaced by $\eta^{S,s}(x)$ in \eqref{eq:FDPhierarchy}. 
  
  Now, we show \eqref{eq:simplifyprod}. For $i\in S$, if $r_i\in \{\cl{}p_i\}\backslash E_i(s_i)$ we have $\eta^{S,s}(x)\lambda_{ir_i}(x) = 0$. This follows for $x_i\in F_{is_i}$ by realizing that $\lambda_{ir_i}(x_i)=0$ and for $x_i\not\in F_{is_i}$ by utilizing $\eta^{S,s}(x) = 0$. Now, assume that $r_i\in E_i(s_i)$ for all $i\in \{\cl{}n_p\}$. If $x\in F^{S,s}$, $\eta^{S,s}(x)=1$ and \eqref{eq:simplifyprod} follows. If $x\in \FD\backslash F^{S,s}$, then there exists an $i$ such that $x_i\in F_{ij}$ for $j\ne s_i$. By Assumption~\ref{assume:vertexPartition}, $r_i\not\in E_i(j)$ since $r_i\in E_i(s_i)$. Therefore, $\lambda_{ir_i} = 0$ and \eqref{eq:simplifyprod} is satisfied.

  Finally, \eqref{eq:uformulation} follows by substituting $\eta^{S,s}(x)$ for $\gamma^{S,s}$ in $u^{S,s} = \gamma^{S,s}x$, expanding $x$ using the barycentric coordinates as follows:
  \begin{equation*}
  x = \sum_{r\in \prod_{i\in S} \{\cl{}p_i\}}\left(\prod_{i\in S}\lambda_{ir_i}(x_i) 
    \begin{bmatrix}
      1\\
      \chi_S(1,r,x)\\
      \vdots\\
      \chi_S(n_p,r,x)
    \end{bmatrix}\right),
  \end{equation*}
  and using \eqref{eq:simplifyprod} for simplification. \qed
  \end{proof}
  
  In the following, we compute the barycentric indicators for 0-1 programs and find them to be the RLT product-factors. We then specialize Proposition~\ref{prop:fdpnonlin} and relate it to how products are simplified in RLT relaxations.
  \begin{remark}\label{rmk:fdpto0-1}
  For 0-1 programs, we write $[0,1]^{n_p}$ as $P=P_1\times\cdots\times P_{n_p}$ where $P_i = \{x_i\mid 0\le x_i\le 1\}$. Then, $p_i = 2$, $V_{i1} = \{0\}$, $V_{i2} = \{1\}$, $\lambda_{i1}(x_i) = 1-x_i$ and $\lambda_{i2}(x_i) = x_i$. Consider the product factor $x^{S',S''}$ such that $S'\cap S''=\emptyset$. Define $S=S'\cup S''$ and $s = (s_i)_{i\in S}$, where $s_i = \{2\}$ if $i\in S'$ and $s_i=\{1\}$ if $i\in S''$. Using \eqref{eq:rewriteeta}, we have $\eta^{S,s}(x) = x^{S',S''}$. Therefore, $\eta^{S,s}(x)$ are the product-factors used in RLT for 0-1 programs. In Theorem~\ref{thm:FDPhierarchy}, we defined $u^{S,s}$ as the linearization of $\eta^{S,s}(x)x$. Using $\eta^{S,s}(x) = x^{S',S''}$, $\chi_S(i,r,S)$ in Proposition~\ref{prop:fdpnonlin} is $1$ if $i\in S'$, $0$ if $i\in S''$ and $x_i$ otherwise. Therefore, as in \eqref{eq:uformulation}, we have $\eta^{S,s}(x)x_i = \eta^{S,s}(x)$ if $i\in S'$, $0$ if $i\in S''$, and $x^{S'\cup\{i\},S''}$ otherwise. These substitutions align with RLT where $x^{S',S''}x_i$ is substituted with $x^{S',S''}$ if $i\in S'$, $0$ if $i\in S'$ and $x^{S'\cup\{i\},S''}$ otherwise. Finally, \eqref{eq:simplifyprod} reduces to $x^{S',S''}x^{\bar{S},S\backslash \bar{S}} = x^{\bar{S},S\backslash \bar{S}}$ if $\bar{S} = S'$ and $0$ otherwise. \qed
  \end{remark}
  
  Remark~\ref{rmk:fdpto0-1} highlights the use of nonlinear expressions for $\gamma^{S,s}$ based on barycentric coordinates. To see this, consider $\gamma^{\{1\},2}x_2 = x_1x_2 = x_2x_1 = \gamma^{\{2\},2}x_1$, an equivalence that is otherwise not apparent. The barycentric coordinates again provide a unifying framework for different relaxation hierarchies.

\section{Conclusions}

In this paper, we develop new relaxation hierarchies for concavo-convex programs, which include affine-convex programs and concave minimization problems as special cases. These relaxations improve upon both disjunctive programming and reformulation-linearization-based hierarchies, and are derived using a constructive approach based on the double description method. We provide an algebraic characterization of barycentric coordinates for polyhedral cones as rational functions of constraint products, enabling the use of inequalities from polynomial and fractional optimization literature to derive tighter relaxations. In summary, the paper develops a common theoretical framework to analyze different relaxation hierarchies for these problem classes.

For prominent cases where RLT relaxations terminate finitely, we recover the corresponding hierarchy using barycentric coordinates. This work demystifies the role of RLT product-factors for constructing relaxations by formally relating them to barycentric coordinates. The techniques introduced in the paper provide a geometric-algebraic bridge connecting DP and RLT for continuous optimization problems. Moreover, Algorithm~\ref{alg:ddwithmult} provides a symbolic iterative procedure for computing barycentric coordinates, addressing a long-standing challenge in computational geometry.

While we leave the computational exploration of the hierarchies for different problem classes to future work, we discuss several challenges and opportunities briefly. Although the hierarchy terminates finitely, the first challenge is that the degree of the resulting expressions can grow rapidly (see Proposition~\ref{prop:mustructure}). A successful implementation will require careful selection of inequalities to execute Algorithm~\ref{alg:ddwithmult} while balancing the associated tradeoffs. Acceleration techniques for double description may provide useful insights in this regard~\cite{Fukuda1995,padberg2013linear}. Second, we showed that by processing just a few constraints, inequalities that are not seen in RLT relaxations can be derived. These ideas can be extended to analyze simple structured sets.
%as in Proposition~\ref{prop:twoineq}. 
As discussed in Corollary~\ref{cor:allorders} and Remark~\ref{rmk:singleineqRLT} when constraints are processed in different orders simultaneously, the algebraic relationships between the barycentric coordinates provide opportunities for strengthening the relaxations just as RLT strengthens lift-and-project relaxations. However, the computational burden increases as well since the number of constraint orders grows exponentially with the number of constraints. The technique detailed in Theorem~\ref{thm:representationresult} can be used after projecting polytopes $P$ and $P^y$ to small dimensions to derive algebraic proofs of validity. Finally, related work \cite{oh2022algorithms,oh2024branchandbound} has leveraged insights from this paper, and specifically the double-description procedure to design a simplicial branch-and-bound algorithm for disjoint bilinear programs that terminates finitely.

\section*{Acknowledgements}
This work was supported by Air Force Office of Scientific Research via grant \#FA9550-22-1-0069. The author would like to thank the anonymous reviewers for their constructive feedback.

\iftoggle{createarxiv}{
  \bibliographystyle{plainnat}
}{\bibliographystyle{spmpsci}}
\iftoggle{createarxiv}{
\bibliography{finitehierarchy_R1}
}{\bibliography{DD}}

\iftoggle{createappendix}{
\footnotesize
\section*{Appendix}
}{}

\iftoggle{allproofsinpaper}{}{%
\subsection{Proof of Lemma~\ref{lemma:disjVrep}}
\disjstandardsimplex
}

\iftoggle{ancilliarydiscussions}{%
\subsection{Brief Description of DD algorithm}\label{app:ddalgorithm}
\ddalgorithm
}{}

\iftoggle{allproofsinpaper}{}{%
\subsection{Proof of Proposition~\ref{prop:DDhypercube}}
\propDDhypercubeproof
}

\iftoggle{allproofsinpaper}{}{%
\subsection{Proof of Proposition~\ref{prop:rationalfunction}}
\proprationalfunctionproof
}

\iftoggle{allproofsinpaper}{}{%
\subsection{Barycentric coordinates in Example~\ref{ex:nontrivialbarycentric}}\label{app:showremainingmu}
\showallmuexpr
}

\iftoggle{allproofsinpaper}{}{
\subsection{Proof of Proposition~\ref{prop:correctmutheta}}
\propcorrectmuthetaproof
}

\iftoggle{allproofsinpaper}{}{%
\subsection{Proof of Lemma~\ref{lem:RpointedWithL}}
\lemRpointedWithLproof
}

\iftoggle{allproofsinpaper}{}{%
\subsection{Proof of Proposition~\ref{prop:twoineq}}
\proptwoineqproof
}

\iftoggle{allproofsinpaper}{}{%
\subsection{Proof of Proposition~\ref{prop:mustructure}}
\propmustructureproof
}

\iftoggle{allproofsinpaper}{}{
\subsection{Proof of Proposition~\ref{prop:barycentricforsums}}
\propbarycentricforsumsproof
}

\iftoggle{allproofsinpaper}{}{
\subsection{Proof of Proposition~\ref{prop:convexhullmu}}
\propconvexhullproof
}

\end{document}